\documentclass[reqno,psamsfonts,11pt]{gen-j-l}
\usepackage{amssymb,amsmath,amsbsy,amsthm,esint,setspace,graphicx}
\usepackage{hyperref}
\usepackage{amsrefs}
\usepackage{enumitem}

\newtheorem{theorem}{Theorem}[section]
\newtheorem{lemma}[theorem]{Lemma}
\newtheorem{proposition}[theorem]{Proposition}
\newtheorem{corollary}[theorem]{Corollary}
\theoremstyle{definition}
\newtheorem{definition}{Definition}[section]

\theoremstyle{remark}
\newtheorem*{remark}{Remark}

\let\temp\phi
\let\phi\varphi
\let\varphi\temp

\begin{document}

\title[The propagation of chaos for hard spheres]{The propagation of chaos for a rarefied gas of hard spheres in the whole space}

\author{Ryan Denlinger}

\begin{abstract}
We discuss old and new results on the mathematical justification of
Boltzmann's equation. The classical result along these lines is a theorem
which was proven by Lanford in the 1970s. This paper is naturally
divided into three parts. \\ 
\emph{I. Classical.}
 We give new proofs of both the uniform bounds required for
Lanford's theorem, as well as the related bounds due to Illner \&
Pulvirenti for a perturbation of vacuum. 
 The proofs use a duality argument
and differential inequalities, instead of a fixed point iteration.\\
\emph{II. Strong chaos.} We introduce a new notion of propagation of
chaos. Our notion of chaos provides for uniform error estimates
on a very precise set of points; this set is closely related to the
notion of strong (one-sided) chaos and the emergence of irreversibility.
\\
\emph{III. Supplemental.}  We
announce and provide a  proof (in Appendix A) of
 propagation of \emph{partial} factorization at some
phase-points where complete factorization is impossible. 
\keywords{Particle models \and Hard spheres \and Lanford's theorem}
\end{abstract}

\maketitle

\section{Introduction}
\label{intro}

We are interested in the system of $N$ identical
 elastic hard spheres of diameter
$\varepsilon > 0$, which move through $d$-dimensional Euclidean space
according to the laws of Newtonian mechanics. This is an important model
in mathematical physics because the rules are relatively simple
and yet they capture in a realistic way the macroscopic
behavior of many physical systems. Usually the number 
of particles is quite large, say $N=10^{23}$, so it seems hopeless
to follow the microscopic dynamics directly.
An alternative strategy, pioneered by Maxwell and Boltzmann, is to
assign probabilities to the possible microscopic configurations of the
system
and study the evolution of these probabilities subject to mechanistic
laws (e.g., conservation of mass, momentum and energy). Given a suitable
choice of spatial and temporal scales, the equation one
formally arrives at through this line of  reasoning is known as Boltzmann's
equation.

Half a century after Boltzmann's work, H. Grad used 
precise physical reasoning in an attempt to give Boltzmann's equation
a firm physical footing. He devised a special scaling limit, known
today as the \emph{Boltzmann-Grad limit}, in which the microscopic
dynamics heuristically reduce to the Boltzmann equation under a
``molecular chaos'' assumption (the mathematical nature of the chaos assumption would
not be clarified fully until the 1970s). \cite{Gr1949} 
 However, this did not resolve the question of deriving Boltzmann's equation
because there was no \emph{mathematical} argument linking the microscopic
Liouville equation to the Boltzmann equation. C. Cercignani gave a
non-rigorous but fairly precise description of the convergence
process. \cite{doi:10.1080/00411457208232538} O.E. Lanford provided the
first rigorous convergence proof for Boltzmann,
by describing the reduced dynamics arising from low-order
correlations, and showing that the high-order correlations
have negligible influence on the behavior of the gas, at least
for a short time. \cite{L1975} More recently, a careful 
\emph{quantitative} analysis of Lanford's theorem has been provided
by I. Gallagher, L. Saint-Raymond and B. Texier. \cite{GSRT2014}

We remark on several related developments. The major limitation in
Lanford's theorem is the short time of validity, which so far has not been
lifted except
in very restrictive perturbative regimes. R. Illner and M. Pulvirenti
were able to overcome the time restriction and prove global convergence
for a highly rarefied gas near vacuum, using inequalities related to the 
dispersive nature of the system. \cite{I1989,IP1986,IP1989}
Different perturbative regimes can be obtained in bounded domains, most
notably a periodic box (equipped with a Gibbs measure which is invariant
under the dynamics). Perturbing only in the initial distribution of a single particle
leads naturally to the (non-conservative) linear Boltzmann equation; perturbing
all the particles in a symmetric way leads to the linearized Boltzmann
equation. Both possibilities have been studied in the literature, most
notably by H. van Beijeren, O. E. Lanford, J. Lebowitz and
H. Spohn, and in a separate contribution by J. Lebowitz and H. Spohn.
\cite{vBLLS1980,LS1982} These perturbative settings have been studied
more recently by T. Bodineau, I. Gallagher and L. Saint-Raymond, who
proved quantitative error estimates on \emph{diverging} timescales
$T_N \approx \left( \log \log N \right)^r $ for some known $r>0$,
leading to hydrodynamic limits (namely Brownian motion, and the
Stokes-Fourier equations). \cite{BGSR2015,BGSR2015III} The perturbation of
equilibrium is an extremely difficult problem (particularly on large time
intervals) and we will not have more to say about it in this work.

There are several other important results which are not directly
related to Lanford's theorem but are nevertheless foundational in
kinetic theory.

\begin{itemize}
\item 
\emph{Stochastic models.} All models we have mentioned so far have been fully
deterministic; this means that randomness is allowed in the choice of initial
data, but the \emph{evolution} for each initial state is fully
determined. However, there is an important class of models in kinetic
theory where the dynamics itself introduces randomness. We specifically
mention the Kac model; in this model, the position coordinates are
treated as hidden variables, and in particular the impact parameter for
each collision is a random variable with some specified law.
When the number of particles tends to infinity, the evolution is seen
to converge to the (nonlinear) space-homogeneous Boltzmann equation with the
appropriate collision kernel. These models were first analyzed in a couple
of influential papers by M. Kac and H. McKean. \cite{Ka1956,MK1967} 
There have been many papers dealing with similar models in the intervening
years, and a very complete treatment has been given by S. Mischler
and C. Mouhot. \cite{MM2013}

\item
\emph{Lorentz gases.} We refer to a class of models first
studied by G. Gallavotti. \cite{Ga1969}
 In these Lorentz gas-type models, the dynamics is indeed deterministic, but 
they differ from  the case of Lanford in that all the particles but one are 
considered \emph{stationary obstacles}, distributed like Poisson scatterers.
The dynamics is much simpler in this case because the background particles 
never move out of place; in the Boltzmann-Grad limit one recovers the linear 
Boltzmann equation for the evolution of the tagged particle. Note that it is
not possible to enforce momentum conservation in a Lorentz gas, so these
models are only physically realistic if the tagged particle is
much lighter than the background particles.

\item
\emph{Vlasov-type mean field limit.} Physical limits in which each particle
feels the influence of the entire gas are generally called mean-field
limits; these models can be fully deterministic, or they can possess
some stochasticity. Mean field limits tend to have a relatively pleasant
mathematical structure because a typical particle's trajectory is
governed by the \emph{average} of the other particles' trajectories; this 
property is very helpful in controlling the correlations generated by the 
dynamics. Whereas the Boltzmann-Grad scaling leads to Boltzmann's kinetic 
equation, the Vlasov-type mean-field models lead to Vlasov-type equations
 in the
limit $N\rightarrow \infty$. The study of Vlasov-type mean field limits 
is a vast field in its own right and we provide only a small sampling of the
relevant literature. 
\cite{D1979,MK1967II,jabin:hal-00609453}
\end{itemize}
Henceforth in this work we will not be concerned with
stochastic models, Lorentz gases, or mean field limits.

\textbf{The goals of the present work are twofold.} First, we shall present a
new proof of the uniform bounds which are central to Lanford's theorem.
We use differential inequalities and a duality argument, instead
of a fixed point argument, to control the growth of correlations in
the BBGKY hierarchy. We will apply this method to prove both the
short-time result of Lanford, as well as the global near-vacuum result
of Illner \& Pulvirenti. \cite{L1975,IP1986,IP1989}
Our second goal is to thoroughly address the issue of uniform convergence
of the marginals in the limit $N\rightarrow \infty$. The motivation
is the notion of \emph{strong (one-sided) chaos} and the appearance of
irreversibility from an underlying reversible dynamics.\footnote{Note that
the rigorous link between irreversibility and strong chaos in general requires
application of the
 Hewitt-Savage theorem;
 see subsection \ref{ssec:nonchaos} below
for some discussion of the connection.}
The issue of irreversibility is
tied to convergence properties along very singular sets in
phase space; for this reason, uniform convergence (on a sufficiently large
set) becomes a central question in the discussion of
irreversibility.\footnote{We would like to thank
L. Saint-Raymond and H. Spohn \emph{(private communications)} for
insightful discussions and comments
 regarding the connection to irreversibility.}

Uniform convergence has been addressed by a number of
authors going back to the 1970s. (See 
\cite{Ki1975,PSS2014,BGSRS2016}, and Appendix A of
\cite{vBLLS1980}.) The label \textbf{strong chaos} is reserved
for any notion of chaos for which the convergence at positive times
is strong enough to allow the re-application of the the convergence
theorem taking the evolved solution as initial data. See
Appendix A of \cite{vBLLS1980}, or \cite{BGSRS2016}, for examples of
strong chaos results (note however that the basic technique yielding
uniformity is actually
due to F. King \cite{Ki1975}). \textbf{By definition}, a strong chaos
result \emph{must} account for the directionality of time due to the
fact that Newton's laws are time-reversible whereas Boltzmann's equation
is irreversible; iteration can only be performed forwards in time,
not backwards, so it is a one-sided notion.
  (The term \emph{strong one-sided chaos} is actually
redundant in the context of Lanford's theorem but some authors use
the term \emph{one-sided} in isolation to emphasize the fact that
convergence is occuring only at ``pre-collisional'' points in phase
space.) 
Unlike previous results (except \cite{BGSRS2016}, which represents
independent concurrent work), our strong chaos result implies uniform error
estimates arbitrarily close to the \emph{boundary} of the reduced phase space, which is
significant because the physical interaction is confined to the
boundary. 
Our error estimates are quantitative, as in
\cite{GSRT2014,PSS2014}, though for simplicity of presentation
we will state our main theorems without explicit rates (the estimates in
the proof itself are also much larger than necessary, again for simplicity of
presentation).

\begin{remark}
A very clear exposition on the topic of strong chaos is found in \cite{BGSRS2016};
we feel that the authors have brought great clarity to the topic and we make
no attempt to replicate their exposition.
\end{remark}

\begin{remark}
Note that in the \textbf{original manuscript} of 
Lanford \cite{L1975} (neglecting his follow-up works) the stated result
is a \textbf{weak chaos} result because the assumptions on the initial data
are much stronger than what is proven at positive times, hence iteration in time
is impossible. Lanford clearly acknowledged this shortcoming and understood
the technical steps required to prove a strong chaos result (the details being
filled in by his own student King at roughly the same time).
\end{remark}

A novel aspect of
our analysis is that, given suitably prepared initial data, we can
propagate \emph{partial} factorization even at phase points where
complete factorization necessarily fails (i.e. ``post-collisional''
configurations with $t>0$). As an application of our result,
one obtains the existence of positive measure sets, parameterized
by $\varepsilon$ in a natural way, with measure tending to zero as
$\varepsilon \rightarrow 0$, upon which
$f_N^{(3)} \approx f_N^{(2)} \otimes f_N^{(1)}$ but further factorization is
impossible.\footnote{Note carefully that $f_N^{(2)} \approx f_N^{(1)} \otimes 
f_N^{(1)}$ at \emph{most} phase points, but the theorem holds even at some points
where $f_N^{(2)}$ does not factorize.}
Partial factorization should be viewed as complementary to
results on correlations, such as \cite{PS2014}. Indeed whereas 
\cite{PS2014} gives remarkably precise estimates on the size of correlations,
there was no characterization of the sets on which correlations were concentrated.
On the flip side, we are able to say something about the structure of correlations
but very little about their size. (The partial factorization
 result holds under the assumption of
 perfect factorization at $t=0$, but this is
a standard assumption in the field.\footnote{We are not aware of any
satisfactory explanation for the physical relevance of a perfectly
factorized initial state. There are well-known arguments based on minimization of entropy,
but there is a problem of topologies: an entropically small perturbation will
not be a uniformly small perturbation in general.}) 
The proof of partial factorization
draws significant inspiration from \cite{BGSR2015III,PS2014}, though our
methods are somewhat different. Partial factorization is easily
generalized to
include \emph{non-chaotic} initial data, in the spirit of the Hewitt-Savage
theorem. \cite{HS1955} We emphasize that non-chaotic initial states have
been discussed in the context of irreversibility; see, e.g.,
\cite{BGSRS2016}.

\textbf{Organization of the paper.} 
In Section \ref{sec:2}, we describe
the ideas behind our proof, and we present our main convergence result.
Section \ref{sec:3} gives the precise physical setting for our problem,
along with a crucial comparison principle. Section \ref{sec:4} briefly
introduces the BBGKY and dual BBGKY hierarchies. Section \ref{sec:6}
\& \ref{sec:7} give proofs of \emph{a priori} bounds on the BBGKY hierarchy
by a duality argument; bounds are proven both locally in time for large
data, and globally in time for data sufficiently close to vacuum.
(These \emph{a priori} estimates are not new, but we use a different
approach for the proofs.) Sections \ref{sec:8}, \ref{sec:9}, \ref{sec:10},
\ref{sec:11} \& \ref{sec:12} introduce a number of important technical
tools and results; our main technical contribution is the stability
result in Section \ref{sec:9}. The detailed convergence proof 
(part \emph{(i)} of Theorem \ref{thm:s2-chaos}) is given
in Section \ref{sec:13}. A proof of part \emph{(ii)} of
Theorem \ref{thm:s2-chaos} is presented in Appendix \ref{sec:AppA}.

\section{Statement of main results}
\label{sec:2}

\subsection{Uniform bounds via duality.} 
We begin by briefly describing the role that duality plays in our proof.
Throughout this work we will rely on the BBGKY hierarchy
(Bogoliubov-Born-Green-Kirkwood-Yvon), which is a sequence of equations
describing the evolution of marginals $f_N^{(s)} (t)$ under the
hard sphere flow. One of the key steps in the proof of Lanford's theorem
is to bound a weighted $L^\infty$ norm of the sequence of marginals,
uniformly in $N$, in terms of the initial data. Lanford proves the
uniform bound by re-writing the BBGKY hierarchy using Duhamel's
formula and then setting up a fixed point argument.\footnote{In fact
Lanford wrote out a series expansion for which he proved $L^\infty$ bounds
uniformly in $N$;
this is effectively equivalent to the fixed point argument, and he
had to prove the same collision operator bounds as in \cite{GSRT2014}.}
\cite{L1975,GSRT2014} We have approached the uniform bounds from a 
somewhat different point of view. Our starting point is the
\emph{dual BBGKY hierarchy}, which is (formally) the semigroup whose
generator is the (formal) adjoint of 
 the semigroup generator for the BBGKY hierarchy.
 We refer to    \cite{Ge2013,CGP1997} for background and
results concerning the dual BBGKY hierarchy.

Physically, the dual
BBGKY hierarchy describes the evolution of \emph{observables}.
We are able to bound the growth of observables in a weighted
$\mathcal{L}^1$ space; then, the classical $L^\infty$
 bounds on the marginals
follow by duality. Using the same strategy, with slight revisions,
we are able to prove uniform bounds \emph{globally in time} in the
physical regime considered by Illner \& Pulvirenti.
\cite{IP1986,IP1989}  We emphasize that all of our results concerning
uniform bounds are classical; the only novelty lies in the method of proof.
Note that certain very special observables, such as the kinetic energy,
exhibit cancellation properties (e.g. conservation). However, our proof
concerns \emph{generic} observables; in particular,
 there seems to be no simple way to account for cancellations. Hence 
we cannot report any improvements beyond the perturbative regime
(small time or large mean free path).

It should be noted that, in the context of Lanford's original theorem
\cite{L1975}, the duality argument does not seem to gain us anything
new. However, there are some technical reasons to prefer the dual point of view.
Most fundamentally, it is always possible to consider weak-$*$ limits
of solutions of the dual BBGKY hierarchy (which will converge to
solutions of the dual Boltzmann hierarchy). This limit process will work for
any observable which is not concentrated on certain submanifolds of
high codimension (see Remark \ref{rem-boltz} below). By contrast, passing to the
limit from the BBGKY hierarchy to the Boltzmann hierarchy is an
incredibly delicate process, which is difficult to characterize using 
standard functional spaces. One would hope to use duality to simplify certain
technical questions concerning the BBGKY hierarchy itself, by characterizing
solutions of BBGKY in terms of their action on well-chosen families of 
observables. Note that the differential inequalities we use to prove
Lanford's uniform bounds (at the level of observables) can be adapted to
give more precise information about observables (this is itself a topic
of ongoing research).

Another advantage of duality
 (and the one which served as the original motivation for
this project) is that it gives a somewhat unique proof of uniform bounds on the BBGKY
hierarchy, globally in
time, for a small perturbation of vacuum. \cite{IP1986,IP1989} 
Note that the correct proof given in \cite{IP1989} (as opposed to the incorrect
proof in \cite{IP1986}) relies on a series expansion; indeed,
the proof of \cite{IP1989} cannot really be viewed as a fixed point argument in the
traditional sense.\footnote{By
contrast the corresponding global-in-time estimate for the Boltzmann hierarchy
\emph{is} a fixed point argument but it requires intertwining the free
transport and collision terms. Intertwining in Lanford's proof is fine on
the short time (see the erratum of \cite{GSRT2014} for instance); unfortunately,
 for the global estimate, intertwining will  ruin the dispersive
property because one of the needed inequalities is \emph{false}.}
In particular, certain dispersive-type bounds for moving Maxwellian 
distributions must
be propagated along an arbitrary sequence of particle creations. Our goal was
to have a proof which could be explained using calculus alone, without
a technical induction process. The duality argument trades one complication for
another since we have to be extremely careful in manipulating weighted norms for
observables. Nevertheless, in our view, the proof presented here is more elegant
than that given in \cite{IP1989} and it seems to us the most novel aspect of
this paper.

\begin{remark}
\label{rem-boltz}
One can ask whether it is possible to treat the Boltzmann hierarchy using
duality, in a manner similar to the BBGKY hierarchy. The answer is
``yes, but...'' The problem is that, whereas the dual BBGKY hierarchy
propagates $\mathcal{L}^1$ regularity, solutions of the dual
Boltzmann hierarchy are \emph{measures} even if the data is
smooth.\footnote{This is related to the fact that the Boltzmann hierarchy
is not well-posed on (weighted) $L^\infty$ (without at least an
 assumption like
factorization or exchangeability) but it \emph{is} well-posed for
\emph{continuous} data.}
Unfortunately, the dual Boltzmann hierarchy isn't well-defined
for measure data due to the possibility of simultaneous collisions of
three or more particles. Most likely it is possible to work with the
dual Boltzmann hierarchy by restricting one's attention to measures that
assign zero weight to manifolds of sufficiently high codimension.
However, we prefer not to confront these technical issues; instead, we
prove uniform bounds for the Boltzmann hierarchy using the standard
fixed-point argument.
\end{remark}

\subsection{Strong convergence.}
We now turn to the content of Theorem \ref{thm:s2-chaos} (especially
part \emph{(i)}),
which is our main new result. Essentially the result states that
if \emph{a priori} bounds are known then chaoticity is propagated forwards
in time; the novelty of the result lies in the strength of the notion of
convergence we employ at positive times. The direction of time is built into
our notion of chaoticity, so the theorem \emph{cannot} be applied to prove
propagation of chaos backwards in time. Our convergence result is a type
of \emph{strong chaos} result; this means that we can take the evolved
state at a time $t>0$ and use this state as \emph{initial data} in order to
iterate the convergence to an even later time. The iteration can 
be continued
as long as uniform bounds are known. We emphasize that strong chaos
results are known in both the classical and very recent literature 
\cite{vBLLS1980,BGSRS2016}; however, to our knowledge, the present
convergence
result is the only one which extends to \emph{all} distance scales
$\left\{ |x_i - x_j | > \varepsilon \right\}$ as long as the backwards
trajectories of all $s$ particles are free (minus a small
set in the \emph{velocity} variables only, but see Remark
\ref{rem:opt} below for a discussion of ways to refine the sets of 
convergence).\footnote{The strong chaos result in
\cite{BGSRS2016} requires $|x_i - x_j| \gtrsim
\varepsilon \log \frac{1}{\varepsilon}$.}
Moreover, as we will
see, our proof extends without too much difficulty to obtain 
a \emph{pointwise} description of two-particle correlations 
(higher-order correlations and better error estimates are topics of
ongoing research).\footnote{By contrast, the authors of 
\cite{PS2014} have provided a
 very precise but \emph{averaged} (not pointwise) description of
 correlations.}

\begin{remark}
The notion of chaos that Lanford originally proved (at positive times)
states that the marginals $f_N^{(s)} (t)$ converge pointwise almost
everywhere to tensor products. It can be shown (see \cite{L1975}) that
this notion of chaos (combined with certain uniform estimates)
 implies that for any box
$\Delta \subset \mathbb{R}^d \times \mathbb{R}^d$, the occupation fraction
\begin{equation}
\frac{1}{N} n_{\Delta} (t) = \frac{1}{N} \sum_{i=1}^N
\mathbf{1}_{(x_i (t),v_i (t)) \in \Delta}
\end{equation}
converges in probability to a constant depending only on $t$ and
$\Delta$ when $\varepsilon \rightarrow 0$. The physical interpretation of
Lanford's result is that fluctuations tend to zero as $\varepsilon 
\rightarrow 0$. We emphasize that it is \emph{not}
true that if the marginals converge pointwise almost everywhere
 (at $t=0$, or
even for all $t\in [0,T]$), then the evolution is governed by
Boltzmann's equation. The classical counter-example uses the 
reversibility of Newton's laws combined with the irreversibility of
Boltzmann's equation. \cite{CIP1994} An even more striking 
counter-example has been constructed by T. Bodineau, I. Gallagher,
L. Saint-Raymond and S. Simonella; these authors found an initial
data such that the marginals converge pointwise almost everywhere to tensor
products at $t=0$ (indeed they obtained \emph{uniform} convergence
off explicit small sets), whereas the evolution is given by
\emph{free transport}. \cite{BGSRS2016}  
\end{remark}

We will need to introduce several sets before we can state our main
result; to this end, we will borrow notation from
Section \ref{sec:3}. We will view $\eta > 0$ as a small velocity
cut-off, and $R>0$ as a large velocity cutoff. The most important
sets we will require are defined as follows:
\begin{equation}
\label{eq:s2-K-s}
\mathcal{K}_s = \left\{ Z_s = \left(X_s,V_s\right)
 \in \overline{\mathcal{D}_s} \left|
\psi_s^{-\tau} Z_s = \left( X_s-V_s \tau,V_s \right) \; \forall
\; \tau > 0 \right. \right\} \subset \mathbb{R}^{2ds}
\end{equation}
\begin{equation}
\label{eq:s2-U-s-eta}
\mathcal{U}_s^\eta = \left\{ Z_s = \left(X_s,V_s\right)\in
\overline{\mathcal{D}_s} \left|
\inf_{1\leq i < j \leq s} \left| v_i - v_j \right| > \eta \right. \right\}
\subset \mathbb{R}^{2ds}
\end{equation}
Our \textbf{main result} will concern uniform convergence on the
set $\mathcal{K}_s \cap \mathcal{U}_s^{\eta(\varepsilon)}$ (with
$\eta(\varepsilon)\rightarrow 0$ as $\varepsilon \rightarrow 0$); it is in
proving uniform convergence on such a precise set that a strong chaos
result is obtained.
The condition $Z_s \in \mathcal{K}_s$ means that particles never
collide under the \emph{backwards} particle flow (but, crucially, they
are allowed to collide under the forwards flow). The condition
$Z_s \in \mathcal{U}_s^\eta$ is a technical condition which forces particles
to disperse at an $\eta$-dependent rate; see Remark \ref{rem:opt} for a
few words on how to relax the definition of $\mathcal{U}_s^\eta$.
The remaining sets $\mathcal{G}_s, \mathcal{V}_s^\eta, 
\hat{\mathcal{U}}_s^\eta$, to be defined next, are required only for stating a partial
factorization result, and can be safely skipped.
\begin{equation}
\label{eq:s2-G-s}
\mathcal{G}_s = \left\{ Z_s = (X_s,V_s) \in \overline{\mathcal{D}_s}
\left| 
\begin{aligned}
& \forall \tau > 0,\; \forall 3 \leq i \leq s,\\
& \qquad \qquad \qquad
\left( \psi_s^{-\tau} Z_s \right)_i = (x_i - v_i \tau, v_i) \\
& \textnormal{and, } \forall \tau > 0, \;
\forall 1 \leq i \leq 2, \; \forall 3 \leq j \leq s, \\
& \qquad \qquad \qquad
|(x_i - x_j) - (v_i - v_j)\tau| > \varepsilon
\end{aligned}
\right. \right\}
\end{equation} 
\begin{equation}
\mathcal{V}_s^\eta = \left\{
(Z_s,Z_s^\prime) \in \overline{\mathcal{D}_s} \times
\overline{\mathcal{D}_s} \left| 
\begin{aligned}
& \inf_{1 \leq i \neq j \leq s} |v_i-v_j^\prime| > \eta  \\
& \qquad \qquad  \textnormal{ and } \\
& \inf_{1\leq i \leq s \; : \; v_i \neq v_i^\prime}
|v_i-v_i^\prime| > \eta
\end{aligned}
\right.
\right\}
\end{equation}
\begin{equation}
\label{eq:s2-U-hat-s-eta}
\hat{\mathcal{U}}_s^\eta = \left\{ Z_s = (X_s,V_s) \in
\mathcal{U}_s^\eta \left| \forall \tau,\tau^\prime > 0,
(\psi_s^{-\tau} Z_s,\psi_s^{-\tau^\prime} Z_s) \in
\mathcal{V}_s^\eta  \right. \right\}
\end{equation}
We write $F_N (t) = \left\{ f_N^{(s)} (t,Z_s)\right\}_{1\leq s \leq N}$
where each $f_N^{(s)}(t,Z_s)$ is a function on 
$[0,\infty)\times\overline{\mathcal{D}_s}$ which is symmetric under
interchange of particles.  

\begin{definition}
\label{def:s2-chaos1}
The sequence of initial data
$\left\{ F_N (0) \left| N \in \mathbb{N}\right.\right\}$ is
\emph{nonuniformly $f_0$-chaotic} if, for some $\kappa\in (0,1)$, 
we have for all $s\in\mathbb{N}$ and all $R>0$ that
\begin{equation}
\label{eq:s2-conv}
 \limsup_{N\rightarrow\infty}
\left\Vert\left( f_N^{(s)} (0,Z_s)-f_0^{\otimes s} (Z_s)\right)
\mathbf{1}_{Z_s \in \mathcal{K}_s \cap \mathcal{U}_s^{\eta(\varepsilon)}}
\mathbf{1}_{E_s (Z_s) \leq R^2} \right\Vert_{L^\infty_{Z_s}} = 0
\end{equation}
where $\eta (\varepsilon) = \varepsilon^{\kappa}$
and $N\varepsilon^{d-1} = \ell^{-1}$.
\end{definition}

\begin{definition}
\label{def:s2-chaos2}
The sequence of initial data
$\left\{ F_N (0) \left| N \in \mathbb{N}\right.\right\}$ is
\emph{2-nonuniformly $f_0$-chaotic} if for some $\kappa \in (0,1)$,
we have for all $s \in \mathbb{N}$ and $R>0$ that
\begin{equation}
 \limsup_{N\rightarrow\infty}
\left\Vert\left( f_N^{(s)} (0,Z_s)-f_0^{\otimes s} (Z_s)\right)
\mathbf{1}_{Z_s \in \mathcal{K}_s \cap \mathcal{U}_s^{\eta(\varepsilon)}}
\mathbf{1}_{E_s (Z_s) \leq R^2} \right\Vert_{L^\infty_{Z_s}} = 0
\end{equation}
\emph{and} we have for all $s\in\mathbb{N}$ with $s\geq 3$ and all $R>0$ that
\begin{equation}
\label{eq:s2-conv2}
\begin{aligned}
& \limsup_{N\rightarrow\infty}
\left\Vert\left( f_N^{(s)} (0,Z_s)-\left(
f_N^{(2)} (0) \otimes f_0^{\otimes (s-2)} \right)(Z_s)\right)
\mathbf{1}_{Z_s \in \mathcal{G}_s \cap 
\hat{\mathcal{U}}_s^{\eta(\varepsilon)}}
\mathbf{1}_{E_s (Z_s) \leq R^2} \right\Vert_{L^\infty_{Z_s}}\\
& \qquad \qquad \qquad \qquad \qquad \qquad \qquad \qquad \qquad
\qquad \qquad \qquad \qquad \qquad 
=0
\end{aligned}
\end{equation}
where $\eta (\varepsilon) = \varepsilon^{\kappa}$
and $N\varepsilon^{d-1} = \ell^{-1}$.
\end{definition}

\begin{remark}
An earlier version of this manuscript contained an incorrect statement
of Definition \ref{def:s2-chaos2} which did not even imply weak
chaoticity; we are indebted to one of the anonymous referees for bringing
this to our attention.
\end{remark}

\begin{remark}
The term \emph{nonuniform chaoticity} is motivated by the fact that the
norm of convergence is \emph{based} on the $L^\infty$ (uniform) norm
in $\mathbb{R}^{2ds}$,
yet crucially the convergence \emph{is not uniform} across the whole
phase space. Indeed, very ``thin'' sets of points are necessarily
excluded via the indicator function $\mathbf{1}_{Z_s \in \mathcal{K}_s \cap
\mathcal{U}_s^{\eta (\varepsilon)}}$.
\end{remark}

\begin{remark}
Observe that the sets $\mathcal{G}_s$ appearing in
Definition \ref{def:s2-chaos2} are not symmetric under particle
interchange. Nevertheless, since we assume that the marginals
$f_N^{(s)}$ are symmetric, the uniform error estimates hold on the
image of the set $\mathcal{G}_s \cap
\hat{\mathcal{U}}_s^{\eta(\varepsilon)}$ under any permutation of
particle labels.
\end{remark}

\begin{remark}
The key difference between Definition \ref{def:s2-chaos1} and
Definition \ref{def:s2-chaos2} is that the set
$\mathcal{K}_s \cap \mathcal{U}_s^{\eta(\varepsilon)}$ is replaced
by $\mathcal{G}_s \cap \hat{\mathcal{U}}_s^{\eta(\varepsilon)}$
in (\ref{eq:s2-conv2}).
Hence, the estimate (\ref{eq:s2-conv}) holds only at phase points
which possibly involve collisions in the \emph{future}, but not
the \emph{past}. On the other hand, the estimate (\ref{eq:s2-conv2})
holds even at points where \emph{at most two} of the particles 
have collisions in the \emph{past}. Also note that the structure
of the set $\hat{\mathcal{U}}_s^{\eta(\varepsilon)}$ is more
complicated than that of $\mathcal{U}_s^{\eta(\varepsilon)}$ due to
its dependence on the particle flow $\psi_s^{-t}$. Let us point out
that the usual formulation of Lanford's theorem says nothing at all
about correlations (except to give a large set where, in a functional
sense, correlations are asymptotically negligible). By contrast,
we can apply part \emph{(ii)} of Theorem \ref{thm:s2-chaos} (stated below) 
to provide
definite information about how the fine structure of the second marginal
(beyond its being asymptotically factorized) affects the third marginal.
\end{remark}

\begin{remark}
It is important to realize that complete factorization is allowed
even at positive times along all of 
$\mathcal{K}_s \cap \mathcal{U}_s^{\eta (\varepsilon)}$, but 
only \emph{partial} factorization is possible at some points of
$\mathcal{G}_s \cap \hat{\mathcal{U}}_s^{\eta(\varepsilon)}$ when
$t>0$; this
is due to the fact that collisions generate correlations. In this
sense, 2-nonuniform chaoticity captures (very crudely) the
fine-scale \emph{structure of correlations} at positive times.
This is of interest even for the almost-perfectly factorized 
data of Section \ref{sec:12} because the dynamics will create correlations
no matter how perfect the initial data happens to be.
There has been some recent interest in precisely characterizing
the size of correlations in the Boltzmann-Grad limit; we refer to
\cite{BGSR2015III,PS2014} for some results along these lines.
Compared to these previous results, the main difference with our
result is that we draw a connection between correlations and strong
chaos.
\end{remark}

Recall the Boltzmann equation
\begin{equation}
\label{eq:s2-boltz}
\left(\partial_t + v \cdot\nabla_x \right) f(t) =
\ell^{-1} Q(f(t),f(t))
\end{equation}
\begin{equation}
\label{eq:s2-Q}
Q(f,f) = \int_{\mathbb{R}^d \times \mathbb{S}^{d-1}}
\left[ \omega \cdot (v_1 - v)\right]_+
\left( f (x,v^*) f (x,v_1^*) - f (x,v) f (x,v_1)\right)
d\omega dv_1
\end{equation}

We are able to show:
\begin{theorem}
\label{thm:s2-chaos}
Suppose that the Boltzmann equation (\ref{eq:s2-boltz})
has a non-negative solution $f(t)$ for $t\in[0,T]$,
with $\int f(t) dx dv = 1$,
and further suppose that there exists $\beta_T > 0$ such that
\begin{equation}
\label{eq:s2-boltz-bound}
\sup_{0\leq t \leq T} \sup_{x,v\in\mathbb{R}^d}
 e^{\frac{1}{2} \beta_T |v|^2}
f (t,x,v) < \infty
\end{equation}
and $ f(t) \in W^{1,\infty} (\mathbb{R}^d \times \mathbb{R}^d)$ for
$t\in [0,T]$. Let $F_N (t)$ solve the hard sphere BBGKY hierarchy, under
the Boltzmann-Grad scaling $N\varepsilon^{d-1}=\ell^{-1}$, and
suppose that there is a $\tilde{\beta}_T > 0$,
$\tilde{\mu}_T \in \mathbb{R}$ such that
\begin{equation}
\label{eq:s2-bbgky-bound}
\sup_{N\in\mathbb{N}} 
\sup_{0\leq t \leq T} \sup_{1\leq s \leq N}
 \sup_{Z_s \in \mathcal{D}_s}
e^{\tilde{\beta}_T E_s (Z_s)} e^{\tilde{\mu}_T s}
 \left| f_N^{(s)} (t,Z_s)\right| < \infty
\end{equation}
Then the following holds:\\
(i) If $\left\{ F_N (0) \right\}_N$ is nonuniformly $f_0$-chaotic, 
then for all $t\in [0,T]$, $\left\{ F_N (t) \right\}_N$ is
nonuniformly $f(t)$-chaotic (with the same $\kappa$).  \\
(ii) If
$\left\{ F_N (0) \right\}_N$ is 2-nonuniformly $f_0$-chaotic, then
for all $t \in [0,T]$, $\left\{ F_N (t) \right\}_N$ is
2-nonuniformly $f(t)$-chaotic (with the same $\kappa$).
\end{theorem}
We will prove in full detail part \emph{(i)} of
Theorem \ref{thm:s2-chaos}. The proof of
part \emph{(ii)} is similar to the proof of part \emph{(i)}; the
main differences are the use of an intermediate (Boltzmann-Enskog)
hierarchy, as in \cite{PS2014}, combined with a  refined 
analysis of pseudo-trajectories. We supply the necessary ideas and
all of the key
technical estimates for part \emph{(ii)} in Appendix \ref{sec:AppA}.

\begin{remark}
The time $T$ in Theorem \ref{thm:s2-chaos} is not necessarily
the time in Lanford's original theorem. For instance, in the case of
a sufficiently small perturbation of vacuum
\cite{IP1986,IP1989},  it is permissible to take $T$ arbitrarily
large. More generally,
if the \emph{a priori} estimate (\ref{eq:s2-bbgky-bound}) is known for
a \emph{specific (factorized) solution} of the BBGKY hierarchy up to
time $T$, then we can propagate (2-)nonuniform chaoticity up to time $T$.
 Note that $T$ is necessarily smaller than the existence
time for classical solutions of the Boltzmann equation.
\end{remark}

\begin{remark}
\label{rem:opt}
There is a reasonable question to be asked about the optimality
of the sets we have defined. Certainly the set $\mathcal{K}_s$ can be
improved by specifying a ``horizon'' into the past beyond which collisions
between particles are allowed (indeed this trivial refinement would be necessary when working
in a bounded or periodic domain). More significantly, the condition
defined by $\mathcal{U}_s^\eta$ is clearly not optimal
(we thank the anonymous referees for bringing this issue to our
attention). The set $\mathcal{U}_s^\eta$ was specifically
chosen to simplify the inductive arguments, but it turns out
that the same arguments apply while allowing some particles
to have the same velocity, if they are far apart from each
other. For example,  
let us define
\begin{equation}
\iota (x,v) = \inf_{\tau \in \mathbb{R}}
\left| x - v \tau \right|
\end{equation}
and introduce the sets (for $0 < \eta < 1$)
\begin{equation}
\tilde{\mathcal{U}}_s^\eta =
\left\{ Z_s = (X_s,V_s) \in \overline{\mathcal{D}_s}
\left| \inf_{1\leq i < j \leq s}
\left( \frac{|v_i-v_j|}{\eta} +
\frac{\iota (x_i-x_j,v_i-v_j)}{\eta \log \frac{1}{\eta}}
\right) > 1 \right. \right\}
\end{equation}
Then if $Z_s \in \tilde{\mathcal{U}}_s^\eta$ then for any
$i\neq j$ there
are only two possibilities: either
\emph{(i)} $|v_i - v_j| > \frac{1}{2} \eta$, or
\emph{(ii)} $|v_i - v_j| \leq \frac{1}{2} \eta$ 
in which case we have
\begin{equation}
\inf_{\tau \in \mathbb{R}}
|(x_i-x_j) - (v_i-v_j)\tau| >
 \frac{1}{2} \eta \log \frac{1}{\eta}
\end{equation}
which implies that the two particles $i,j$ can never get close enough to
possibly prevent convergence. Creating particles is easy:
when the choice can be made, we always choose to
enforce $|v_i - v_j| > \eta$. To summarize our argument, we find
that if $\mathcal{U}_s^{\eta(\varepsilon)}$ is replaced
by $\tilde{\mathcal{U}}_s^{\eta(\varepsilon)}$ in
Definition \ref{def:s2-chaos1}, then
Theorem \ref{thm:s2-chaos} is still true (specifically
part \emph{(i)} of the theorem). The proof is
unchanged, apart from
replacing $\mathcal{U}_s^\eta$ by
$\tilde{\mathcal{U}}_s^\eta$  throughout and choosing new
constants where necessary. The sets of convergence can be refined
even further (e.g. it is possible to require simply
$\inf_{i\neq j} \left(|v_i-v_j|+
\left( \log \frac{1}{\eta} \right)^{-1} |x_i-x_j|\right) \gg \eta$)
but we would pay a price, both in terms of readability
of the proof and the error estimates themselves (since a loss is required
at each step of induction). A similar situation holds
with part \emph{(ii)} of Theorem \ref{thm:s2-chaos} but
we omit the details.
\end{remark}

\subsection{Non-chaotic data}
\label{ssec:nonchaos}

So far we have drawn a connection between a particular notion of
chaos and irreversibility. However, chaoticity is \emph{not} a
necessary condition for irreversible behavior.
There is no novelty here. The relation between particle systems
and the Hewitt-Savage theorem is a classical observation of
H. Spohn \cite{Sp1981}, which was also implicit for instance in the work of
O. E. Lanford and others through the use of the Boltzmann hierarchy.
We refer to \cite{CIP1994,BGSRS2016,Sp1991} for expository accounts of the
connection between propagation of chaos and the Hewitt-Savage theorem. 

\begin{remark}
The results of this subsection are, in some sense, not really more general
than Theorem \ref{thm:s2-chaos}, due to the Hewitt-Savage theorem and
the linearity of the BBGKY and Boltzmann hierarchies. This purpose of this
short discussion is simply to emphasize that a good notion of chaos leads
naturally to a good notion of convergence, even for a broad class of
non-chaotic initial conditions.
\end{remark}

 Let us suppose that $f_N^{(s)} (t)$
are the marginals of an underlying $N$-particle probability
density $f_N (t)$ which is symmetric under particle interchange.
Assume that as $N\rightarrow \infty$, the marginals $f_N^{(s)} (t)$
converge to functions $f_\infty^{(s)} (t)$ which satisfy the 
properties of
\emph{non-negativity}, \emph{normalization} and
\emph{consistency} (respectively): (these are all true at finite
$N$ in any case)
\begin{equation}
f_\infty^{(s)} (t) \geq 0
\end{equation} 
\begin{equation}
\int_{\mathbb{R}^{2ds}} f_\infty^{(s)} (t) dZ_s = 1
\end{equation}
\begin{equation}
f_\infty^{(s)} (t,Z_s) =
\int_{\mathbb{R}^{2d}} f_\infty^{(s+1)} (t,Z_{s+1})
dz_{s+1}
\end{equation}
If the functions $\left\{ f_\infty^{(s)} (t) \right\}_{s\in \mathbb{N}}$
are non-negative, normalized, and consistent, and symmetric under
particle interchange, then the Hewitt-Savage theorem 
\cite{HS1955} tells us that there
exists a time-dependent probability measure
$\pi_t \in \mathcal{P} \left( \mathcal{P} \left(
\mathbb{R}^{2d}\right) \right)$ 
\footnote{Here
$\mathcal{P}(X)$ is the set of Borel probability measures on the
Polish space
$X$.} such that
\begin{equation}
f_\infty^{(s)}(t) = \int_{\mathcal{P}\left( \mathbb{R}^{2d}\right)}
h^{\otimes s} (Z_s) d\pi_t (h) 
\end{equation}
Hence, in very great generality, we are free to assume that the limiting
distribution is a convex combination of factorized distributions.
If the convergence of $\left\{ f_N^{(s)} (t) \right\}_{1\leq s \leq N}$
to $\left\{ f_\infty^{(s)} (t) \right\}_{s\in \mathbb{N}}$ is 
sufficiently strong, and we have sufficient control on
solutions to Boltzmann's equation, then it is possible to
explicitly characterize the measure $\pi_t$.

It is possible to show the following result through a slight refinement
of the proof of Theorem \ref{thm:s2-chaos}:
\begin{theorem}
\label{thm:s2-nonchaos}
Let $\pi \in \mathcal{P} \left( \mathcal{P} \left(
\mathbb{R}^{2d}\right) \right)$. Suppose that for 
$\pi-a.e.$ $h_0$ there exists a non-negative solution
$h(t)$ of Boltzmann's equation on $[0,T]$ with
$h(0) = h_0$, and with
$\int h(t) dx dv = 1$, and further suppose that there
exist $C_T,\beta_T > 0$ (which are constants on a set
of full $\pi$-measure) such that
\begin{equation}
\sup_{0\leq t \leq T} \sup_{x,v \in \mathbb{R}^d}
e^{\frac{1}{2} \beta_T |v|^2} h(t,x,v) \leq C_T
\end{equation}
\begin{equation}
\sup_{0 \leq t \leq T} 
\left\Vert h(t) \right\Vert_{W^{1,\infty} (\mathbb{R}^d \times
\mathbb{R}^d)} \leq C_T
\end{equation}
Let $F_N (t)=\left\{ f_N^{(s)} (t) \right\}_{1\leq s \leq N}$
 solve the hard sphere BBGKY hierarchy, under the
Boltzmann-Grad scaling $N\varepsilon^{d-1} = \ell^{-1}$.
Assume that there is a $\tilde{\beta}_T > 0$,
$\tilde{\mu}_T \in \mathbb{R}$ such that
\begin{equation}
\sup_{N\in \mathbb{N}} \sup_{0 \leq t \leq T}
\sup_{1\leq s \leq N} \sup_{Z_s \in \mathcal{D}_s}
e^{\tilde{\beta}_T E_s (Z_s)} e^{\tilde{\mu}_T s}
\left| f_N^{(s)} (t,Z_s) \right| < \infty
\end{equation}
Suppose that for some $\kappa \in (0,1)$, we have for all
$s\in \mathbb{N}$ and all $R > 0$ that
\begin{equation}
\begin{aligned}
& \limsup_{N\rightarrow \infty} \left\Vert \left( f_N^{(s)} (0) -
\int_{\mathcal{P}\left( \mathbb{R}^{2d}\right)} h_0^{\otimes s}
d\pi(h_0) \right) \mathbf{1}_{Z_s \in \mathcal{K}_s
\cap \mathcal{U}_s^{\eta(\varepsilon)}} \mathbf{1}_{E_s (Z_s)
\leq R^2} \right\Vert_{L^\infty_{Z_s}} = 0
\end{aligned}
\end{equation}
where $\eta(\varepsilon) = \varepsilon^\kappa$. Then for all
$t \in [0,T]$, all $s\in \mathbb{N}$, and all $R> 0$ we have:
\begin{equation}
\begin{aligned}
& \limsup_{N\rightarrow \infty} \left\Vert \left( f_N^{(s)} (t) -
\int_{\mathcal{P}\left( \mathbb{R}^{2d}\right)} h(t)^{\otimes s}
d\pi(h_0) \right) \mathbf{1}_{Z_s \in \mathcal{K}_s
\cap \mathcal{U}_s^{\eta(\varepsilon)}} \mathbf{1}_{E_s (Z_s)
\leq R^2} \right\Vert_{L^\infty_{Z_s}} = 0
\end{aligned}
\end{equation}
\end{theorem}
\begin{remark}
To see why Theorem \ref{thm:s2-nonchaos} is true, it is enough to
realize that the proof of Theorem \ref{thm:s2-chaos} is through
a comparison between the BBGKY and Boltzmann hierarchies (similar to
\cite{L1975,GSRT2014}). The Boltzmann hierarchy is linear, and therefore
convex combinations of solutions are again solutions; uniqueness of
the Boltzmann hierarchy follows from the estimates of \cite{L1975}
which are recalled in the present work.
\end{remark}

Theorem \ref{thm:s2-nonchaos} is a generalization of the propagation of
nonuniform chaoticity when there is some uncertainty in the initial
data $h_0$ for Boltzmann's equation.  It is similarly possible to
generalize the propagation of 2-nonuniform chaoticity to the situation
where $h_0$ is random. However, one must be quite careful when dealing
with 2-nonuniform chaoticity because the representation formula
\begin{equation}
\int_{\mathcal{P} (\mathbb{R}^{2d})} h(t)^{\otimes s} 
d\pi (h_0)
\end{equation}
fails in general (when $t>0$)
at phase points for which a collision has occurred
in the \emph{past}. We have, by slight refinements (the details being
mostly notational in nature)
of the proof of Theorem \ref{thm:s2-chaos}, the following result:
\begin{theorem}
\label{thm:s2-nonchaos-2}
Under the assumptions of Theorem \ref{thm:s2-nonchaos}, let us further
suppose that for $\pi-a.e.$ $h_0$ we have sequences
$H_N(t;h_0) = \left\{ h_N^{(s)} (t;h_0) \right\}_{1\leq s \leq N}$
such that $H_N (t;h_0)$ solves the hard sphere BBGKY hierarchy
for $\pi-a.e.\; h_0$ fixed,
and $\left\{ H_N (t;h_0)\right\}_N$ is 2-nonuniformly
$h(t)$-chaotic (with $\kappa$ fixed once and for all)
 for each $t \in [0,T]$. (The existence of such sequences
$\left\{ H_N (t;h_0)\right\}_N$ can be proven on a short time inverval
using Theorem \ref{thm:s2-chaos} and Lanford's uniform bounds.)
Assume as in the statement of Theorem \ref{thm:s2-nonchaos}
that
\begin{equation}
\sup_{N\in \mathbb{N}} \sup_{0 \leq t \leq T}
\sup_{1\leq s \leq N} \sup_{Z_s \in \mathcal{D}_s}
e^{\tilde{\beta}_T E_s (Z_s)} e^{\tilde{\mu}_T s}
\left| h_N^{(s)} (t,Z_s;h_0) \right| < C
\end{equation}
where $C,\tilde{\beta}_T,\tilde{\mu}_T$ are
constant on a set of full $\pi$-measure.
Assume that
\begin{equation}
\begin{aligned}
& \limsup_{N\rightarrow \infty}
\left\Vert \left( f_N^{(s)} (0) - \int_{\mathcal{P}(\mathbb{R}^{2d})}
h_N^{(s)} (0;h_0) d\pi (h_0)\right)
\mathbf{1}_{Z_s \in \mathcal{G}_s \cap
\hat{\mathcal{U}}_s^{\eta (\varepsilon)}}
\mathbf{1}_{E_s (Z_s) \leq R^2}
\right\Vert_{L^\infty_{Z_s}} \\
& \qquad \qquad \qquad \qquad \qquad \qquad \qquad \qquad \qquad
\qquad \qquad \qquad \qquad \qquad = 0
\end{aligned}
\end{equation}
where $\eta (\varepsilon) = \varepsilon^\kappa$. 
Then for all $t \in [0,T]$ we have
\begin{equation}
\begin{aligned}
& \limsup_{N\rightarrow \infty}
\left\Vert \left( f_N^{(s)} (t) - \int_{\mathcal{P}(\mathbb{R}^{2d})}
h_N^{(s)} (t;h_0) d\pi (h_0)\right)
\mathbf{1}_{Z_s \in \mathcal{G}_s \cap
\hat{\mathcal{U}}_s^{\eta (\varepsilon)}}
\mathbf{1}_{E_s (Z_s) \leq R^2}
\right\Vert_{L^\infty_{Z_s}} \\
& \qquad \qquad \qquad \qquad \qquad \qquad \qquad \qquad \qquad
\qquad \qquad \qquad \qquad \qquad = 0
\end{aligned}
\end{equation}
\end{theorem}

\section{Notation and a comparison principle}
\label{sec:3}

We will work in the spatial domain $\mathbb{R}^d$ for some $d\geq 2$. Let
$\varepsilon > 0$ and $N \in \mathbb{N}$ satisfy the Boltzmann-Grad
scaling $N \varepsilon^{d-1} = \ell^{-1}$ for some fixed 
parameter $\ell > 0$\footnote{The parameter $\ell$ is of order the mean
free path length, insofar as the mean free path is well-defined.};
we will henceforth suppress
the implicit dependence on $\varepsilon,\ell$ in our notation, 
though they will be retained in formulas and estimates.
If $1\leq s \leq N$ then we define the phase space  
\begin{equation}
\label{eq:s3-D-s}
\mathcal{D}_s =
\left\{ \left. Z_s = \left(X_s,V_s\right) \in
 \mathbb{R}^{ds} \times \mathbb{R}^{d s} \right| |x_i-x_j| > \varepsilon \;
\forall \; 1 \leq i < j \leq s \right\}
\end{equation}
Suppose $Z_s \in \partial \mathcal{D}_s$, with
$x_j = x_i + \varepsilon \omega$, $\omega \in \mathbb{S}^{d-1}$,
$\omega \cdot (v_j - v_i)\neq 0$, $i<j$, and 
$|x_{j^\prime} - x_{i^\prime}| > \varepsilon$ whenever
$i^\prime < j^\prime$ and $(i^\prime,j^\prime) \neq (i,j)$; then
the image point $Z_s^* = \left( x_1,v_1,\dots,x_i,v_i^*,\dots,
x_j,v_j^*,\dots,x_s,v_s\right)$ is defined by the following rule:
\begin{equation}
\label{eq:s3-collision}
\begin{cases}
v_i^* = v_i + \omega \omega \cdot \left( v_j-v_i\right)\\
v_j^* = v_j - \omega \omega \cdot \left( v_j-v_i\right) 
\end{cases}
\end{equation}
Note that the map $Z_s \mapsto Z_s^*$ is a measurable involution of
$\partial \mathcal{D}_s$; and, in the identity $Z_s^{**} = Z_s$ a.e.
 $Z_s \in \partial \mathcal{D}_s$, we use the same 
$\omega \in \mathbb{S}^{d-1}$ for each transformation.

Let us denote by $\psi_s^t Z_s$ the image of $Z_s$ under the forward
time evolution of $s$ hard spheres at time $t$; that is, if
$Z_s = Z_s (0)$, and the function $Z_s (t) = \left( X_s (t) , V_s (t)\right)$
is piecewise differentiable\footnote{classically differentiable on the
complement of a closed set of isolated points} and has left and right limits
at all points $t\in \mathbb{R}$, and there holds 
\begin{equation}
\label{eq:s3-hsflow}
\begin{cases}
\frac{d}{dt}  Z_s (t)  = \left( V_s (t), 0\right) \;\;
\textnormal{ if } \;\; Z_s (t)
 \notin \partial \mathcal{D}_s \\
Z_s (t^+) = \left(Z_s (t^-) \right)^* \;\; \textnormal{ if }
 Z_s (t) \in \partial \mathcal{D}_s 
\end{cases}
\end{equation}
for all $t\in \mathbb{R}$ then we write $\psi_s^t Z_s = Z_s (t)$.
This ``definition,'' unfortunately, does \emph{not} define $\psi_s^t Z_s$
uniquely in general, since there is no way to continuously extend the
map $Z_s \mapsto Z_s^*$ to all of $\partial \mathcal{D}_s$. Indeed,
discontinuities will be observed whenever one particle
simultaneously collides with at least two other particles.
Nevertheless, up to deletion of a Lebesgue measure zero subset of
initial phase points $Z_s \in \mathcal{D}_s$, we may assume that
$\psi_s^t Z_s$ is defined for all $t\in \mathbb{R}$, that all collisions
are non-grazing, and that all collisions are binary and linearly ordered
in time (i.e. disjoint \emph{pairs} of particles do not simultaneously 
collide). \cite{Al1975} One can then show that, for each 
$t\in\mathbb{R}$, $\psi_s^t$ may be
viewed as a measurable map $\mathcal{D}_s \rightarrow \mathcal{D}_s$
preserving the induced Lebesgue measure. On bounded time intervals, the
map $(t,Z_s)\mapsto \psi_s^t Z_s$ is actually \emph{jointly continuous}
away from certain higher codimension submanifolds of the domain, provided
that one chooses to identify $Z_s \in \partial \mathcal{D}_s$ with its
image $Z_s^*$. However, we will not make such an identification;
instead, we choose to enforce the convention that, for a.e.
$Z_s \in \mathcal{D}_s$, there holds for all $t\in \mathbb{R}$ that
$\psi_s^t Z_s = \psi_s^{t+} Z_s$. We will say that a point 
$Z_s \in \partial \mathcal{D}_s$ is a \emph{pre-collisional configuration}
if $Z_s = \psi_s^{t-} Z_s$; or, we will call it a
\emph{post-collisional configuration} if $Z_s = \psi_s^{t+} Z_s$.
Note in particular that, according to our conventions,
$Z_s \neq \psi_s^0 Z_s$ for a.e. pre-collisional
 $Z_s \in \partial \mathcal{D}_s$.

Suppose $f_N (0)$ is a probability measure supported on 
$\overline{\mathcal{D}_N}$
and absolutely continuous with respect to the Lebesgue measure on
$\mathbb{R}^{2dN}$; by abuse of notation, we call the corresponding density
$f_N (0,Z_N)$. We will denote by $\mathcal{S}_N$ the symmetric group on
$N$ indices; if $Z_N \in \mathcal{D}_N$ then $\sigma \in \mathcal{S}_N$
acts on $Z_N$ by permutation of particle indices:
$\sigma (z_1,\dots,z_N) = (z_{\sigma (1)},\dots,z_{\sigma (N)})$.
We will always assume that $f_N (0)$ is 
\emph{symmetric}, i.e. for any $\sigma \in \mathcal{S}_N$
there holds $f_N (0,\sigma Z_N) = f_N (0,Z_N)$. Then for
$t\in \mathbb{R}$ we will define $f_N (t,Z_N) = f_N (0,\psi_N^{-t} Z_N)$;
equivalently, since $\psi_N^t$ preserves Lebesgue measure on
$\mathbb{R}^{2dN}$, we can say that $f_N (t)$ is the pushforward of
$f_N (0)$ under $\psi_N^t$. We will denote
$Z_{s:s+k} = (z_s,z_{s+1},\dots,z_{s+k})$,
$Z_s^{(i)} = (z_1,\dots,z_{i-1},z_{i+1},\dots,z_s)$, and
similarly $Z_{s:s+k}^{(i)}$ in the case $s \leq i\leq s+k$.
We extend $f_N (t)$ by zero so that it is defined on $\mathbb{R}^{2dN}$;
then the marginals $f_N^{(s)} (t,Z_s)$ are defined by
$f_N^{(s)} (t,Z_s) = \int_{\mathbb{R}^{2d(N-s)}}
f_N (t,Z_N) dZ_{(s+1):N}$. Each $f_N^{(s)} (t)$ is a symmetric probability
density supported on $\overline{\mathcal{D}_s}$; and, since
$f_N^{(s)} (t)$ is the marginal of $f_N^{(s+1)}(t)$ for each
$1 \leq s < N$, we say that the sequence
$\left\{f_N^{(s)} (t)\right\}_{1\leq s \leq N}$ is
\emph{consistent}. We also define the energy
$E_s (Z_s) = \frac{1}{2} \sum_{i=1}^s |v_i|^2$, and we will also let
$I_s (Z_s) = \frac{1}{2} \sum_{i=1}^s |x_i|^2$, and additionally
$\mathcal{Y}_s (Z_s) = \sum_{i=1}^s x_i \cdot v_i$.

\begin{remark}
Sometimes we will want to consider sequences
$\left\{ f_N^{(s)} \right\}_{1\leq s \leq N}$ which are \emph{not}
consistent, and not necessarily normalized nor even non-negative.
We will only point out this distinction when it is important for
the analysis. For the remainder of this section, we will assume that
$\left\{ f_N^{(s)} \right\}_{1\leq s \leq N}$ is a consistent
sequence of symmetric probability densities.
\end{remark}

We now turn to a comparison principle; this result is due to
Illner \& Pulvirenti \cite{IP1986,IP1989,I1989}
 and is specific to the whole space case.

\begin{lemma}
\label{lemma:s3-ip1}
For a.e. $Z_s = (X_s,V_s) \in \mathcal{D}_s$ and all $t \geq 0$,
\begin{equation}
\label{eq:s3-ip2}
\mathcal{Y}_s (\psi_s^t Z_s)  \geq
2 t E_s (Z_s) + \mathcal{Y}_s (Z_s)
\end{equation}
\end{lemma}
\begin{proof}
Fix $Z_s \in \mathcal{D}_s$ such that $\psi_s^t Z_s$ is defined for
all $t\in \mathbb{R}$, with all collisions binary and non-grazing. Let
 $r (t) = \mathcal{Y}_s (\psi_s^t Z_s)
- 2 t E_s (\psi_s^t Z_s)$; then $r (0) = \mathcal{Y}_s (Z_s)$. 
Between collisions we have $\frac{d}{dt} r(t) = 0$, and $r$ can only
increase across collisions. We use the energy conservation identity
$E_s (\psi_s^t Z_s) = E_s (Z_s)$ to conclude.
\qed
\end{proof}

\begin{lemma}
\label{lemma:s3-ip3}
For a.e. $Z_s = (X_s,V_s) \in \mathcal{D}_s$ and all $t \in \mathbb{R}$,
\begin{equation}
\label{eq:s3-ip4}
I_s (\psi_s^t Z_s) \geq I_s ((X_s + V_s t,V_s))
\end{equation}
\end{lemma}
\begin{proof}
Due to time reversibility, it suffices to consider the case $t\geq 0$.
Fix $Z_s \in \mathcal{D}_s$ such that $\psi_s^t Z_s$ is defined for all
$t\in \mathbb{R}$, with all collisions binary and non-grazing. Let
$b(t) = I_s (\psi_s^t Z_s) - I_s ((X_s + V_s t,V_s))$; observe that
$b(0)=0$, and $b(t)$ is continuous and piecewise smooth. Between
collisions we have
\begin{equation}
\label{eq:s3-dbdt}
\frac{d}{dt} b(t) = \mathcal{Y}_s (\psi_s^t Z_s)
- 2 t E_s (Z_s) - \mathcal{Y}_s (Z_s) \geq 0 
\end{equation}
where we have used Lemma \ref{lemma:s3-ip1}. Therefore
$b(t)\geq 0$ for all $t>0$, and the result follows.
\qed
\end{proof}

\section{The BBGKY and dual BBGKY hierarchies}
\label{sec:4}

The BBGKY hierarchy is a sequence of equations which describe the
evolution of the marginals $f_N^{(s)} (t)$ of a solution
$f_N (t)$ of Liouville's equation. The BBGKY hierarchy is
one of the classical tools in the mathematical analysis of many-particle
systems. Many derivations of the BBGKY hierarchy have been devised;
we refer to \cite{GSRT2014}, which will be the approach most convenient
for us. We give a slightly generalized version of the weak form of the
BBGKY hierarchy derived in \cite{GSRT2014}, since it will enable us
to easily read off the \emph{dual BBGKY hierarchy}. The dual BBGKY
hierarchy is the sequence of equations whose semigroup generator is
the adjoint of that of the BBGKY hierarchy. We will be using the dual
BBGKY hierarchy in order to derive uniform bounds in Sections
\ref{sec:6} and \ref{sec:7}. The main advantage of the dual BBGKY 
hierarchy is that the semigroup generator makes sense without strong
regularity assumptions; this is useful because the BBGKY hierarchy does
not propagate smoothness of the marginals.

Suppose we are given a sequence of functions
$\left\{ f_N^{(s)} (t,Z_s) \right\}_{1 \leq s \leq N}$, with
$f_N^{(s)}$ defined on
$[0,\infty) \times \overline{\mathcal{D}_s}$ and
$\left(\partial_t + V_s\cdot\nabla_{X_s}\right)f_N^{(s)} \in
L^1 \left( \mathcal{O} \right)$ for any bounded open set
$\mathcal{O} \subset [0,\infty)\times \overline{\mathcal{D}_s}$. Further
suppose the marginals satisfy permutation symmetry and the
boundary condition $f_N^{(s)} (t,Z_s^*) = f_N^{(s)} (t,Z_s)$
for a.e. $(t,Z_s) \in[0,\infty)\times \partial \mathcal{D}_s$. Then we
will say that the sequence 
$\left\{f_N^{(s)} (t,Z_s) \right\}_{1\leq s \leq N}$ \emph{solves the
weak form of the BBGKY hierarchy} provided that for every test function
$\phi_s (t,Z_s)
\in C^1_c ([0 ,\infty)\times\overline{\mathcal{D}_s})$, satisfying
permutation symmetry, there holds:
\begin{equation}
\label{eq:s4-def-bbgky}
\begin{aligned}
& \int_0^\infty \int_{\mathcal{D}_s} \left[\left(
\partial_t + V_s \cdot \nabla_{X_s} \right) \phi_s (t,Z_s) \right]
f_N^{(s)} (t,Z_s) dZ_s dt = \\
& = \int_{\mathcal{D}_s}\phi_s(0,Z_s)f_N^{(s)} (0,Z_s)dZ_s\\
&-\varepsilon^{d-1} \sum_{1 \leq i < j \leq s} \int_0^\infty
\int_{\mathbb{R}^{ds} \times \mathbb{R}^{d(s-1)} \times 
\mathbb{S}^{d-1}} \mathbf{1}_{Z_s \in \partial \mathcal{D}_s}
\omega \cdot (v_j - v_i) \times \\
& \;\;\;\;\;\;\;\;\;\;\;\;\;\;\;\;\;\;\;\; \times
\left( \phi_s f_N^{(s)} \right)
(t,\dots,x_i,v_i,\dots,x_i + \varepsilon \omega,v_j,\dots) 
d\omega dX_s^{(j)} dV_s dt \\
&  - (N-s)\varepsilon^{d-1}  \sum_{1 \leq i \leq s} \int_0^\infty
\int_{\mathbb{R}^{d s} \times \mathbb{R}^{ds} \times \mathbb{R}^d
\times  \mathbb{S}^{d-1}}
\mathbf{1}_{Z_{s+1} \in \partial \mathcal{D}_{s+1}}
\omega \cdot (v_{s+1} - v_i) \times \\
& \;\;\;\;\;\;\;\;\;\;\;\;\;\;\;\;\;\;\;\; \times
\phi_s (t,Z_s)  f_N^{(s+1)} (t,Z_s,x_i + \varepsilon \omega,v_{s+1}) 
 d\omega dv_{s+1} dX_s dV_s dt
\end{aligned}
\end{equation}
If $f_N(0) \in \mathcal{C}_0^\infty (\mathcal{D}_N)$ and $f_N (t)$
satisfies Liouville's equation, then the sequence of marginals
$\left\{f_N^{(s)} (t) \right\}_{1\leq s \leq N}$ solves the weak form
of the BBGKY hierarchy. However, note that it is also possible to have
solutions of the BBGKY hierarchy which are \emph{not} sequences of
marginals. Under suitable re-scalings, such solutions may have
physical interpretations in the \emph{grand canonical ensemble},
where the total number of particles is considered random. In our
treatment, however, we will always be working in the 
\emph{canonical ensemble}, since the total number of particles
is just $N$.

We now turn to the dual BBGKY hierarchy. Given a pair of densities
$F_N =\left\{f_N^{(s)}\right\}_{1 \leq s \leq N}$
and test functions $\Phi_N = \left\{\phi_N^{(s)}\right\}_{1 \leq s \leq N}$,
with each $f_N^{(s)},\phi_N^{(s)}$ symmetric under particle interchange,
we define a duality bracket \cite{Ge2013}:
\begin{equation}
\label{eq:s5-duality1}
\left< \Phi_N, F_N\right> = \sum_{s=1}^N \frac{1}{s!}
\int_{\mathcal{D}_s}
\phi_N^{(s)} (Z_s) f_N^{(s)} (Z_s) dZ_s
\end{equation}
We would like to define the dual BBGKY hierarchy by the following duality
relation:
\begin{equation}
\label{eq:s5-duality2}
\left< \Phi_N (t),F_N (0)\right> = \left<\Phi_N (0),F_N (t)\right>
\end{equation}
which should hold whenever $F_N (t)$ solves the BBGKY hierarchy
and $\Phi_N (t)$  solves the dual BBGKY hierarchy.
Applying (\ref{eq:s5-duality2}) and considering arbitrary weak solutions
$F_N (t)$ of the BBGKY hierarchy, one can show that observables evolve
according to the following hierarchy of equations (this is equivalent
to equation 15 in \cite{Ge2013}, up to trivial re-scaling):
\begin{equation}
\label{eq:s5-def-dual-bbgky}
\left( \partial_t - V_s \cdot \nabla_{X_s} \right) \phi_N^{(s)} (t,Z_s) = 0
\qquad  \left( Z_s \in \mathcal{D}_s, \; s = 1,\dots,N \right)
\end{equation}
\begin{equation}
\label{eq:s5-bc-dual-bbgky}
\begin{aligned}
\frac{\phi_N^{(s)}(t,Z_s^*)}{N-s+1} & + \phi_N^{(s-1)}(t,(Z_s^*)^{(i)})
+ \phi_N^{(s-1)}(t,(Z_s^*)^{(j)}) = \\
& =  \frac{\phi_N^{(s)}(t,Z_s)}{N-s+1}
 + \phi_N^{(s-1)}(t,Z_s^{(i)}) + \phi_N^{(s-1)}(t,Z_s^{(j)})
\end{aligned}
\end{equation}
\begin{equation*}
\quad \quad \quad \quad \quad \quad \quad \quad \quad 
 \left(Z_s\in\left(\Sigma_s (i,j)\times \mathbb{R}^{d s}\right)
\cap \partial \mathcal{D}_s,
\; s = 2,\dots,N\right)
\end{equation*}
where
\begin{equation}
\Sigma_s (i,j) = \left\{
X_s \in \mathbb{R}^{ds} \left| |x_i-x_j|=\varepsilon \right. \right\}
\end{equation}
Given an initial data $\phi_N^{(s)}(0)$, $1 \leq s \leq N$, we can solve
this hierarchy recursively. The nonzero observable of lowest order (at
the initial time, and therefore all time) simply evolves via the backwards
Liouville flow. Once $\phi_N^{(s-1)}(t)$ is known for all $t\geq 0$, it is
possible to determine $\phi_N^{(s)}(t)$ by integrating along characteristics.
One uses the knowledge of $\phi_N^{(s-1)}$ to determine the amount by which
$\phi_N^{(s)}$ ``jumps'' at particle collisions. Let us point out that
as $Z_s$ ranges over an open subset of
$\left(\Sigma_s(i,j) \times \mathbb{R}^{d s}\right)
\cap \partial \mathcal{D}_s$, the coordinates
$Z_s^{(i)},\dots,$ cover an open subset of $\mathcal{D}_{s-1}$.
Thus the source terms arising from $\phi_N^{(s-1)}$ are always well-defined
functions on the set $\partial \mathcal{D}_s$. Note that, by a density
argument involving a Duhamel-type formula, it is possible to use initial
data $\Phi_N (0)$ which does not satisfy the boundary condition
(\ref{eq:s5-bc-dual-bbgky}).

\section{Local \emph{a priori} bounds on observables}
\label{sec:6}

We will prove weighted $\mathcal{L}^1$ bounds on observables which are
independent of $N$; the stylized $\mathcal{L}$ is intended to distinguish
the spaces in which we bound observables. The proof is a dualization of the 
classical proof of \emph{a priori} bounds on the marginals $f_N^{(s)}$ in
weighted $L^\infty$ spaces, originally due to Lanford. 
\cite{L1975,GSRT2014} As in the case of Lanford's theorem, the 
\emph{a priori} bounds will only hold on a short time interval. Let us fix
weight parameters $\beta > 0, \mu \in \mathbb{R}$, and define the norms
\begin{equation}
\label{eq:s6-norm1}
\left\Vert \Phi_N \right\Vert_{\mathcal{L}^1_{\beta,\mu}} =
\sum_{s=1}^N \frac{1}{s!} \int_{\mathcal{D}_s}
\left| \phi_N^{(s)} (Z_s) \right| e^{-\beta E_s(Z_s)} e^{-\mu s} dZ_s
\end{equation}
\begin{equation}
\label{eq:s6-norm2}
\left| F_N \right|_{L^\infty_{\beta,\mu}} =
\sup_{1\leq s \leq N} \sup_{Z_s \in \mathcal{D}_s}
\left| f_N^{(s)} (Z_s) \right| e^{\beta E_s (Z_s)} e^{\mu s}
\end{equation}
where $E_s (Z_s) = \frac{1}{2} \sum_{i=1}^s |v_i|^2$. Then we have
\begin{equation}
\label{eq:s6-duality1}
\left< \Phi_N, F_N \right> \leq
\left\Vert \Phi_N \right\Vert_{\mathcal{L}^1_{\beta,\mu}}
\left| F_N \right|_{L^\infty_{\beta,\mu}}
\end{equation}
Since $\phi_N^{(s)}$ is transported along characteristics within
$\mathcal{D}_s$, $\left| \phi_N^{(s)} (t,Z_s) \right|$
is transported as well. Therefore we have
\begin{equation*}
\begin{aligned}
& \frac{\partial}{\partial t} \int_{\mathcal{D}_s}
\left| \phi_N^{(s)} (t,Z_s)\right| e^{-\beta E_s (Z_s)} e^{-\mu s}dZ_s=\\
& = \int_{\mathcal{D}_s} V_s \cdot \nabla_{X_s} \left|
\phi_N^{(s)} (t,Z_s) \right| e^{-\beta E_s (Z_s)} e^{-\mu s} dZ_s \\
& = \sum_{1 \leq i < j \leq s}\int_{\mathbb{R}^{ds}}
\int_{\Sigma_s (i,j)} n^{i,j} \cdot V_s
\left| \phi_N^{(s)} (t,Z_s) \right| e^{-\beta E_s (Z_s)}
e^{-\mu s} d\sigma^{i,j} dV_s \\
& = \frac{1}{2} \sum_{1 \leq i < j \leq s}
\int_{\mathbb{R}^{ds}} \int_{\Sigma_s (i,j)}
n^{i,j} \cdot V_s \times \\
&  \qquad \qquad \times \left( \left| \phi_N^{(s)} (t,Z_s) \right| -
\left| \phi_N^{(s)} (t,Z_s^*) \right|\right)
 e^{-\beta E_s (Z_s)} e^{-\mu s} d\sigma^{i,j} dV_s \\
& \leq \frac{1}{2} \sum_{1 \leq i < j \leq s}
\int_{\mathbb{R}^{ds}} \int_{\Sigma_s (i,j)}
\left| n^{i,j} \cdot V_s \right| \times \\
& \qquad \qquad
 \times \left| \phi_N^{(s)} (t,Z_s) - \phi_N^{(s)} (t,Z_s^*) \right|
 e^{-\beta E_s (Z_s)} e^{-\mu s} d\sigma^{i,j} dV_s \\
\end{aligned}
\end{equation*}
Now we employ the boundary condition to write
\begin{equation*}
\begin{aligned}
& \frac{\partial}{\partial t} \int_{\mathcal{D}_s}\left| 
\phi_N^{(s)} (t,Z_s)\right| e^{-\beta E_s(Z_s)} e^{-\mu s} dZ_s \leq \\
& \leq \frac{N}{2} \sum_{1 \leq i < j \leq s}
\int_{\mathbb{R}^{ds}} \int_{\Sigma_s (i,j)}
\left| n^{i,j} \cdot V_s \right| \times \\
&  \times
\left| \phi_N^{(s-1)} (t,(Z_s^*)^{(i)}) + \phi_N^{(s-1)}(t,(Z_s^*)^{(j)})
- \phi_N^{(s-1)}(t,Z_s^{(i)}) - \phi_N^{(s-1)}(t,Z_s^{(j)}) \right|\times\\
& \times e^{-\beta E_s (Z_s)} e^{-\mu s} d\sigma^{i,j}dV_s\\
& = \frac{N}{4} \sum_{i\neq j = 1}^s
\int_{\mathbb{R}^{ds}} \int_{\Sigma_s (i,j)}
\left| n^{i,j} \cdot V_s \right| \times \\ 
& \times
\left| \phi_N^{(s-1)} (t,(Z_s^*)^{(i)}) + \phi_N^{(s-1)}(t,(Z_s^*)^{(j)})
- \phi_N^{(s-1)}(t,Z_s^{(i)}) - \phi_N^{(s-1)}(t,Z_s^{(j)}) \right|\times\\
& \times e^{-\beta E_s (Z_s)} e^{-\mu s} d\sigma^{i,j} dV_s\\
& \leq \frac{N}{4} \sum_{i\neq j = 1}^s
\int_{\mathbb{R}^{ds}} \int_{\Sigma_s (i,j)}
\left| n^{i,j} \cdot V_s \right|
\left(\left| \phi_N^{(s-1)} (t,(Z_s^*)^{(i)}) \right| + \left|
\phi_N^{(s-1)}(t,(Z_s^*)^{(j)})\right| + \right. \\
& \qquad \qquad \left. + \left|
 \phi_N^{(s-1)}(t,Z_s^{(i)})\right| + 
\left|\phi_N^{(s-1)}(t,Z_s^{(j)})\right|\right)
 e^{-\beta E_s (Z_s)} e^{-\mu s}d\sigma^{i,j}dV_s\\
&= N \sum_{i \neq j = 1}^s \int_{\mathbb{R}^{ds}}
\int_{\Sigma_s (i,j)}
\left| n^{i,j} \cdot V_s \right|
\left| \phi_N^{(s-1)} (t,Z_s^{(i)})\right|
e^{-\beta E_s (Z_s)} e^{-\mu s}d\sigma^{i,j}dV_s
\end{aligned}
\end{equation*}
We can generalize this inequality to the case of time-dependent weights.
\begin{equation}
\label{eq:s6-bound1}
\begin{aligned}
& \frac{\partial}{\partial t} \int_{\mathcal{D}_s}\left| 
\phi_N^{(s)} (t,Z_s)\right| e^{-\beta(t) E_s (Z_s)} e^{-\mu(t) s} dZ_s
\leq \\
& \leq N \sum_{i \neq j = 1}^s \int_{\mathbb{R}^{ds}}
\int_{\Sigma_s (i,j)}
\left| n^{i,j} \cdot V_s\right| \times \\
& \qquad \qquad \times 
\left| \phi_N^{(s-1)} (t,Z_s^{(i)})\right|
 e^{-\beta(t) E_s (Z_s)} e^{-\mu(t) s}
d\sigma^{i,j}dV_s + \\
& \;\;\;\;\; +\int_{\mathcal{D}_s}
\left|\phi_N^{(s)} (t,Z_s)\right|
\left\{ -\beta^\prime (t) E_s (Z_s) - \mu^\prime (t) s\right\}
e^{-\beta(t) E_s (Z_s)} e^{-\mu(t) s} dZ_s
\end{aligned}
\end{equation}
Note that in the case $s=1$ the first term on the RHS vanishes (there are
no source terms at the boundary).

Let us estimate just the first term. The integral over the hypersurface
$\Sigma_s (i,j) = \left\{ X_s \in \mathbb{R}^{d s} \left|
|x_i - x_j| = \varepsilon\right.\right\}$ brings down a factor
of $\varepsilon^{d-1}$, which is then eliminated by virtue of the
scaling $N \varepsilon^{d-1} = \ell^{-1}$.
\begin{equation*}
\begin{aligned}
& N \sum_{i \neq j = 1}^s \int_{\mathbb{R}^{ds}}
\int_{\Sigma_s (i,j)}
\left| n^{i,j} \cdot V_s \right|
\left| \phi_N^{(s-1)} (t,Z_s^{(i)})\right|
e^{-\beta(t) E_s (Z_s)} e^{-\mu(t) s}
d\sigma^{i,j}dV_s \leq \\
& \leq \ell^{-1} \sum_{i=1}^s \int_{\mathbb{R}^{ds}} 
\int_{\mathbb{R}^{d(s-1)}} \int_{\mathbb{S}^{d-1}}
\left( \sum_{j\neq i} \left| \omega \cdot (v_j - v_i)\right|\right)
\left| \phi_N^{(s-1)}(t,Z_s^{(i)})\right| \times\\
&\qquad \qquad 
\times e^{-\beta(t) E_s (Z_s)} e^{-\mu(t) s} d\omega dX_s^{(i)} dV_s \\
& \leq \ell^{-1} \sum_{i=1}^s \int_{\mathbb{R}^{ds}} 
\int_{\mathbb{R}^{d(s-1)}} \int_{\mathbb{S}^{d-1}} 
\left| \phi_N^{(s-1)}(t,Z_s^{(i)})\right|\times \\
&  \times \left(\sqrt{2} (s-1)^{\frac{1}{2}} E_{s-1}(Z_s^{(i)})^{\frac{1}{2}}
+ (s-1)|v_i| \right)
e^{-\beta(t) E_s (Z_s)} e^{-\mu(t) s} d\omega dX_s^{(i)} dV_s \\
\end{aligned}
\end{equation*}
\begin{equation*}
\begin{aligned}
& \leq  C_d \ell^{-1} e^{-\mu(t)} \beta(t)^{-\frac{d}{2}} s
\int_{\mathbb{R}^{d (s-1)}\times\mathbb{R}^{d(s-1)}}
\left|\phi_N^{(s-1)}(t,Z_{s-1})\right|\times\\
&\qquad \times \left( (s-1)^{\frac{1}{2}} E_{s-1}(Z_{s-1})^{\frac{1}{2}}
+ (s-1) \beta(t)^{-\frac{1}{2}} \right) \times \\
& \qquad \qquad \times
e^{-\beta(t) E_{s-1} (Z_{s-1})} e^{-\mu(t)(s-1)} dX_{s-1} dV_{s-1}  \\
\end{aligned}
\end{equation*}
We may sum over $s$ to obtain:
\begin{equation}
\label{eq:s6-bound2}
\begin{aligned}
& \frac{\partial}{\partial t}
\left\Vert \Phi_N (t)\right\Vert_{\mathcal{L}^1_{\beta(t),\mu(t)}} \leq \\
& \leq \sum_{s=2}^N \frac{1}{s!}
C_d \ell^{-1} e^{-\mu(t)} \beta(t)^{-\frac{d}{2}} s
\int_{\mathcal{D}_{s-1}}
\left| \phi_N^{(s-1)}(t,Z_{s-1})\right|\times \\
&  \times \left( (s-1)^{\frac{1}{2}} E_{s-1}(Z_{s-1})^{\frac{1}{2}}
+ \frac{(s-1)}{\beta(t)^{\frac{1}{2}}} \right)
e^{-\beta(t) E_{s-1} (Z_{s-1})} e^{-\mu(t)(s-1)} dZ_{s-1} +  \\
& + \sum_{s=1}^N \frac{1}{s!} \int_{\mathcal{D}_s}
\left| \phi_N^{(s)} (t,Z_s)\right| \left\{ -\beta^\prime (t) E_s (Z_s)
- \mu^\prime (t) s \right\} e^{-\beta(t) E_s (Z_s)} e^{-\mu(t) s} dZ_s
\end{aligned}
\end{equation}

We re-index the first term and combine; we furthermore assume that
$\beta^\prime (t), \mu^\prime(t) > 0$ (this is \emph{opposite} the
usual convention because of duality). Then we have:
\begin{equation}
\label{eq:s6-bound3}
\begin{aligned}
& \frac{\partial}{\partial t}
\left\Vert \Phi_N (t) \right\Vert_{\mathcal{L}^1_{\beta(t),\mu(t)}} \leq \\
& \leq \sum_{s=1}^{N-1} \frac{1}{s!} \int_{\mathcal{D}_s}
\left| \phi_N^{(s)} (t,Z_s) \right| \times \\
& \times
\left[ C_d \ell^{-1} e^{-\mu (t)} \beta(t)^{-\frac{d}{2}} \left(
s^{\frac{1}{2}} E_s (Z_s)^{\frac{1}{2}} + s \beta(t)^{-\frac{1}{2}}\right)
- \beta^\prime(t) E_s (Z_s) - \mu^\prime(t) s\right] \times \\
& \qquad \qquad \times e^{-\beta(t) E_s (Z_s)} e^{-\mu(t) s} dZ_s
\end{aligned}
\end{equation}
It is now apparent that $\Phi_N (t)$ is controlled as long as the
quantity in brackets is everywhere nonpositive, for $0 \leq t \leq T$
and $Z_s \in \mathcal{D}_s$. For instance, let us suppose
that $\beta_0 > 0$, $ \mu_0 \in \mathbb{R}$ are given. Then as long as
$T_L >0$ is chosen so that
\begin{equation}
\label{eq:s6-T-L}
T_L \leq C_d^\prime \ell e^{\mu_0} \beta_0^{\frac{d+1}{2}}
\end{equation}
then we shall have
\begin{equation}
\label{eq:s6-bound4}
\sup_{0 \leq t \leq T_L} \left\Vert\Phi_N (t)
\right\Vert_{\mathcal{L}^1_{\beta_0,\mu_0}}
\leq \left\Vert \Phi_N (0)
\right\Vert_{\mathcal{L}^1_{\frac{1}{2}\beta_0,(\mu_0-1)}}
\end{equation}
which implies by duality
\begin{equation}
\label{eq:s6-bound-F}
\sup_{0 \leq t \leq T_L}
\left| F_N (t) \right|_{L^\infty_{\frac{1}{2} \beta_0,
(\mu_0-1)}} \leq 
\left| F_N (0) \right|_{L^\infty_{\beta_0,\mu_0}}
\end{equation}
since the initial observable $\Phi_N (0)$ is arbitrary. 
Hence we obtain:
\begin{theorem}
\label{thm:s6-lanford}
Suppose $F_N (t) = \left\{ f_N^{(s)} (t)\right\}_{1\leq s \leq N}$
is a solution of the BBGKY hierarchy (\ref{eq:s4-def-bbgky}),
subject to the Boltzmann-Grad scaling $N \varepsilon^{d-1}=\ell^{-1}$,
and with each 
function $f_N^{(s)} (t,Z_s)$
 symmetric under particle interchange.
Further suppose that for some $\beta_0 > 0$, $\mu_0 \in \mathbb{R}$,
\begin{equation}
\label{eq:s6-bbgky1}
\sup_{1\leq s \leq N} \sup_{Z_s \in \mathcal{D}_s}
\left| f_N^{(s)} (0,Z_s)\right| 
e^{\beta_0 E_s (Z_s)} e^{\mu_0 s} \leq 1
\end{equation}
Then there is a constant $C_d>0$, depending only on $d$, such that if
 $T_L < C_d \ell e^{\mu_0} \beta_0^{\frac{d+1}{2}}$ then there holds
\begin{equation}
\label{eq:s6-bbgky2}
\sup_{0 \leq t \leq T_L} \sup_{1\leq s \leq N} \sup_{Z_s \in \mathcal{D}_s}
\left| f_N^{(s)} (t,Z_s)\right| 
e^{\frac{1}{2} \beta_0 E_s (Z_s)}
 e^{(\mu_0 - 1)s} \leq 1
\end{equation}
\end{theorem}
\begin{remark}
Theorem \ref{thm:s6-lanford} does not require the 
functions $f_N^{(s)}$ to be non-negative, nor does it require that
they form a consistent sequence of marginals.
\end{remark}

The bound
(\ref{eq:s6-bound-F}) is just the classical \emph{a priori} bound
of Lanford \cite{L1975,GSRT2014}; note that the same argument based on
observables would have worked in a periodic domain as well. Moreover, for
any fixed initial datum, the Lanford time $T_L$ increases in direct
proportion to the mean free path, which is $\mathcal{O}(\ell)$.

\section{Global \emph{a priori} bounds on observables}
\label{sec:7}

Our goal is to extend the \emph{a priori} bounds from the previous section
to the entire time interval, $t\in [0,\infty)$, as soon as the
mean free path $\mathcal{O}(\ell)$
 is sufficiently large. The relevant
estimates were first proved by Illner \& Pulvirenti
\cite{IP1989}, using the dispersive inequalities we have stated in
Lemmas \ref{lemma:s3-ip1}, \ref{lemma:s3-ip3}. Our approach is
slightly different, in that we will be working with the \emph{dual}
hierarchy. Note that once the correct weights are chosen,
the rest amounts to a computation, plus one application of
Lemma \ref{lemma:s3-ip1}.

Let us be given a time $T>0$,
and smooth increasing functions $\beta(t):[0,T]\rightarrow \mathbb{R}^+$,
$\mu(t):[0,T]\rightarrow\mathbb{R}$. The spaces 
$\mathcal{L}^1_{\beta,\mu}$, $L^\infty_{\beta,\mu}$ are as defined
in the previous section. We are given functions
$\Phi_N (t) = \left\{ \phi_N^{(s)} (t)\right\}_{1\leq s \leq N}$, with
each $\phi_N^{(s)}:[0,T]\times\mathcal{D}_s \rightarrow \mathbb{R}$ 
symmetric under particle
interchange, such that $\Phi_N$ satisfies the dual hierarchy
(\ref{eq:s5-def-dual-bbgky}-\ref{eq:s5-bc-dual-bbgky}) for
$t\in[0,T]$. Define the functions 
$\tilde{\Phi}_N (t) = \left\{ \tilde{\phi}_N^{(s)} (t)
\right\}_{1\leq s \leq N}$ by the formula
\begin{equation}
\label{eq:s7-phi-tilde}
\tilde{\phi}_N^{(s)} (t,Z_s) = 
\phi_N^{(s)} (t,Z_s) e^{-\beta (t) I_s ((X_s - (T-t)V_s,V_s))}
\end{equation}
We will be estimating $\left\Vert \tilde{\Phi}_N (t)
\right\Vert_{\mathcal{L}^1_{\beta(t),\mu(t)}}$ for $t\in[0,T]$.

Observe first that $\left( \partial_t - V_s\cdot \nabla_{X_s}\right)
I_s ( (X_s - (T-t)V_s,V_s)) = 0$ on any open subset of $\mathcal{D}_s$.
On the other hand, for $Z_s = (X_s,V_s)\in \mathcal{D}_s$ we have
\begin{equation}
\label{eq:s7-I-s-1}
I_s ((X_s - (T-t)V_s,V_s)) = I_s (Z_s) 
- (T-t) \mathcal{Y}_s (Z_s) + (T-t)^2 E_s (Z_s)
\end{equation}
Clearly if $Z_s \in \partial \mathcal{D}_s$ then 
$I_s (Z_s^*) = I_s (Z_s)$, and $E_s (Z_s^*) = E_s (Z_s)$.
Hence by Lemma \ref{lemma:s3-ip1},
\begin{equation}
\label{eq:s7-I-s-2}
\begin{aligned}
& I_s ( (X_s - (T-t)V_s,V_s)) \geq  I_s ((X_s - (T-t)V_s^*,V_s^*)) \\
& \qquad  \textnormal{ whenever } t \in [0,T]
\textnormal{ and } Z_s =(X_s,V_s) \in \partial \mathcal{D}_s
\textnormal{ is \emph{pre}-collisional}
\end{aligned}
\end{equation}
The restriction $t \leq T$ in (\ref{eq:s7-I-s-2}) is crucial; without
this restriction the inequality could go the \emph{wrong way} where we
need it in the proof.

On any open subset of $\mathcal{D}_s$ we have
\begin{equation}
\label{eq:s7-transport1}
\left( \frac{\partial}{\partial t} - V_s\cdot \nabla_{X_s}\right)
\left| \phi_N^{(s)} (t,Z_s) \right| = 0
\end{equation}
and likewise
\begin{equation}
\label{eq:s7-transport2}
 \left( \frac{\partial}{\partial t} - V_s \cdot \nabla_{X_s}\right)
I_s ( (X_s - (T-t)V_s,V_s))  = 0
\end{equation}
Therefore by the divergence theorem we obtain the \emph{equality}:
\begin{equation}
\label{eq:s7-equality1}
\begin{aligned}
&\frac{\partial}{\partial t} \int_{\mathcal{D}_s}
 \left| \tilde{\phi}_N^{(s)} (t,Z_s)\right| e^{-\beta(t) E_s (Z_s)}
 e^{-\mu(t) s} dZ_s = \\
& = \frac{1}{2} \sum_{1 \leq i \neq j \leq s}\int_{\mathbb{R}^{ds}}
\int_{\Sigma_s (i,j)} n^{i,j} \cdot V_s
\left| \phi_N^{(s)} (t,Z_s) \right| \times \\
& \qquad \qquad \times
e^{-\beta (t) \left[ I_s( ( X_s-(T-t)V_s,V_s)) + E_s(Z_s) \right]}
e^{-\mu (t) s} d\sigma^{i,j} dV_s + \\
& + \int_{\mathcal{D}_s} \left| \tilde{\phi}_N^{(s)} (t,Z_s)\right|
e^{-\beta(t) E_s (Z_s)} e^{-\mu(t) s} \times \\
& \qquad \qquad \times
\left\{ -\beta^\prime (t) \left[ I_s((X_s-(T-t)V_s,V_s)) 
+ E_s(Z_s) \right] - \mu^\prime (t) s \right\} dZ_s
\end{aligned}
\end{equation}

The boundary term can be re-written as an integral over 
\emph{pre}-collisional configurations. Recall that, according to our
conventions, $n^{i,j}\cdot V_s = - \frac{x_j-x_i}{\varepsilon \sqrt{2}}
\cdot (v_j-v_i)$ along $\Sigma_s (i,j)\times \mathbb{R}^{ds}$;
therefore, $n^{i,j}\cdot V_s > 0$ for pre-collisional
configurations. We have:
\begin{equation}
\label{eq:s7-equality2}
\begin{aligned}
&\frac{\partial}{\partial t} \int_{\mathcal{D}_s}
 \left| \tilde{\phi}_N^{(s)} (t,Z_s)\right| e^{-\beta(t) E_s (Z_s)}
 e^{-\mu(t) s} dZ_s = \\
& = \frac{1}{2} \sum_{1 \leq i \neq j \leq s}\int_{\mathbb{R}^{ds}}
\int_{\Sigma_s^{\textnormal{inc}} (i,j)} \left| n^{i,j} \cdot V_s \right|
\left| \phi_N^{(s)} (t,Z_s) \right| \times \\
& \qquad \qquad \times
e^{-\beta (t) \left[ I_s( ( X_s-(T-t)V_s,V_s)) + E_s(Z_s) \right]}
e^{-\mu (t) s} d\sigma^{i,j} dV_s  \\
& - \frac{1}{2} \sum_{1 \leq i \neq j \leq s}\int_{\mathbb{R}^{ds}}
\int_{\Sigma_s^{\textnormal{inc}} (i,j)} \left| n^{i,j} \cdot V_s \right|
\left| \phi_N^{(s)} (t,Z_s^*) \right| \times \\
& \qquad \qquad \times
e^{-\beta (t) \left[ I_s( ( X_s-(T-t)V_s^*,V_s^*)) + E_s(Z_s^*) \right]}
e^{-\mu (t) s} d\sigma^{i,j} dV_s  \\
& + \int_{\mathcal{D}_s} \left| \tilde{\phi}_N^{(s)} (t,Z_s)\right|
e^{-\beta(t) E_s (Z_s)} e^{-\mu(t) s} \times \\
& \qquad \qquad \times
\left\{ -\beta^\prime (t) \left[ I_s((X_s-(T-t)V_s,V_s)) 
+ E_s(Z_s) \right] - \mu^\prime (t) s \right\} dZ_s
\end{aligned}
\end{equation} 

According to the boundary condition (\ref{eq:s5-bc-dual-bbgky}), for
any $Z_s \in \partial \mathcal{D}_s$,
\begin{equation}
\label{eq:s7-bc}
\begin{aligned}
\left| \phi_N^{(s)} (t,Z_s) \right| \leq
& \left| \phi_N^{(s)} (t,Z_s^*) \right| + 
N \left| \phi_N^{(s-1)} (t,Z_s^{(i)}) \right| +
N \left| \phi_N^{(s-1)} (t,Z_s^{(j)}) \right| + \\
& + N \left| \phi_N^{(s-1)} (t,(Z_s^*)^{(i)}) \right| +
N \left| \phi_N^{(s-1)} (t,(Z_s^*)^{(j)}) \right|
\end{aligned}
\end{equation}

Therefore,
\begin{equation}
\label{eq:s7-bound1}
\begin{aligned}
&\frac{\partial}{\partial t} \int_{\mathcal{D}_s}
 \left| \tilde{\phi}_N^{(s)} (t,Z_s)\right| e^{-\beta(t) E_s (Z_s)}
 e^{-\mu(t) s} dZ_s \leq \\
& \leq \frac{1}{2} \sum_{1 \leq i \neq j \leq s}\int_{\mathbb{R}^{ds}}
\int_{\Sigma_s^{\textnormal{inc}} (i,j)} \left| n^{i,j} \cdot V_s \right|
\left| \phi_N^{(s)} (t,Z_s^*) \right| \times \\
& \qquad \qquad \qquad \times
e^{-\beta (t) \left[ I_s( ( X_s-(T-t)V_s,V_s)) + E_s(Z_s) \right]}
e^{-\mu (t) s} d\sigma^{i,j} dV_s  \\
& + N \sum_{1 \leq i \neq j \leq s}\int_{\mathbb{R}^{ds}}
\int_{\Sigma_s^{\textnormal{inc}} (i,j)} \left| n^{i,j} \cdot V_s \right|
\left| \phi_N^{(s-1)} (t,Z_s^{(i)}) \right| \times \\
& \qquad \qquad \qquad \times
e^{-\beta (t) \left[ I_s( ( X_s-(T-t)V_s,V_s)) + E_s(Z_s) \right]}
e^{-\mu (t) s} d\sigma^{i,j} dV_s  \\
& + N \sum_{1 \leq i \neq j \leq s}\int_{\mathbb{R}^{ds}}
\int_{\Sigma_s^{\textnormal{inc}} (i,j)} \left| n^{i,j} \cdot V_s \right|
\left| \phi_N^{(s-1)} (t,(Z_s^*)^{(i)}) \right| \times \\
& \qquad \qquad \qquad \times
e^{-\beta (t) \left[ I_s( ( X_s-(T-t)V_s,V_s)) + E_s(Z_s) \right]}
e^{-\mu (t) s} d\sigma^{i,j} dV_s  \\
& - \frac{1}{2} \sum_{1 \leq i \neq j \leq s}\int_{\mathbb{R}^{ds}}
\int_{\Sigma_s^{\textnormal{inc}} (i,j)} \left| n^{i,j} \cdot V_s \right|
\left| \phi_N^{(s)} (t,Z_s^*) \right| \times \\
& \qquad \qquad \times 
e^{-\beta (t) \left[ I_s( ( X_s-(T-t)V_s^*,V_s^*)) + E_s(Z_s^*) \right]}
e^{-\mu (t) s} d\sigma^{i,j} dV_s  \\
& + \int_{\mathcal{D}_s} \left| \tilde{\phi}_N^{(s)} (t,Z_s)\right|
e^{-\beta(t) E_s (Z_s)} e^{-\mu(t) s} \times \\
& \qquad \qquad \times
\left\{ -\beta^\prime (t) \left[ I_s((X_s-(T-t)V_s,V_s)) 
+ E_s(Z_s) \right] - \mu^\prime (t) s \right\} dZ_s
\end{aligned}
\end{equation}

We apply (\ref{eq:s7-I-s-2}) to the first and third terms on the right
hand side, for $0\leq t \leq T$.
\begin{equation}
\label{eq:s7-bound2}
\begin{aligned}
&\frac{\partial}{\partial t} \int_{\mathcal{D}_s}
 \left| \tilde{\phi}_N^{(s)} (t,Z_s)\right| e^{-\beta(t) E_s (Z_s)}
 e^{-\mu(t) s} dZ_s \leq \\
& \leq \frac{1}{2} \sum_{1 \leq i \neq j \leq s}\int_{\mathbb{R}^{ds}}
\int_{\Sigma_s^{\textnormal{inc}} (i,j)} \left| n^{i,j} \cdot V_s \right|
\left| \phi_N^{(s)} (t,Z_s^*) \right| \times \\
& \qquad \qquad \times
e^{-\beta (t) \left[ I_s( ( X_s-(T-t)V_s^*,V_s^*)) + E_s(Z_s^*) \right]}
e^{-\mu (t) s} d\sigma^{i,j} dV_s  \\
& +  N  \sum_{1 \leq i \neq j \leq s}\int_{\mathbb{R}^{ds}}
\int_{\Sigma_s^{\textnormal{inc}} (i,j)} \left| n^{i,j} \cdot V_s \right|
\left| \phi_N^{(s-1)} (t,Z_s^{(i)}) \right| \times \\
& \qquad \qquad \times
e^{-\beta (t) \left[ I_s( ( X_s-(T-t)V_s,V_s)) + E_s(Z_s) \right]}
e^{-\mu (t) s} d\sigma^{i,j} dV_s  \\
& + N  \sum_{1 \leq i \neq j \leq s}\int_{\mathbb{R}^{ds}}
\int_{\Sigma_s^{\textnormal{inc}} (i,j)} \left| n^{i,j} \cdot V_s \right|
\left| \phi_N^{(s-1)} (t,(Z_s^*)^{(i)}) \right| \times \\
& \qquad \qquad \times
e^{-\beta (t) \left[ I_s( ( X_s-(T-t)V_s^*,V_s^*)) + E_s(Z_s^*) \right]}
e^{-\mu (t) s} d\sigma^{i,j} dV_s  \\
& - \frac{1}{2} \sum_{1 \leq i \neq j \leq s}\int_{\mathbb{R}^{ds}}
\int_{\Sigma_s^{\textnormal{inc}} (i,j)} \left| n^{i,j} \cdot V_s \right|
\left| \phi_N^{(s)} (t,Z_s^*) \right| \times \\
& \qquad \qquad \times
e^{-\beta (t) \left[ I_s( ( X_s-(T-t)V_s^*,V_s^*)) + E_s(Z_s^*) \right]}
e^{-\mu (t) s} d\sigma^{i,j} dV_s  \\
& + \int_{\mathcal{D}_s} \left| \tilde{\phi}_N^{(s)} (t,Z_s)\right|
e^{-\beta(t) E_s (Z_s)} e^{-\mu(t) s}\times \\
& \qquad \qquad \times
\left\{ -\beta^\prime (t) \left[ I_s((X_s-(T-t)V_s,V_s)) 
+ E_s(Z_s) \right] - \mu^\prime (t) s \right\} dZ_s
\end{aligned}
\end{equation}

Now the first term precisely cancels the fourth term, whereas the
second and third terms combine to yield an integral over all
of $\Sigma_s (i,j)$.
\begin{equation}
\label{eq:s7-bound3}
\begin{aligned}
&\frac{\partial}{\partial t} \int_{\mathcal{D}_s}
 \left| \tilde{\phi}_N^{(s)} (t,Z_s)\right| e^{-\beta(t) E_s (Z_s)}
 e^{-\mu(t) s} dZ_s \leq \\
& \leq N  \sum_{1 \leq i \neq j \leq s}\int_{\mathbb{R}^{ds}}
\int_{\Sigma_s (i,j)} \left| n^{i,j} \cdot V_s \right|
\left| \phi_N^{(s-1)} (t,Z_s^{(i)}) \right| \times \\
& \qquad \qquad \times
e^{-\beta (t) \left[ I_s( ( X_s-(T-t)V_s,V_s)) + E_s(Z_s) \right]}
e^{-\mu (t) s} d\sigma^{i,j} dV_s  \\
& + \int_{\mathcal{D}_s} \left| \tilde{\phi}_N^{(s)} (t,Z_s)\right|
e^{-\beta(t) E_s (Z_s)} e^{-\mu(t) s} \times \\
& \qquad \qquad \times
\left\{ -\beta^\prime (t) \left[ I_s((X_s-(T-t)V_s,V_s)) 
+ E_s(Z_s) \right] - \mu^\prime (t) s \right\} dZ_s
\end{aligned}
\end{equation}

The following inequality is immediate and
 holds for all $Z_s \in \mathbb{R}^{2ds}$ and $t\in\mathbb{R}$:
\begin{equation}
\label{eq:s7-I-s-3}
\begin{aligned}
& I_s ( (X_s - (T-t) V_s,V_s)) + E_s (Z_s) \geq \\
&  \qquad \qquad \qquad 
\geq \frac{1}{2} \left( \left| x_i - (T-t)v_i\right|^2 + |v_i|^2\right)
+ E_{s-1} (Z_s^{(i)}) 
\end{aligned}
\end{equation}
We may eliminate $x_i$ from the right-hand side of
(\ref{eq:s7-I-s-3}) whenever
$Z_s \in \Sigma_s (i,j) \times \mathbb{R}^{ds}$, due to the
condition $x_j = x_i + \varepsilon \omega$. Combining this fact with
the Boltzmann-Grad scaling $N\varepsilon^{d-1}=\ell^{-1}$, we obtain
the following from (\ref{eq:s7-bound3}):
\begin{equation}
\label{eq:s7-bound4}
\begin{aligned}
&\frac{\partial}{\partial t} \int_{\mathcal{D}_s}
 \left| \tilde{\phi}_N^{(s)} (t,Z_s)\right| e^{-\beta(t) E_s (Z_s)}
 e^{-\mu(t) s} dZ_s \leq \\
& \leq \ell^{-1}  \sum_{i=1}^s 
\int_{ \mathbb{R}^{2d(s-1)}}
 \left| \tilde{\phi}_N^{(s-1)} (t,Z_s^{(i)}) \right|\times \\ 
& \qquad \times
\left[  \sum_{\substack{j=1\\ j \neq i}}^s
\int_{\mathbb{R}^d \times \mathbb{S}^{d-1}}
\left| \omega \cdot (v_j - v_i )\right|
e^{-\frac{1}{2} \beta (t) 
\left[ |x_j-\varepsilon \omega-(T-t)v_i|^2 + |v_i|^2 \right]} 
d\omega dv_i \right] \times \\
& \qquad \times
e^{-\beta (t) E_{s-1} (Z_s^{(i)})}
e^{-\mu (t) s} dZ_s^{(i)} \\
& + \int_{\mathcal{D}_s} \left| \tilde{\phi}_N^{(s)} (t,Z_s)\right|
e^{-\beta(t) E_s (Z_s)} e^{-\mu(t) s}\times \\
& \qquad \qquad \times
\left\{ -\beta^\prime (t) \left[ I_s((X_s-(T-t)V_s,V_s)) 
+ E_s(Z_s) \right] - \mu^\prime (t) s \right\} dZ_s
\end{aligned}
\end{equation}
The integral in brackets is controlled using the classical dispersive
inequality \cite{BD1985}:
\begin{equation}
\label{eq:s7-decay}
\left\Vert \zeta ( x-vt,v) \right\Vert_{L^\infty_x L^1_v} \leq
|t|^{-d} \left\Vert \zeta(x,v) \right\Vert_{L^1_x L^\infty_v}
\end{equation}
Hence,
\begin{equation}
\label{eq:s7-bound5}
\begin{aligned}
&\frac{\partial}{\partial t} \int_{\mathcal{D}_s}
 \left| \tilde{\phi}_N^{(s)} (t,Z_s)\right| e^{-\beta(t) E_s (Z_s)}
 e^{-\mu(t) s} dZ_s \leq \\
& \leq \ell^{-1} s
\int_{ \mathbb{R}^{2d(s-1)}}
 \left| \tilde{\phi}_N^{(s-1)} (t,Z_{s-1}) \right|\times \\ 
& \times \left[ C_d [1+(T-t)]^{-d} \beta(t)^{-\frac{d}{2}}
\left( (s-1)^{\frac{1}{2}} E_{s-1}(Z_{s-1})^{\frac{1}{2}}
+ (s-1) \beta(t)^{-\frac{1}{2}} \right) \right] \times \\
& \qquad \qquad\times e^{-\beta (t) E_{s-1} (Z_{s-1})}
e^{-\mu (t) s} dZ_{s-1} \\
& + \int_{\mathcal{D}_s} \left| \tilde{\phi}_N^{(s)} (t,Z_s)\right|
e^{-\beta(t) E_s (Z_s)} e^{-\mu(t) s} \times \\
& \qquad \qquad \times
\left\{ -\beta^\prime (t) \left[ I_s((X_s-(T-t)V_s,V_s)) 
+ E_s(Z_s) \right] - \mu^\prime (t) s \right\} dZ_s
\end{aligned}
\end{equation}

We can sum over $s$ to obtain, for $0\leq t \leq T$,
\begin{equation}
\label{eq:s7-bound6}
\begin{aligned}
&\frac{\partial}{\partial t} \left\Vert \tilde{\Phi}_N (t) 
\right\Vert_{\mathcal{L}^1_{\beta(t),\mu(t)}} \leq \\
& \leq \sum_{s=1}^{N-1} \frac{1}{s!}
\int_{\mathcal{D}_s} \left| \tilde{\phi}_N^{(s)} (t,Z_s)\right|
e^{-\beta (t) E_s (Z_s)} e^{-\mu (t) s} \times \\
& \times \left[ \frac{C_d  e^{-\mu (t)} \beta(t)^{- \frac{d}{2}} }
{ \ell \left[1+(T-t)\right]^d}
\left( s^{\frac{1}{2}} E_s (Z_s)^{\frac{1}{2}}
+ s \beta(t)^{-\frac{1}{2}} \right) -
\beta^\prime (t) E_s (Z_s) - \mu^\prime (t) s \right] dZ_s
\end{aligned}
\end{equation}
Suppose $\beta_0 > 0$ and $\mu_0 \in \mathbb{R}$ are given.
Then fixing any $T>0$ we define
\begin{equation}
\label{eq:s7-weight1}
\beta (t) = \beta_0 -
\frac{1}{2} \beta_0 \left(1- \left[ 1 + (T-t) \right]^{-(d-1)}\right)
\end{equation}
\begin{equation}
\label{eq:s7-weight2}
\mu (t) = \mu_0 -
 \left(1-\left[1+(T-t)\right]^{-(d-1)}\right)
\end{equation}
We have $\beta (T)=\beta_0$, $\mu(T)=\mu_0$,
$\inf_{0 \leq t \leq T}\beta (t) \geq \frac{1}{2}\beta_0$,
$\inf_{0\leq t \leq T} \mu(t) \geq ( \mu_0 - 1 )$, and
\begin{equation}
\label{eq:s7-weight3}
\beta^\prime (t) = 
\frac{1}{2} \beta_0 (d-1) [1+(T-t)]^{-d}
\end{equation}
\begin{equation}
\label{eq:s7-weight4}
\mu^\prime (t) = 
 (d-1) [1+(T-t)]^{-d}
\end{equation}
Then if we assume further that
$\ell^{-1} e^{-\mu_0} \beta_0^{-\frac{d+1}{2}}$ is sufficiently
small (depending only on $d$), then
\begin{equation}
\label{eq:s7-bound7}
\sup_{0\leq t \leq T} \left\Vert
\tilde{\Phi}_N (t) \right\Vert_{\mathcal{L}^1_{\beta(t),\mu(t)}}
\leq \left\Vert \tilde{\Phi}_N (0) 
\right\Vert_{\mathcal{L}^1_{\frac{1}{2}\beta_0,(\mu_0 - 1)}}
\end{equation}
In particular,
\begin{equation}
\label{eq:s7-bound8}
\left\Vert \tilde{\Phi}_N (T) \right\Vert_{\mathcal{L}^1_{\beta_0,\mu_0}}
\leq \left\Vert \tilde{\Phi}_N (0)
\right\Vert_{\mathcal{L}^1_{\frac{1}{2} \beta_0,(\mu_0 - 1)}}
\end{equation}
Since $T>0$ is arbitrary, recalling the definition of
$\tilde{\Phi}_N$ and using duality we obtain:
\begin{theorem} (Illner \& Pulvirenti 1989)
\label{thm:s7-ip}
Suppose $F_N (t) = \left\{ f_N^{(s)} (t)\right\}_{1\leq s \leq N}$
is a solution of the BBGKY hierarchy (\ref{eq:s4-def-bbgky}),
subject to the Boltzmann-Grad scaling $N \varepsilon^{d-1}=\ell^{-1}$,
and with each 
function $f_N^{(s)} : [0,\infty) \times \mathbb{R}^{2ds} 
\rightarrow \mathbb{R}$ symmetric under particle interchange.
Further suppose that for some $\beta_0 > 0$, $\mu_0 \in \mathbb{R}$,
\begin{equation}
\label{eq:s7-bbgky1}
\sup_{1\leq s \leq N} \sup_{Z_s \in \mathcal{D}_s}
\left| f_N^{(s)} (0,Z_s)\right| 
e^{\beta_0 \left[ E_s (Z_s) + I_s (Z_s)\right]}
 e^{\mu_0 s} \leq 1
\end{equation}
Then if $\ell^{-1} e^{-\mu_0} \beta_0^{-\frac{d+1}{2}}$ is sufficiently
small (depending only on $d$) then we have
\begin{equation}
\label{eq:s7-bbgky2}
\sup_{t \geq 0} \sup_{1\leq s \leq N} \sup_{Z_s \in \mathcal{D}_s}
\left| f_N^{(s)} (t,Z_s)\right| 
e^{\frac{1}{2} \beta_0
 \left[ E_s (Z_s) + I_s ((X_s - V_s t,V_s))\right]}
 e^{(\mu_0 -1)s} \leq 1
\end{equation}
\end{theorem}
\begin{remark}
Theorem \ref{thm:s7-ip} does not require the 
functions $f_N^{(s)}$ to be non-negative, nor does it require that
they form a consistent sequence of marginals.
\end{remark}

\section{Representation of marginals via pseudo-trajectories}
\label{sec:8}

We recall that any solution $f_N^{(s)}(t)$ of the BBGKY hierarchy may
be decomposed in terms of the initial data propagated along 
``pseudo-trajectories.'' This technique is first due to Lanford,
and is now a standard tool in the analysis
of the Boltzmann-Grad limit for hard spheres.
To begin, observe that if 
$\left\{ f_N^{(s)} (t,Z_s) \right\}_{1 \leq s \leq N}$
solves (\ref{eq:s4-def-bbgky}), then by considering test functions
 which vanish along $[0,\infty)\times\partial \mathcal{D}_s$,
it follows that the densities $f_N^{(s)}$ solve the following equation
in the sense of distributions:
\begin{equation}
\label{eq:s8-bbgky1}
\left( \frac{\partial}{\partial t} + V_s \cdot \nabla_{X_s} \right)
f_N^{(s)} (t,Z_s) = (N-s) \varepsilon^{d-1} C_{s+1}
 f_N^{(s+1)} (t,Z_s)
\end{equation}
where $f_N^{(s)} (t,Z_s) = f_N^{(s)} (t,Z_s^*)$ a.e. 
$(t,Z_s) \in [0,\infty)\times\partial \mathcal{D}_s$, and $C_{s+1}$ 
is the collision operator
\begin{equation}
\label{eq:s8-coll1}
C_{s+1} = \sum_{i=1}^s C_{i,s+1}
\end{equation}
\begin{equation}
\label{eq:s8-coll2}
C_{i,s+1} = 
C_{i,s+1}^+ - C_{i,s+1}^-
\end{equation}
\begin{equation}
\label{eq:s8-coll3}
\begin{aligned}
& C_{i,s+1}^+ f_N^{(s+1)} (t,Z_s)  =\int_{\mathbb{R}^d}
\int_{\mathbb{S}^{d-1}}
\mathbf{1}_{Z_{s+1} \in \partial \mathcal{D}_{s+1}}
\left[\omega \cdot (v_{s+1}-v_i)\right]_{+} \times \\
& \;\; \qquad \times f_N^{(s+1)} (t,x_1,v_1,\dots,x_i,v_i^*,\dots,
x_s,v_s,x_i + \varepsilon \omega, v_{s+1}^*) d\omega dv_{s+1}
\end{aligned}
\end{equation}
\begin{equation}
\label{eq:s8-coll4}
\begin{aligned}
& C_{i,s+1}^- f_N^{(s+1)} (t,Z_s)  = \int_{\mathbb{R}^d}
\int_{\mathbb{S}^{d-1}}
\mathbf{1}_{Z_{s+1} \in \partial \mathcal{D}_{s+1}}
\left[\omega \cdot (v_{s+1}-v_i)\right]_{-} \times \\
& \;\; \qquad \times f_N^{(s+1)} (t,x_1,v_1,\dots,x_i,v_i,\dots,
x_s,v_s,x_i + \varepsilon \omega, v_{s+1}) d\omega dv_{s+1}
\end{aligned}
\end{equation}
where
\begin{equation}
\label{eq:s8-vstar}
\begin{cases}
v_i^* = v_i + \omega \omega \cdot \left( v_j-v_i\right)\\
v_j^* = v_j - \omega \omega \cdot \left( v_j-v_i\right) 
\end{cases}
\end{equation}

We can re-write (\ref{eq:s8-bbgky1}) by means of Duhamel's
formula, using the transport operator $T_s (t)$ defined by
$\left( T_s (t) g_s \right) (Z_s) = g_s (\psi_s^{-t} Z_s)$
for any $g_s \in L^1 (\mathcal{D}_s)$. The operators 
$T_s (t)$ form a strongly continuous semigroup on
$L^1 (\mathcal{D}_s)$, with generator given by
$-V_s \cdot \nabla_{X_s}$ and specular
reflection at the boundary $\partial \mathcal{D}_s$. We have
\begin{equation}
\label{eq:s8-bbgky2}
f_N^{(s)} (t) = T_s (t) f_N^{(s)} (0) +
(N-s)\varepsilon^{d-1} \int_0^t T_s (t-t_1)
 C_{s+1} f_N^{(s+1)} (t_1) dt_1
\end{equation}
Now by iterating this formula we can write the marginal $f_N^{(s)}(t)$
as a \emph{finite} sum of terms, each of which depends only
on the initial data:
\begin{equation}
\label{eq:s8-series1}
\begin{aligned}
& f_N^{(s)} (t) = \sum_{k=0}^{N-s} a_{N,k,s}\times \\
& \qquad \;\;
\times \int_0^t \int_0^{t_1} \dots \int_0^{t_{k-1}} T_s (t-t_1) 
C_{s+1}\dots T_{s+k} (t_k) f_N^{(s+k)} (0) dt_k \dots dt_1
\end{aligned}
\end{equation}
where
\begin{equation}
\label{eq:s8-a-N-k-s}
a_{N,k,s} = \frac{(N-s)!}{(N-s-k)!} \varepsilon^{k (d-1)}
\end{equation}
Since we enforce the Boltzmann-Grad scaling 
$N \varepsilon^{d-1} = \ell^{-1}$,
we have $0 \leq a_{N,k,s} \leq \ell^{-k}$ and
 $a_{N,k,s} \ell^k \rightarrow 1$ as $N \rightarrow \infty$ 
with $k,s$ fixed.

The Duhamel series (\ref{eq:s8-series1}) may be
interpreted as a way of describing the solution $F_N (t)$ in
terms of the data $F_N (0)$ propagated along a family of artificial
trajectories, or ``pseudo-trajectories.''
\cite{L1975,GSRT2014,PSS2014} 
Given $Z_s \in \mathcal{D}_s$, along with
times $0\leq t_k \leq \dots \leq t_1 \leq t$,
velocities $v_{s+1},\dots,v_{s+k}$, impact parameters
$\omega_1,\dots,\omega_k$, and indices $i_1,\dots,i_k$, we will define
\begin{equation}
\label{eq:s8-Z1}
Z_{s,s+k} \left[ Z_s , t ; t_1,\dots,t_k;v_{s+1},\dots,v_{s+k};
 \omega_1,\dots,\omega_k; i_1,\dots,i_k \right] \in
\mathcal{D}_{s+k}
\end{equation}
We assume $i_1 \in \left\{ 1,\dots,s\right\}$,
$i_2 \in \left\{ 1,\dots,s,s+1\right\}$, \dots, 
$i_j \in \left\{ 1,2,\dots,s+j-1\right\}$; we
will also need to assume that certain ``exclusion conditions''
are satisfied, as will become clear. To begin the induction,
for $Z_s \in \mathcal{D}_s$ and $t > 0$ we define
\begin{equation}
\label{eq:s8-Z2}
Z_{s,s} \left[ Z_s,t\right] = \psi_s^{-t} Z_s
\end{equation} 
More generally, if the symbol
\begin{equation}
\label{eq:s8-Z3}
Z_{s,s+k} \left[ Z_s , t ; t_1,\dots,t_k; v_{s+1},\dots,v_{s+k};
\omega_1,\dots,\omega_k;i_1,\dots,i_k \right] \in
\mathcal{D}_{s+k}
\end{equation}
is defined, then for $\tau > 0$ we define
\begin{equation}
\label{eq:s8-Z4}
\begin{aligned}
& Z_{s,s+k} \left[ Z_s , t+\tau ; t_1+\tau,\dots,t_k+\tau;
v_{s+1},\dots,v_{s+k};\omega_1,\dots,\omega_k;
i_1,\dots,i_k \right] = \\ 
& \qquad \;\;  = \psi_{s+k}^{-\tau} Z_{s,s+k} \left[ Z_s,t;
t_1,\dots,t_k;v_{s+1},\dots,v_{s+k};\omega_1,\dots,\omega_k;
i_1,\dots,i_k\right]
\end{aligned}
\end{equation}
Similarly, if the symbol
\begin{equation}
\label{eq:s8-Z5}
\begin{aligned}
& Z_{s,s+k} \left[ Z_s , t ; t_1,\dots,t_k; v_{s+1},\dots,v_{s+k};
\omega_1,\dots,\omega_k;i_1,\dots,i_k \right] = \\
& \qquad \qquad \qquad \qquad \qquad \qquad
\qquad \qquad \qquad \;\;  
 =\left( X_{s+k}^\prime,V_{s+k}^\prime \right) \in \mathcal{D}_{s+k}
\end{aligned}
\end{equation}
is defined (including the possibility $k=0$) then for
any given velocity $v_{s+k+1} \in \mathbb{R}^d$, any index
$i_{k+1}\in \left\{ 1,\dots,s,s+1,\dots,s+k\right\}$,
and any ``suitable'' choice of impact parameter
$\omega_{k+1}\in \mathbb{S}^{d-1}$, if 
$\omega_{k+1}\cdot\left(v_{s+k+1}-v_{i_{k+1}}^\prime\right)\leq 0$
then we define
\begin{equation}
\label{eq:s8-Z6}
\begin{aligned}
& Z_{s,s+k+1} \left[ Z_s , t ; t_1,\dots,t_k,0; 
v_{s+1},\dots,v_{s+k},v_{s+k+1};\right.\\
& \qquad \qquad \qquad \qquad  \qquad \qquad\left.
\omega_1,\dots,\omega_k,\omega_{k+1}; i_1,\dots,i_k,i_{k+1} \right] =\\
& \qquad =\left( 
x_1^\prime,v_1^\prime,\dots,x_{i_{k+1}}^\prime,v_{i_{k+1}}^\prime,
\dots,x_s^\prime,v_s^\prime,
x_{i_{k+1}}^\prime+\varepsilon \omega_{k+1},v_{s+k+1}\right)
\end{aligned}
\end{equation}
whereas if 
$\omega_{k+1}\cdot\left(v_{s+k+1}-v_{i_{k+1}}^\prime\right)>0$ then
we define
\begin{equation}
\label{eq:s8-Z7}
\begin{aligned}
& Z_{s,s+k+1} \left[ Z_s , t ; t_1,\dots,t_k,0; 
v_{s+1},\dots,v_{s+k},v_{s+k+1}; \right. \\
& \qquad \qquad \qquad \qquad \qquad \qquad \left.
\omega_1,\dots,\omega_k,\omega_{k+1}; i_1,\dots,i_k,i_{k+1} \right] =\\
& \qquad =\left( 
x_1^\prime,v_1^\prime,\dots,x_{i_{k+1}}^\prime,
v_{i_{k+1}}^\prime+
\omega_{k+1}\omega_{k+1}\cdot\left(v_{s+k+1}-v_{i_{k+1}}^\prime\right),
\right.\\
& \qquad \qquad \left. \dots,x_s^\prime,v_s^\prime,
x_{i_{k+1}}^\prime+\varepsilon \omega_{k+1},
v_{s+k+1}-
\omega_{k+1}\omega_{k+1}\cdot\left(v_{s+k+1}-v_{i_{k+1}}^\prime\right)\right)
\end{aligned}
\end{equation}
Here a ``suitable'' impact parameter $\omega$ is one
for which $\left| x_{i_{k+1}}^\prime + \varepsilon \omega
- x_j^\prime \right| > \varepsilon$ for each
$j \in \left\{ 1,\dots,s,s+1,\dots,s+k\right\} \backslash
\left\{ i_{k+1} \right\}$; note that
the set of suitable impact parameters may be empty.

\begin{remark}
The physical interpretation of the above construction is that $s$
particles begin in configuration $Z_s \in \mathcal{D}_s$ at time
$t$, then evolve under the \emph{backwards} hard sphere flow for a time
$t-t_1$; at time $t_1$, the $(s+1)$st particle is created adjacent to the
$i_1$st particle with velocity $v_{s+1}$. If the pair
$\left(i_1,s+1\right)$ is in a post-collisional configuration,
then we perform an instantaneous collision to place the particles
in a pre-collisional configuration. To continue the flow,
we push the system through the backwards flow of $(s+1)$ hard spheres for
a time $t_1-t_2$, and so forth. The state of
the process at time $0$ is then $Z_{s,s+k} \left[ Z_s,t;\left\{
t_j,v_{s+j},\omega_j,i_j\right\}_{j=1}^k \right]$.
\end{remark}

\begin{remark}
As a matter of convenience, we have enforced a convention whereby
particles are always in a pre-collisional configuration at the moment
that a new particle is created. Keep in mind, however, that
the backwards flow can subsequently place particles into
a post-collisional configuration, though this can only happen between
particle creations.
\end{remark}

We will also require an iterated collision kernel
\begin{equation}
\label{eq:s8-b1}
b_{s,s+k} \left[ Z_s,t;\left\{t_j,v_{s+j},\omega_j,i_j
\right\}_{j=1}^k \right]
\end{equation}
in order to account for
each added particle. First we define
\begin{equation}
\label{eq:s8-b2}
b_{s,s} \left[ Z_s,t\right] = \mathbf{1}_{Z_s \in \mathcal{D}_s}
\end{equation}
If we have defined
\begin{equation}
\label{eq:s8-b3}
b_{s,s+k} \left[ Z_s,t;t_1,\dots,t_k;
v_{s+1},\dots,v_{s+k};\omega_1,\dots,\omega_k;i_1,\dots,i_k\right] 
\end{equation}
then there are two cases: (i)
 $Z_{s,s+k}\left[Z_s,t;\left\{t_j,v_{s+j},\omega_j,i_j\right\}_{j=1}^k
\right] = \left( X_{s+k}^\prime,V_{s+k}^\prime \right)
\in \mathcal{D}_{s+k}$ is well-defined as above, in which
case
\begin{equation}
\label{eq:s8-b4}
\begin{aligned}
& b_{s,s+k} \left[
Z_s,t+\tau;t_1+\tau,\dots,t_k+\tau;v_{s+1},\dots,v_{s+k};
\omega_1,\dots,\omega_k;i_1,\dots,i_k\right] = \\
& \qquad \qquad \qquad
= b_{s,s+k} \left[ Z_s,t;t_1,\dots,t_k;v_{s+1},\dots,v_{s+k};
\omega_1,\dots,\omega_k;i_1,\dots,i_k\right]
\end{aligned}
\end{equation} 
\begin{equation}
\label{eq:s8-b5}
\begin{aligned}
& b_{s,s+k+1} \left[
Z_s,t;t_1,\dots,t_k,0; v_{s+1},\dots,v_{s+k},v_{s+k+1};\right. \\
& \qquad \qquad \qquad \qquad \qquad \qquad\left.
\omega_1,\dots,\omega_k,\omega_{k+1};
i_1,\dots,i_k,i_{k+1}\right] = \\
& = \omega_{k+1}\cdot \left( v_{s+k+1}-v_{i_{k+1}}^\prime\right)\times\\
&\qquad \qquad \times 
\left( \prod_{j\in \left\{ 1,\dots,s,s+1,\dots,s+k\right\}\backslash
\left\{i_{k+1}\right\}}
\mathbf{1}_{\left| x_{i_{k+1}}^\prime + \varepsilon \omega_{k+1}
-x_j^\prime\right| > \varepsilon} \right) \times \\
& \qquad \qquad \times b_{s,s+k} \left[ Z_s,t;t_1,\dots,t_k;
v_{s+1},\dots,v_{s+k};\omega_1,\dots,\omega_k;i_1,\dots,i_k\right]
\end{aligned}
\end{equation} 
(ii) otherwise,
\begin{equation}
\label{eq:s8-b6}
\begin{aligned}
& b_{s,s+k} \left[
Z_s,t+\tau;t_1+\tau,\dots,t_k+\tau;v_{s+1},\dots,v_{s+k};
\omega_1,\dots,\omega_k;i_1,\dots,i_k\right] = \\
& \qquad \;\; 
= b_{s,s+k} \left[ Z_s,t;t_1,\dots,t_k;v_{s+1},\dots,v_{s+k};
\omega_1,\dots,\omega_k;i_1,\dots,i_k\right] \;\; \left( = 0 \right)
\end{aligned}
\end{equation} 
\begin{equation}
\label{eq:s8-b7}
\begin{aligned}
& b_{s,s+k+1} \left[
Z_s,t;t_1,\dots,t_k,0;v_{s+1},\dots,v_{s+k},v_{s+k+1};\right.\\
&\qquad\qquad\qquad\qquad\qquad\qquad
 \left. \omega_1,\dots,\omega_k,\omega_{k+1};
i_1,\dots,i_k,i_{k+1}\right] = 0
\end{aligned}
\end{equation} 
Then the Duhamel series (\ref{eq:s8-series1}) becomes
\begin{equation}
\label{eq:s8-series2}
\begin{aligned}
& f_N^{(s)} (t,Z_s) = \sum_{k=0}^{N-s} a_{N,k,s}\times \\
& \times \sum_{i_1 = 1}^s \dots \sum_{i_k = 1}^{s+k-1}
\int_0^t \dots \int_0^{t_{k-1}}
\int_{\mathbb{R}^{dk}} \int_{\left(\mathbb{S}^{d-1}\right)^k} 
\left( \prod_{m=1}^k d\omega_m dv_{s+m} dt_m \right) \times \\
& \times \left( b_{s,s+k} \left[\cdot\right] f_N^{(s+k)} \left(0,
Z_{s,s+k}\left[\cdot\right]\right)\right)
\left[ Z_s,t;\left\{t_j,v_{s+j},\omega_j,i_j\right\}_{j=1}^k\right]
\end{aligned}
\end{equation}
\begin{remark}
The collision kernel $b_{s,s+k}\left[\dots\right]$ vanishes automatically
whenever $Z_{s,s+k}\left[\dots\right]$ fails to be well-defined. This
convention is convenient because it allows us to specify a fixed
$N$-independent domain of integration in (\ref{eq:s8-series2}).
\end{remark}

\section{Stability of pseudo-trajectories}
\label{sec:9}

The purpose of this section is to prove that typical pseudo-trajectories
are stable with respect to the creation of a new particle, at least
outside a small set of creation times, velocities, and impact parameters.
The main novelty of this stability
result, compared to previous results
\cite{GSRT2014}, is that we are able to allow particles to pass arbitrarily
close to each other in \emph{space} under the backwards flow, as long as
they do not collide. The price we pay for this improvement is that we must
make explicit use of the time integrals appearing in the Duhamel series 
(\ref{eq:s8-series2}), and employ an unusual cut-off for nearby velocities.
This proof is inspired in part by
the ideas from \cite{PSS2014}; note, however, that 
there the authors required more sophisticated
cut-offs to deal with rather general physical interactions.
 
We will require the following elementary geometrical fact
(the proof is trivial):
\begin{lemma}
\label{lemma:s9-sphere}
Fix $v\in \mathbb{R}^d\backslash \left\{ 0 \right\}$, and for
$\omega \in \mathbb{S}^{d-1}\subset \mathbb{R}^d$ (where $\mathbb{S}^{d-1}$
is the unit sphere centered on the origin) define
\begin{equation}
\label{eq:s9-u-omega}
u_\omega = |v|^{-1} \left( 2 \omega \omega \cdot v - v\right)
\end{equation}
then $u_\omega \in \mathbb{S}^{d-1}$ for each $\omega \in\mathbb{S}^{d-1}$.
If $\mathbb{S}^{d-1}_v = \left\{ \omega\in\mathbb{S}^{d-1}\left|
\;\omega\cdot v > 0\right.\right\}$ then the map
$\omega\mapsto u_\omega$ restricts to a diffeomorphism
$\mathbb{S}^{d-1}_v \rightarrow \mathbb{S}^{d-1}\backslash
\left\{-|v|^{-1} v\right\}$.
\end{lemma}

We will also need:
\begin{lemma}
\label{lemma:s9-cylinder}
Let $L\subset \mathbb{R}^d$ ($d\geq 2$)
 be a line, and for $\rho > 0$ consider the solid cylinder 
\begin{equation}
\label{eq:s9-K-rho}
\mathcal{C}_\rho = \left\{ \left. u\in\mathbb{R}^d \right|
\textnormal{dist}\left(u,L\right) \leq \rho \right\}
\end{equation}
 Then
\begin{equation}
\label{eq:s9-measure1}
\int_{\mathbb{S}^{d-1}} \mathbf{1}_{\omega \in \mathcal{C}_\rho} d\omega
\leq C_d \rho^{(d-1)/2}
\end{equation}
where the constant $C_d$ does not depend on the choice of line $L$.
\end{lemma}

\begin{proof}
There are two cases: either $\mathcal{C}_\rho$ contains
a point which is within distance $1-3\rho$ of the sphere's center, or it does
not. In the first case, we clearly have
\begin{equation*}
\int_{\mathbb{S}^{d-1}} \mathbf{1}_{\omega \in \mathcal{C}_\rho}
d\omega \leq C_d \rho^{d-\frac{3}{2}}
\end{equation*}
In the second case, we can estimate the size of a spherical cap to
obtain that
\begin{equation*}
\int_{\mathbb{S}^{d-1}} \mathbf{1}_{\omega \in \mathcal{C}_\rho}
d\omega \leq C_d \rho^{(d-1)/2}
\end{equation*}
Since $d\geq 2$, we can take the maximum of these two bounds to
obtain (\ref{eq:s9-measure1}).
\qed
\end{proof}

We now turn to the main result for this section. To state
the proposition, we must fix a parameter $\eta > 0$ and introduce
the following sets:
\begin{equation}
\label{eq:s9-K-s}
\mathcal{K}_s = \left\{ Z_s = \left(X_s,V_s\right)
 \in \overline{\mathcal{D}_s} \left|
\psi_s^{-\tau} Z_s = \left( X_s-V_s \tau,V_s \right) \; \forall
\; \tau > 0 \right. \right\}
\end{equation}
\begin{equation}
\label{eq:s9-U-s-eta}
\mathcal{U}_s^\eta = \left\{ Z_s = \left(X_s,V_s\right)\in
\overline{\mathcal{D}_s} \left|
\inf_{1\leq i < j \leq s} \left| v_i - v_j \right| > \eta \right. \right\}
\end{equation}

\begin{remark}
The condition $Z_s \in \mathcal{U}_s^\eta$ is meant to force particles to
disperse away from each other under the action of the free flow.
\end{remark}

\begin{proposition}
\label{prop:s9-stability}
There is a constant $c_d > 0$, depending only on the dimension $d$,
such that all the following holds:
Assume that
\begin{equation}
\label{eq:s9-Z1}
\begin{aligned}
& Z_{s,s+k} \left[ Z_s,t;t_1,\dots,t_k;v_{s+1},\dots,v_{s+k};
\omega_1,\dots,\omega_k;i_1,\dots,i_k
\right] = \\
& \qquad \qquad \qquad \qquad \qquad \qquad \qquad \qquad 
= \left(X_{s+k}^\prime,V_{s+k}^\prime\right)
\in \mathcal{K}_{s+k}\cap \mathcal{U}_{s+k}^\eta
\end{aligned}
\end{equation}
and $E_{s+k} \left( Z_{s+k}^\prime \right) \leq 2 R^2$; then, \\
(i)  for all $\tau \geq 0$ we have
\begin{equation}
\label{eq:s9-Z2}
\begin{aligned}
& Z_{s,s+k} \left[Z_s,t+\tau;t_1+\tau,\dots,t_k+\tau;v_{s+1},\dots,v_{s+k};
\omega_1,\dots,\omega_k;i_1,\dots,i_k
\right] \\
& \qquad \qquad \qquad \qquad \qquad \qquad \qquad \qquad \qquad \qquad
\qquad  \;\; \in \mathcal{K}_{s+k}\cap \mathcal{U}_{s+k}^\eta
\end{aligned}
\end{equation}
(ii) for any $i_{k+1}\in\left\{1,\dots,s,s+1,\dots,s+k\right\}$,
and for any $\theta,\alpha,y > 0$ such that 
$\sin \theta > c_d y^{-1} \varepsilon$, there exists a measurable set 
$\mathcal{B} \subset [0,\infty)\times \mathbb{R}^d \times \mathbb{S}^{d-1}$,
which may depend on $Z_s$, $t$, and
$\left\{ t_j,v_{s+j},\omega_j,i_j\right\}_{j=1}^k$, such that
\begin{equation}
\label{eq:s9-measure2}
\begin{aligned}
& \forall \;\; \eta \;< \;R,\;\; \forall \;\; T \; > \; 0,\\
& \int_0^T \int_{B_{2R}^d} \int_{\mathbb{S}^{d-1}} 
 \mathbf{1}_{\left(\tau,v_{s+k+1},\omega_{k+1}\right) \in \mathcal{B}} 
\; d\omega_{k+1} dv_{s+k+1} d\tau \leq \\
& \qquad \qquad \qquad 
\leq C_d \left( s+k \right) T R^d \left[ \alpha +  \frac{y}{\eta T} +
 C_{d,\alpha} \left(\frac{\eta}{R}\right)^{d-1} +
 C_{d,\alpha}\theta^{(d-1)/2}\right]
\end{aligned}
\end{equation}
and
\begin{equation}
\label{eq:s9-Z3}
\begin{aligned}
Z_{s,s+k+1} & \left[ Z_s,t+\tau; t_1+\tau,\dots,t_k+\tau,0;
v_{s+1},\dots,v_{s+k},v_{s+k+1};\right.\\
& \qquad \qquad \qquad \qquad \qquad \qquad \left.
\omega_1,\dots,\omega_k,\omega_{k+1}; i_1,\dots,i_k,i_{k+1}\right] \\
& \qquad \qquad \qquad \qquad \qquad \qquad \qquad
 \qquad \qquad 
 \in \mathcal{K}_{s+k+1}\cap \mathcal{U}_{s+k+1}^\eta
\end{aligned}
\end{equation}
whenever
$\left( \tau,v_{s+k+1},\omega_{k+1}\right) \in
 \left[0,\infty\right)\times \mathbb{R}^d
\times\mathbb{S}^{d-1} \backslash\mathcal{B}$.
\end{proposition}
\begin{remark}
The important conclusion from (\ref{eq:s9-measure2}) is that
$\mathcal{B}$ is a  set of small measure; on the complement of this small-measure
set, the inductive hypothesis (``we are in $\mathcal{K}_{s}\cap \mathcal{U}_s^\eta$'')
 is
propagated due to (\ref{eq:s9-Z3}). To see why $\mathcal{B}$ is of small measure,
assume that $R$ is a large
velocity cut-off, either constant or diverging very gently as
$\varepsilon \rightarrow 0^+$. The parameter $\eta > 0$ represents the
minimal velocity between particles and therefore we will
always have $\eta \ll R$. Similarly $y$ is a minimal distance between
particles at any time of particle creation; since particles are moving
relatively with speed at least $\eta$, we will eventually require
$y \ll \eta T$ so that particles are only nearby for a short time.
The angle $\alpha$ is a technical cutoff on near-grazing collisions and
is therefore small. The small angle $\theta$ is the opening angle of a cone 
inside of which particles may ``recollide'' (recollisions are the
geometric mechanism by which correlations are generated). The purely
geometric condition $\sin \theta > c_d y^{-1} \varepsilon$ forces
particles to be widely separated (compared to their diameter) at the
time of particle creation.
\end{remark}
\begin{proof}
Claim \emph{(i)} is trivial. For claim \emph{(ii)}, we distinguish
between two possibilities for the added particle: either
$\left(\tau,v_{s+k+1},\omega_{k+1}\right)$ is such that
$\omega_{k+1}\cdot \left(v_{s+k+1}-v_{i_{k+1}}^\prime\right) > 0$,
or else $\omega_{k+1}\cdot \left(v_{s+k+1}-v_{i_{k+1}}^\prime\right)\leq 0$.
We introduce two sets,
\begin{equation}
\label{eq:s9-A-plus}
\mathcal{A}^+ =\left\{
\begin{aligned}
& \left(\tau,v_{s+k+1},\omega_{k+1}\right)\subset
[0,\infty)\times\mathbb{R}^d \times \mathbb{S}^{d-1}
\textnormal{ such that }\\
& \qquad \qquad \qquad \omega_{k+1} \cdot \left(v_{s+k+1}-v_{i_{k+1}}^\prime
\right)>0
\end{aligned}
\right\}
\end{equation}
\begin{equation}
\label{eq:s9-A-minus}
\mathcal{A}^- =\left\{
\begin{aligned}
& \left(\tau,v_{s+k+1},\omega_{k+1}\right)\subset
[0,\infty)\times\mathbb{R}^d\times\mathbb{S}^{d-1} \textnormal{ such that } \\
& \qquad \qquad \qquad
\omega_{k+1} \cdot \left(v_{s+k+1}-v_{i_{k+1}}^\prime\right)
\leq 0
\end{aligned}
\right\}
\end{equation}
then we write $\mathcal{B}=\mathcal{B}^+ \cup \mathcal{B}^-$ where
$\mathcal{B}^+ \subset \mathcal{A}^+$ and
$\mathcal{B}^- \subset \mathcal{A}^-$.

\emph{Construction of $\mathcal{B}^-$.} 
We first eliminate creation times $\tau$ which could result in
spatial concentrations of particles. This is where we use the
property that $Z_{s+k}^\prime \in \mathcal{U}_{s+k}^\eta$, since this
condition guarantees that two particles can only be close to each
other for a short time (as long as the $(s+k)$ particles evolve under 
the free flow). We introduce the set
\begin{equation}
\label{eq:s9-B-I-minus}
\mathcal{B}^-_I = \left\{
\begin{aligned}
& \left(\tau,v_{s+k+1},\omega_{k+1}\right)\in
\mathcal{A}^- \textnormal{ such that } \\
& \inf_{i\in \left\{1,\dots,s,s+1,\dots,s+k\right\}
\backslash \left\{i_{k+1}\right\}} \left|
\left(x_{i_{k+1}}^\prime - x_i^\prime\right) - 
\tau \left(v_{i_{k+1}}^\prime - v_i^\prime\right)\right| \leq y  
\end{aligned}
\right\}
\end{equation}
then we have
\begin{equation}
\label{eq:s9-measure3}
\begin{aligned}
\int_0^T \int_{B_{2R}^d} \int_{\mathbb{S}^{d-1}} 
 \mathbf{1}_{\left(\tau,v_{s+k+1},\omega_{k+1}\right) \in \mathcal{B}^-_I}
& d\omega_{k+1} dv_{s+k+1} d\tau \leq \\
& \qquad \leq C_d \left(s+k-1\right) R^d \eta^{-1} y
\end{aligned}
\end{equation}
As a technical matter, we must also guarantee that the $(s+k+1)$-particle
state lives in $\mathcal{U}_{s+k+1}^\eta$ at the time of particle
creation. Hence, we will define
\begin{equation}
\label{eq:s9-B-II-minus}
\mathcal{B}^-_{II} = \left\{ \left(\tau,v_{s+k+1},\omega_{k+1}\right)\in
\mathcal{A}^- \left| \inf_{1\leq i \leq s+k} \left|
v_{s+k+1} - v_i^\prime \right| \leq \eta
\right. \right\}
\end{equation}
then we have
\begin{equation}
\label{eq:s9-measure4}
\int_0^T \int_{B_{2R}^d} \int_{\mathbb{S}^{d-1}} 
 \mathbf{1}_{\left(\tau,v_{s+k+1},\omega_{k+1}\right) \in 
\mathcal{B}^-_{II}} \; d\omega_{k+1} dv_{s+k+1} d\tau \leq
C_d \left( s+k\right) T \eta^d
\end{equation}

Lastly, we will guarantee (with high probability) 
that the created particle does not ``recollide''
under the backwards flow; that is, the $(s+k+1)$-particle state must
live in $\mathcal{K}_{s+k+1}$ at the time of particle creation.
To this end, for $i \in \left\{ 1,\dots,s,s+1,\dots,s+k \right\}\backslash
\left\{ i_{k+1}\right\}$ we introduce the set 
\begin{equation}
\label{eq:s9-B-III-i-minus}
\mathcal{B}^-_{III,i} = \left\{ 
\begin{aligned}
& \left(\tau,v_{s+k+1},\omega_{k+1}\right)\in
\mathcal{A}^-  \textnormal{ such that } \\
& \frac{\left|\left(\left(x_{i_{k+1}}^\prime - x_i^\prime\right)-
\tau \left(v_{i_{k+1}}^\prime - v_i^\prime\right)\right)\cdot
\left(v_{s+k+1}-v_i^\prime\right)\right|}
{\left|\left(x_{i_{k+1}}^\prime - x_i^\prime\right)-
\tau\left(v_{i_{k+1}}^\prime - v_i^\prime \right)\right|
\left| v_{s+k+1}-v_i^\prime \right| } \geq \cos \theta
\end{aligned}
 \right\}
\end{equation}
and we let 
$\mathcal{B}^-_{III}=\bigcup_{i\in
\left\{ 1,\dots,s+k\right\}\backslash\left\{i_{k+1}\right\}}
\mathcal{B}^-_{III,i}$; 
then we have
\begin{equation}
\label{eq:s9-measure5}
\begin{aligned}
\int_0^T \int_{B_{2R}^d} \int_{\mathbb{S}^{d-1}} 
 \mathbf{1}_{\left(\tau,v_{s+k+1},\omega_{k+1}\right) \in 
\mathcal{B}^-_{III}} & d\omega_{k+1} dv_{s+k+1} d\tau \leq \\
& \qquad  \leq C_d \left( s+k-1 \right) T R^d \theta^{d-1}
\end{aligned}
\end{equation}
\begin{remark}
The vector
\begin{equation*}
\left(x_{i_{k+1}}^\prime - x_i^\prime \right) -
\tau \left(v_{i_{k+1}}^\prime - v_i^\prime\right)
\end{equation*}
is just the relative displacement between the 
$i_{k+1}$st particle and the $i$th particle at the time of the particle
creation. On the other hand, $\left(v_{s+k+1}-v_i^\prime\right)$ is
the relative velocity between the $(s+k+1)$st particle and the
$i$th particle at the time of particle creation. Note that the $(s+k+1)$st
particle is created at a distance of $\varepsilon$ from the
$i_{k+1}$st particle. Hence the formula defining
$\mathcal{B}^-_{III,i}$ is a ``cone condition'' whose complementary event
 prevents the newly created $(s+k+1)$st particle from colliding with the
$i$th particle under the backwards hard sphere flow, as long as
$\theta$ is not too small.
\end{remark}

To conclude, we let
$\mathcal{B}^- = \mathcal{B}^-_I \cup
\mathcal{B}^-_{II} \cup \mathcal{B}^-_{III}$; then we have
\begin{equation}
\label{eq:s9-measure6}
\begin{aligned}
\int_0^T \int_{B_{2R}^d} \int_{\mathbb{S}^{d-1}} 
 \mathbf{1}_{\left(\tau,v_{s+k+1},\omega_{k+1}\right) \in 
\mathcal{B}^-} & d\omega_{k+1} dv_{s+k+1} d\tau \leq \\
& \leq C_d \left( s+k \right) T R^d \left[
\frac{y}{\eta T} + \left(\frac{\eta}{R}\right)^d + \theta^{d-1}\right]
\end{aligned}
\end{equation}
Then again, by assumption, $\sin \theta > c_d y^{-1} \varepsilon$; by
choosing $c_d$ sufficiently large we may guarantee that 
\begin{equation}
\label{eq:s9-Z4}
\begin{aligned}
Z_{s,s+k+1} & \left[ Z_s,t+\tau; t_1+\tau,\dots,t_k+\tau,0;
v_{s+1},\dots,v_{s+k},v_{s+k+1}; \right. \\
& \qquad \qquad \qquad \qquad \qquad \qquad \left.
\omega_1,\dots,\omega_k,\omega_{k+1}; i_1,\dots,i_k,i_{k+1}\right] \\
& \qquad \qquad \qquad \qquad \qquad \qquad \qquad
\qquad \qquad  
 \in \mathcal{K}_{s+k+1}\cap \mathcal{U}_{s+k+1}^\eta
\end{aligned}
\end{equation}
whenever
$\left( \tau,v_{s+k+1},\omega_{k+1}\right) \in
 \mathcal{A}^- \backslash\mathcal{B}^-$.

\emph{Construction of $\mathcal{B}^+$.} 
The construction of $\mathcal{B}^+$ is very similar to the
construction of $\mathcal{B}^-$; the main difference is that we have
to account for the change of variables arising from one collision.
We will find it helpful to define the following notation:
\begin{equation}
\label{eq:s9-vstar1}
v_{s+k+1}^* = v_{s+k+1} -
\omega_{k+1} \omega_{k+1}\cdot\left(v_{s+k+1}-
v_{i_{k+1}}^\prime\right)
\end{equation}
\begin{equation}
\label{eq:s9-vstar2}
v_{i_{k+1}}^{\prime *} = v_{i_{k+1}}^\prime +
\omega_{k+1} \omega_{k+1}\cdot \left(v_{s+k+1}-
v_{i_{k+1}}^\prime\right)
\end{equation}
Note that $Z_{s+k}^\prime$ is \emph{fixed} as in the statement of
the proposition, whereas $\left( \tau,v_{s+k+1},\omega_{k+1}\right)
\in \mathcal{A}^+$ are considered free parameters.

We eliminate creation times $\tau$ for which particles are
too concentrated in space:
\begin{equation}
\label{eq:s9-B-I-plus}
\mathcal{B}^+_I = \left\{
\begin{aligned}
& \left(\tau,v_{s+k+1},\omega_{k+1}\right)\in
\mathcal{A}^+ \textnormal{ such that }\\
&  \inf_{i\in \left\{1,\dots,s,s+1,\dots,s+k\right\}
\backslash \left\{i_{k+1}\right\}} 
\left| \left(x_{i_{k+1}}^\prime - x_i^\prime\right) - 
\tau \left(v_{i_{k+1}}^\prime - v_i^\prime\right)\right| \leq y
\end{aligned}
\right\}
\end{equation}
then we have
\begin{equation}
\label{eq:s9-measure7}
\begin{aligned}
\int_0^T \int_{B_{2R}^d} \int_{\mathbb{S}^{d-1}} 
 \mathbf{1}_{\left(\tau,v_{s+k+1},\omega_{k+1}\right) \in \mathcal{B}^+_I}
& d\omega_{k+1} dv_{s+k+1} d\tau \leq \\
& \qquad \leq C_d \left(s+k-1\right) R^d \eta^{-1} y
\end{aligned}
\end{equation}
We find it convenient to eliminate collisions which are too
close to grazing; therefore, we define
\begin{equation}
\label{eq:s9-B-II-plus}
\mathcal{B}^+_{II} = \left\{
\begin{aligned}
& \left(\tau,v_{s+k+1},\omega_{k+1}\right)\in
\mathcal{A}^+ \textnormal{ such that } \\
& \left| \omega_{k+1}\cdot\left(v_{s+k+1}-v_{i_{k+1}}^\prime\right) \right|
\leq\left(\sin\alpha\right) \left| v_{s+k+1}-v_{i_{k+1}}^\prime \right|
\end{aligned}
\right\}
\end{equation}
then we have
\begin{equation}
\label{eq:s9-measure8}
\int_0^T \int_{B_{2R}^d} \int_{\mathbb{S}^{d-1}} 
\mathbf{1}_{\left(\tau,v_{s+k+1},\omega_{k+1}\right)\in\mathcal{B}^+_{II}}
\; d\omega_{k+1} dv_{s+k+1} d\tau \leq
C_d T R^d \alpha
\end{equation}
We introduce the next three sets to guarantee that the $(s+k+1)$-particle 
state lives in $\mathcal{U}_{s+k+1}^\eta$. In this instance we must impose
\emph{multiple} conditions, since both the $(s+k+1)$st particle and the
$i_{k+1}$st particle are modified by the collision.
Note that $\left| v_{s+k+1}^*-v_{i_{k+1}}^{\prime *} \right| =
\left| v_{s+k+1}-v_{i_{k+1}}^\prime\right|$.
\begin{equation}
\label{eq:s9-B-III-plus}
\mathcal{B}^+_{III} = \left\{
\begin{aligned}
& \left(\tau,v_{s+k+1},\omega_{k+1}\right)\in
\mathcal{A}^+\backslash\mathcal{B}^+_{II}\; \textnormal{ such that } \\
& \inf_{i\in\left\{1,\dots,s,s+1,\dots,s+k\right\}
\backslash \left\{i_{k+1}\right\}}\left|
v_{s+k+1}^* - v_i^\prime \right| \leq \eta
\end{aligned}
\right\}
\end{equation}
\begin{equation}
\label{eq:s9-B-IV-plus}
\mathcal{B}^+_{IV} = \left\{
\begin{aligned}
& \left(\tau,v_{s+k+1},\omega_{k+1}\right)\in
\mathcal{A}^+ \backslash \mathcal{B}^+_{II} \; \textnormal{ such that }\\
&  \inf_{i\in\left\{1,\dots,s,s+1,\dots,s+k\right\}
\backslash \left\{i_{k+1}\right\} } \left|
v_{i_{k+1}}^{\prime *} - v_i^\prime \right| \leq \eta
\end{aligned} \right\}
\end{equation}
\begin{equation}
\label{eq:s9-B-V-plus}
\mathcal{B}^+_V = \left\{ \left( \tau,v_{s+k+1},\omega_{k+1}\right)\in
\mathcal{A}^+ \left| \left|v_{s+k+1}-v_{i_{k+1}}^\prime\right|
\leq \eta\right.\right\}
\end{equation}
Then using Lemma \ref{lemma:s9-sphere} and the definition
of $\mathcal{B}^+_{II}$, we obtain:
\begin{equation}
\label{eq:s9-measure9}
\begin{aligned}
\int_0^T \int_{B_{2R}^d} \int_{\mathbb{S}^{d-1}} 
 \mathbf{1}_{\left(\tau,v_{s+k+1},\omega_{k+1}\right) \in 
\mathcal{B}^+_{III}} & d\omega_{k+1} dv_{s+k+1} d\tau \leq \\
& \qquad \leq C_{d,\alpha} (s+k-1) T R \eta^{d-1}
\end{aligned}
\end{equation}
\begin{equation}
\label{eq:s9-measure10}
\begin{aligned}
\int_0^T \int_{B_{2R}^d} \int_{\mathbb{S}^{d-1}} 
 \mathbf{1}_{\left(\tau,v_{s+k+1},\omega_{k+1}\right) \in 
\mathcal{B}^+_{IV}} & d\omega_{k+1} dv_{s+k+1} d\tau \leq \\
& \qquad \leq C_{d,\alpha} (s+k-1) T R \eta^{d-1}
\end{aligned}
\end{equation}
\begin{equation}
\label{eq:s9-measure11}
\int_0^T \int_{B_{2R}^d} \int_{\mathbb{S}^{d-1}} 
 \mathbf{1}_{\left(\tau,v_{s+k+1},\omega_{k+1}\right) \in 
\mathcal{B}^+_V} \; d\omega_{k+1} dv_{s+k+1} d\tau \leq
C_d T \eta^d
\end{equation}

We will now show that, with high probability, the particle creation yields 
an $(s+k+1)$-particle state in $\mathcal{K}_{s+k+1}$, hence the backwards
hard sphere flow coincides with the free flow.
For $i\in\left\{1,\dots,s,s+1,\dots,s+k\right\}\backslash
\left\{i_{k+1}\right\}$, we define
\begin{equation}
\label{eq:s9-B-VI-i-plus}
\mathcal{B}^+_{VI,i} = \left\{
\begin{aligned}
& \left(\tau,v_{s+k+1},\omega_{k+1}\right)\in
\mathcal{A}^+\backslash\mathcal{B}^+_{II} \textnormal{ such that } \\
& \frac{\left|\left(\left(x_{i_{k+1}}^\prime - x_i^\prime\right)-
\tau \left(v_{i_{k+1}}^\prime - v_i^\prime\right)\right)\cdot
\left(v_{s+k+1}^*-v_i^\prime\right)\right|}
{\left|\left(x_{i_{k+1}}^\prime - x_i^\prime\right)-
\tau\left(v_{i_{k+1}}^\prime - v_i^\prime \right)\right|
\left| v_{s+k+1}^*-v_i^\prime \right| } \geq \cos \theta
\end{aligned}
\right\}
\end{equation}
\begin{equation}
\label{eq:s9-B-VII-i-plus}
\mathcal{B}^+_{VII,i} = \left\{
\begin{aligned}
& \left(\tau,v_{s+k+1},\omega_{k+1}\right)\in
\mathcal{A}^+\backslash\mathcal{B}^+_{II} \textnormal{ such that } \\
& \frac{\left|\left(\left(x_{i_{k+1}}^\prime - x_i^\prime\right)-
\tau \left(v_{i_{k+1}}^\prime - v_i^\prime\right)\right)\cdot
\left(v_{i_{k+1}}^{\prime *}-v_i^\prime\right)\right|}
{\left|\left(x_{i_{k+1}}^\prime - x_i^\prime\right)-
\tau\left(v_{i_{k+1}}^\prime - v_i^\prime \right)\right|
\left| v_{i_{k+1}}^{\prime *}-v_i^\prime \right| }  \geq \cos \theta
\end{aligned}
\right\}
\end{equation}
\begin{equation}
\label{eq:s9-B-VI-plus}
\mathcal{B}^+_{VI}=\bigcup_{i\in\left\{1,\dots,s,s+1,
\dots,s+k\right\}\backslash\left\{i_{k+1}\right\}}\mathcal{B}^+_{VI,i}
\end{equation}
\begin{equation}
\label{eq:s9-B-VII-plus}
\mathcal{B}^+_{VII}=\bigcup_{i\in\left\{1,\dots,s,s+1,\dots,s+k\right\}
\backslash\left\{i_{k+1}\right\}}\mathcal{B}^+_{VII,i}
\end{equation}
Then using Lemmas \ref{lemma:s9-sphere} and 
\ref{lemma:s9-cylinder}, and the definition of
$\mathcal{B}^+_{II}$, we have
\begin{equation}
\label{eq:s9-measure12}
\begin{aligned}
\int_0^T \int_{B_{2R}^d} \int_{\mathbb{S}^{d-1}} 
 \mathbf{1}_{\left(\tau,v_{s+k+1},\omega_{k+1}\right) \in 
\mathcal{B}^+_{VI}} & d\omega_{k+1} dv_{s+k+1} d\tau \leq \\
& \qquad \leq C_{d,\alpha} \left( s+k-1 \right) T R^d \theta^{(d-1)/2}
\end{aligned}
\end{equation}
\begin{equation}
\label{eq:s9-measure13}
\begin{aligned}
\int_0^T \int_{B_{2R}^d} \int_{\mathbb{S}^{d-1}} 
 \mathbf{1}_{\left(\tau,v_{s+k+1},\omega_{k+1}\right) \in 
\mathcal{B}^+_{VII}} & d\omega_{k+1} dv_{s+k+1} d\tau \leq \\
& \qquad \leq C_{d,\alpha} \left( s+k-1 \right) T R^d \theta^{(d-1)/2}
\end{aligned}
\end{equation}

To conclude, we let
$\mathcal{B}^+ = \mathcal{B}^+_I \cup
\mathcal{B}^+_{II} \cup \mathcal{B}^+_{III} \cup
\mathcal{B}^+_{IV}\cup\mathcal{B}^+_V \cup
\mathcal{B}^+_{VI} \cup \mathcal{B}^+_{VII}$; then we have
\begin{equation}
\label{eq:s9-measure14}
\begin{aligned}
&\int_0^T \int_{B_{2R}^d} \int_{\mathbb{S}^{d-1}} 
 \mathbf{1}_{\left(\tau,v_{s+k+1},\omega_{k+1}\right) \in 
\mathcal{B}^+} \; d\omega_{k+1} dv_{s+k+1} d\tau \leq \\
& \qquad  \qquad \qquad  
\leq C_d \left( s+k \right) T R^d \left[\alpha +
\frac{y}{\eta T} + C_{d,\alpha}\left(\frac{\eta}{R}\right)^{d-1} +
C_{d,\alpha} \theta^{(d-1)/2}\right]
\end{aligned}
\end{equation}
Then again, by assumption, we have $\sin \theta > c_d y^{-1} \varepsilon$;
as long as $c_d$ is chosen sufficiently large, we always have
\begin{equation}
\label{eq:s9-Z5}
\begin{aligned}
& Z_{s,s+k+1} \left[ Z_s,t+\tau; t_1+\tau,\dots,t_k+\tau,0;
v_{s+1},\dots,v_{s+k},v_{s+k+1}; \right. \\
& \qquad \qquad \qquad \qquad \qquad \qquad \qquad \left.
\omega_1,\dots,\omega_k,\omega_{k+1}; i_1,\dots,i_k,i_{k+1}\right] \\
& \qquad \qquad \qquad \qquad \qquad \qquad
\qquad \qquad \qquad 
 \in \mathcal{K}_{s+k+1}\cap \mathcal{U}_{s+k+1}^\eta
\end{aligned}
\end{equation}
whenever
$\left( \tau,v_{s+k+1},\omega_{k+1}\right) \in
 \mathcal{A}^+ \backslash\mathcal{B}^+$.
\qed
\end{proof}

\section{The Boltzmann hierarchy}
\label{sec:10}

We will say that a sequence of continuous symmetric functions
$\left\{f_\infty^{(s)} (t,Z_s)\right\}_{s\in\mathbb{N}}$, with
$Z_s \in \mathbb{R}^{2ds}$, satisfies the
Boltzmann hierarchy if the following equation holds for each
$s$ in the sense of distributions:
\begin{equation}
\label{eq:s10-boltz1}
\left( \frac{\partial}{\partial t} + V_s \cdot \nabla_{X_s} \right)
f_\infty^{(s)} (t,Z_s) = \ell^{-1} C_{s+1}^0
 f_\infty^{(s+1)} (t,Z_s)
\end{equation}
The collision operators $C_{s,s+1}^0$ are defined as follows:
\begin{equation}
\label{eq:s10-coll1}
C_{s+1}^0 = \sum_{i=1}^s C_{i,s+1}^{0}
\end{equation}
\begin{equation}
\label{eq:s10-coll2}
C_{i,s+1}^0 = 
C_{i,s+1}^{0,+} - C_{i,s+1}^{0,-}
\end{equation}
\begin{equation}
\label{eq:s10-coll3}
\begin{aligned}
& C_{i,s+1}^{0,+} f_\infty^{(s+1)} (t,Z_s) =\int_{\mathbb{R}^d}
\int_{\mathbb{S}^{d-1}}
\left[\omega \cdot (v_{s+1}-v_i)\right]_{+} \times \\
& \;\; \qquad \times f_\infty^{(s+1)} (t,x_1,v_1,\dots,x_i,v_i^*,\dots,
x_s,v_s,x_i, v_{s+1}^*) d\omega dv_{s+1} 
\end{aligned}
\end{equation}
\begin{equation}
\label{eq:s10-coll4}
\begin{aligned}
& C_{i,s+1}^{0,-} f_\infty^{(s+1)} (t,Z_s) =\int_{\mathbb{R}^d}
\int_{\mathbb{S}^{d-1}}
\left[\omega \cdot (v_{s+1}-v_i)\right]_{-} \times \\
& \;\; \qquad \times f_\infty^{(s+1)} (t,x_1,v_1,
\dots,x_i,v_i,\dots,x_s,v_s,x_i, v_{s+1})
d\omega dv_{s+1}
\end{aligned}
\end{equation}
where
\begin{equation}
\label{eq:s10-vstar}
\begin{cases}
v_i^* = v_i + \omega \omega \cdot \left( v_j-v_i\right)\\
v_j^* = v_j - \omega \omega \cdot \left( v_j-v_i\right) 
\end{cases}
\end{equation}
We also define the free transport operators $T_s^0 (t)$, which act
on functions $f_\infty^{(s)}:\mathbb{R}^{2ds}\rightarrow\mathbb{R}$ 
as follows:
\begin{equation}
\label{eq:s10-T-s-0}
\left(T_s^0 (t) f_\infty^{(s)}\right) (X_s,V_s) = 
f_\infty^{(s)} (X_s-V_s t,V_s)
\end{equation}
Just as for the BBGKY hierarchy, the Boltzmann hierarchy admits
a \emph{formal} Duhamel series expressing the solution 
in terms of the data,
\begin{equation}
\label{eq:s10-series1}
\begin{aligned}
f_\infty^{(s)} (t) & = \sum_{k=0}^\infty \ell^{-k} \\
& \times
\int_0^t \int_0^{t_1} \dots \int_0^{t_{k-1}} T_s^0 (t-t_1) C_{s+1}^0\dots
T_{s+k}^0 (t_k) f_\infty^{(s+k)} (0) dt_k \dots dt_1
\end{aligned}
\end{equation}
The convergence of this series (for small data) follows from the
well-posedness theorem which is proven in the following section.

\begin{remark}
If $f_t (x,v)$ is a sufficiently smooth solution of the Boltzmann
equation then the sequence $\left\{f_t^{\otimes s}\right\}_{s\in\mathbb{N}}$
is a solution of the Boltzmann hierarchy.
\end{remark}

We will now construct pseudo-trajectories
for the Boltzmann hierarchy, directly analogous to those we have
constructed for the BBGKY hierarchy. \cite{L1975,GSRT2014,PSS2014}
Given $Z_s \in \mathbb{R}^{2ds}$, along with
times $0\leq t_k \leq \dots \leq t_1 \leq t$, velocities 
$v_{s+1},\dots,v_{s+k}$, impact parameters
$\omega_1,\dots,\omega_k$, and indices $i_1,\dots,i_k$, we will define
\begin{equation}
\label{eq:s10-Z1}
Z_{s,s+k}^0 \left[ Z_s , t ; t_1,\dots,t_k;v_{s+1},\dots,v_{s+k};
 \omega_1,\dots,\omega_k;i_1,\dots,i_k \right] \in
\mathbb{R}^{2d(s+k)}
\end{equation}
We assume $i_1 \in \left\{ 1,\dots,s\right\}$,
$i_2 \in \left\{ 1,\dots,s,s+1\right\}$, \dots,
$i_j \in \left\{ 1,2,\dots,s+j-1\right\}$.
 To begin the induction,
for $Z_s = \left(X_s,V_s\right) \in \mathbb{R}^{2ds}$ and $t > 0$
 we define
\begin{equation}
\label{eq:s10-Z2}
Z_{s,s}^0 \left[ Z_s,t\right] = \left(X_s-V_s t,V_s\right)
\end{equation} 
More generally, if the symbol
\begin{equation}
\label{eq:s10-Z3}
\begin{aligned}
& Z_{s,s+k}^0 \left[ Z_s , t ; t_1,\dots,t_k;v_{s+1},\dots,v_{s+k};
 \omega_1,\dots,\omega_k;i_1,\dots,i_k \right] = \\
& \qquad \qquad \qquad \qquad \qquad \qquad \qquad \qquad \;\;
=\left( X_{s+k}^\prime,V_{s+k}^\prime\right) \in \mathbb{R}^{2d(s+k)}
\end{aligned}
\end{equation}
is defined, then for $\tau > 0$ we define
\begin{equation}
\label{eq:s10-Z4}
\begin{aligned}
& Z_{s,s+k}^0 \left[ Z_s , t+\tau ; t_1+\tau,\dots,t_k+\tau;
v_{s+1},\dots,v_{s+k};\omega_1,\dots,\omega_k;
i_1,\dots,i_k \right] = \\
& \qquad \qquad \qquad \qquad \qquad \qquad \qquad \qquad \qquad \qquad
= \left( X_{s+k}^\prime - V_{s+k}^\prime \tau,V_{s+k}^\prime \right)
\end{aligned}
\end{equation}
Similarly, if the symbol
\begin{equation}
\label{eq:s10-Z5}
\begin{aligned}
& Z_{s,s+k}^0 \left[ Z_s , t ; t_1,\dots,t_k;v_{s+1},\dots,v_{s+k};
 \omega_1,\dots,\omega_k; i_1,\dots,i_k \right]  = \\
& \qquad \qquad \qquad \qquad \qquad \qquad \qquad \qquad
= \left( X_{s+k}^\prime,V_{s+k}^\prime \right) \in \mathbb{R}^{2d(s+k)}
\end{aligned}
\end{equation}
is defined (including the possibility $k=0$) then for
any given velocity $v_{s+k+1} \in \mathbb{R}^d$, any index
$i_{k+1}\in \left\{ 1,\dots,s,s+1,\dots,s+k\right\}$,
and any choice of impact parameter $\omega_{k+1}\in \mathbb{S}^{d-1}$, 
if $\omega_{k+1} \cdot \left( v_{s+k+1}-v_{i_{k+1}}^\prime\right)\leq 0$
 we define
\begin{equation}
\label{eq:s10-Z6}
\begin{aligned}
& Z_{s,s+k+1}^0 \left[ Z_s , t ; t_1,\dots,t_k,0; 
v_{s+1},\dots,v_{s+k},v_{s+k+1}; \right.\\
& \qquad \qquad \qquad \qquad \qquad \qquad
\left.\omega_1,\dots,\omega_k,\omega_{k+1};
 i_1,\dots,i_k,i_{k+1} \right] =\\
& \qquad =\left( x_1^\prime,v_1^\prime,\dots,
x_{i_{k+1}}^\prime,v_{i_{k+1}}^\prime,\dots,x_s^\prime,v_s^\prime,
x_{i_{k+1}}^\prime,v_{s+k+1}\right)
\end{aligned}
\end{equation}
whereas if $\omega_{k+1}\cdot\left(v_{s+k+1}-v_{i_{k+1}}^\prime\right)>0$ 
then we define
\begin{equation}
\label{eq:s10-Z7}
\begin{aligned}
& Z_{s,s+k+1}^0 \left[ Z_s , t ; t_1,\dots,t_k,0; 
v_{s+1},\dots,v_{s+k},v_{s+k+1}; \right. \\
& \qquad \qquad \qquad \qquad \qquad \qquad
\left. \omega_1,\dots,\omega_k,\omega_{k+1};
 i_1,\dots,i_k,i_{k+1} \right] =\\
& \qquad =\left( x_1^\prime,v_1^\prime,\dots,
x_{i_{k+1}}^\prime,v_{i_{k+1}}^\prime+\omega_{k+1}
 \omega_{k+1}\cdot\left(v_{s+k+1}-v_{i_{k+1}}^\prime\right),\right. \\ 
&\qquad \qquad \left.\dots,x_s^\prime,v_s^\prime,
x_{i_{k+1}}^\prime,v_{s+k+1}-\omega_{k+1} 
\omega_{k+1}\cdot\left(v_{s+k+1}-v_{i_{k+1}}^\prime\right)\right)
\end{aligned}
\end{equation}

Now we construct the collision kernel
$b_{s,s+k}^0 \left[ Z_s,t;\left\{t_j,v_{s+j},\omega_j,i_j
\right\}_{j=1}^k \right]$. First we define
\begin{equation}
\label{eq:s10-b1}
b_{s,s}^0 \left[ Z_s,t\right] = 1
\end{equation}
If we have defined
\begin{equation}
\label{eq:s10-b2}
b_{s,s+k}^0 \left[ Z_s,t;t_1,\dots,t_k;v_{s+1},\dots,v_{s+k};
\omega_1,\dots,\omega_k;i_1,\dots,i_k\right] 
\end{equation}
then for any $\tau > 0$ we define
\begin{equation}
\label{eq:s10-b3}
\begin{aligned}
& b_{s,s+k}^0 \left[
Z_s,t+\tau;t_1+\tau,\dots,t_k+\tau;v_{s+1},\dots,v_{s+k};
\omega_1,\dots,\omega_k;i_1,\dots,i_k\right] = \\
& \qquad  \qquad \;\;
= b_{s,s+k}^0 \left[ Z_s,t;t_1,\dots,t_k;v_{s+1},\dots,v_{s+k};
\omega_1,\dots,\omega_k;i_1,\dots,i_k\right]
\end{aligned}
\end{equation} 
and we also define
\begin{equation}
\label{eq:s10-b4}
\begin{aligned}
& b_{s,s+k+1}^0 \left[
Z_s,t;t_1,\dots,t_k,0;v_{s+1},\dots,v_{s+k},v_{s+k+1};\right. \\
& \qquad \qquad \qquad \qquad \qquad \qquad
\left. \omega_1,\dots,\omega_k,\omega_{k+1};i_1,\dots,i_k,i_{k+1}\right] = \\
& \qquad = \omega_{k+1}\cdot \left( v_{s+k+1}-v_{i_{k+1}}^\prime\right)
 \times \\
& \qquad \qquad \times b_{s,s+k}^0 \left[ Z_s,t;t_1,\dots,t_k;
v_{s+1},\dots,v_{s+k};\omega_1,\dots,\omega_k;i_1,\dots,i_k\right]
\end{aligned}
\end{equation} 
Then the formal
 Duhamel series (\ref{eq:s10-series1}) becomes
\begin{equation}
\label{eq:s10-series2}
\begin{aligned}
& f_\infty^{(s)} (t,Z_s) = \sum_{k=0}^\infty \ell^{-k} \times \\
& \times \sum_{i_1 = 1}^s \dots \sum_{i_k = 1}^{s+k-1}
\int_0^t \dots \int_0^{t_{k-1}}
\int_{\mathbb{R}^{dk}} \int_{\left(\mathbb{S}^{d-1}\right)^k} 
\left( \prod_{m=1}^k d\omega_m dv_{s+m} dt_m \right) \times \\
& \times \left( b_{s,s+k}^0 \left[\cdot\right]
f_\infty^{(s+k)} (0,Z_{s,s+k}^0 \left[\cdot\right])\right)
\left[ Z_s,t;\left\{t_j,v_{s+j},\omega_j,i_j
\right\}_{j=1}^k \right] 
\end{aligned}
\end{equation}

\section{Small solutions of the Boltzmann hierarchy}
\label{sec:11}

We will prove a global well-posedness result for the Boltzmann hierarchy
with small data 
$F_\infty(0)=\left\{f_\infty^{(s)}(0)\right\}_{s\in\mathbb{N}}$
 in vacuum. The proof is based on a fixed point iteration and
a dispersive estimate. \cite{IP1986,BD1985}
 If, in addition to the hypotheses of the theorem,
we have $f_\infty^{(s)} (0)=f_0^{\otimes s}$ for some smooth
 function $f_0(x,v)$, then it is well-known that the Boltzmann equation has 
a unique non-negative smooth solution $f_t$ \cite{CIP1994,BD2000}, 
and $\left\{ f_t^{\otimes s}\right\}_{s\in\mathbb{N}}$
solves the Boltzmann hierarchy. Then the uniqueness part of the
following theorem implies that 
$F_\infty(t)=\left\{ f_t^{\otimes s}\right\}_{s\in\mathbb{N}}$, i.e.,
the Boltzmann hierarchy propagates chaoticity.

\begin{theorem} 
\label{thm:s11-ip1}
(Illner \& Pulvirenti 1986)
Suppose $F_\infty(0) = 
\left\{ f_\infty^{(s)} (0)\right\}_{s\in\mathbb{N}}$ is a
sequence of functions such that each 
$f_\infty^{(s)}(0):\mathbb{R}^{2ds}\rightarrow\mathbb{R}$
is continuous and symmetric, and for some $\beta_0>0,\mu_0\in\mathbb{R}$,
\begin{equation}
\label{eq:s11-boltz1}
\sup_{s\in\mathbb{N}} \sup_{Z_s\in\mathbb{R}^{2ds}}
\left| f_\infty^{(s)} (0,Z_s)\right| 
e^{\beta_0 \left[ E_s(Z_s)+I_s(Z_s)\right]}e^{\mu_0 s} \leq 1
\end{equation}
Then if $d\geq 3$ and
 $\ell^{-1} e^{-\mu_0} \beta_0^{-\frac{d+1}{2}}$ is sufficiently
small (depending only on $d$), then there exists a unique sequence
$F_\infty (t) = \left\{ f_\infty^{(s)} (t)\right\}_{s\in\mathbb{N}}$, 
with each 
$f_\infty^{(s)}(t,Z_s):[0,\infty)\times\mathbb{R}^{2ds}\rightarrow\mathbb{R}$
continuous and symmetric, such that
\begin{equation}
\label{eq:s11-boltz2}
\sup_{t\geq 0} \sup_{s\in\mathbb{N}} \sup_{Z_s\in\mathbb{R}^{2ds}}
\left| f_\infty^{(s)} (t,Z_s)\right| 
e^{\frac{1}{2} \beta_0
 \left[ E_s(Z_s)+I_s((X_s-V_s t,V_s))\right]}e^{(\mu_0-1) s} \leq 2
\end{equation}
and for each $s\in\mathbb{N}$ there holds
\begin{equation}
\label{eq:s11-boltz3}
\left(\frac{\partial}{\partial t} + V_s\cdot \nabla_{X_s}\right)
f_\infty^{(s)} (t,Z_s) = \ell^{-1} C_{s+1}^0 f_\infty^{(s+1)} (t,Z_s)
\end{equation}
in the sense of distributions.
\end{theorem}

\begin{proof} 
Recall the free evolution 
$\left(T_{s}^{0}\left(t\right)f_\infty^{\left(s\right)}\right)
\left(Z_{s}\right)
=f_\infty^{\left(s\right)}\left(X_{s}-V_{s}t,V_{s}\right)$, where
$Z_s \in \mathbb{R}^{2ds}$.
Subject to the estimates stated in the theorem, and the continuity
of $f_\infty^{(s)}(t,Z_s)$, the weak form of the Boltzmann hierarchy is
equivalent to the following mild form:
\begin{equation}
\label{eq:s11-boltz4}
f_\infty^{\left(s\right)}\left(t\right)= 
T_{s}^{0}\left(t\right)f_\infty^{\left(s\right)}\left(0\right)+
\ell^{-1} \int_{0}^{t}T_{s}^{0}\left(t-\tau\right)
C_{s+1}^0 f_\infty^{\left(s+1\right)}\left(\tau\right)d\tau
\end{equation}
At this point it is convenient to change the coordinates. Let us define 
$G_\infty \left(t\right)=
\left\{ g_\infty^{\left(s\right)}\left(t\right)\right\} _{s\geq1}$
by $g_\infty^{\left(s\right)}\left(t\right)=T_{s}^{0}
\left(-t\right)f_\infty^{\left(s\right)}\left(t\right)$,
and write 
\begin{equation}
\label{eq:s11-V1}
V_{s+1}^{0}\left(\tau\right)
=T_{s}^{0}\left(-\tau\right)C_{s+1}^0 T_{s+1}^{0}\left(\tau\right)
\end{equation}
Then we have
\begin{equation}
\label{eq:s11-boltz5}
g_\infty^{\left(s\right)}\left(t\right)=
g_\infty^{\left(s\right)}\left(0\right)+
\ell^{-1} \int_{0}^{t}V_{s+1}^{0}\left(\tau\right)
g_\infty^{\left(s+1\right)} \left(\tau\right)d\tau
\end{equation}
We record an explicit formula for the action of the
operator $V_{s,s+1}^{0}\left(\tau\right)$:
\begin{equation}
\label{eq:s11-V2}
V_{s+1}^0 (\tau) = V_{s+1}^{0,+} (\tau) - V_{s+1}^{0,-} (\tau)
\end{equation}
\begin{equation}
\label{eq:s11-V3}
\begin{aligned}
&\left(V_{s+1}^{0,+}\left(\tau\right)g_\infty^{\left(s+1\right)}(t)
\right)\left(Z_{s}\right) =\sum_{i=1}^{s}\int_{\mathbb{R}^{d}}
\int_{\mathbb{S}^{d-1}} d\omega dv_{s+1}
\left[\omega\cdot\left(v_{s+1}-v_{i}\right)\right]_+ \times \\
&\qquad \times g_\infty^{\left(s+1\right)}\left(t,x_{1},v_{1},\dots,
x_{i}-\left(v_{i}^{*}-v_{i}\right)\tau,v_{i}^{*},\dots, \right. \\
& \qquad \qquad \qquad \qquad \qquad \qquad \qquad \left. \dots,
x_{s},v_{s},x_{i}-\left(v_{s+1}^{*}-v_{i}\right)\tau,v_{s+1}^{*} \right)
\end{aligned}
\end{equation}
\begin{equation}
\label{eq:s11-V4}
\begin{aligned}
&\left(V_{s+1}^{0,-}\left(\tau\right)g_\infty^{\left(s+1\right)}(t)\right)
\left(Z_{s}\right) =\sum_{i=1}^{s}
\int_{\mathbb{R}^{d}}\int_{\mathbb{S}^{d-1}} d\omega dv_{s+1}
\left[\omega\cdot\left(v_{s+1}-v_{i}\right)\right]_- \times\\
 & \qquad \times g_\infty^{\left(s+1\right)}
\left(t,x_{1},v_{1},\dots,x_{i},v_{i},\dots,x_{s},v_{s},
x_{i}-\left(v_{s+1}-v_{i}\right)\tau,v_{s+1}\right)
\end{aligned}
\end{equation}

We will prove pointwise bounds for the operators
$V_{s+1}^{0,\pm}(\tau)$. If $0<\beta^\prime <\beta$,
$\mu^\prime < \mu$, $t,\tau \geq 0$, then we have:
\begin{equation*}
\begin{aligned}
& \left| \left( e^{\mu^\prime s} e^{\beta^\prime (E_s(Z_s)+I_s(Z_s))}
V_{s+1}^{0,+} (\tau) g_\infty^{(s+1)}(t)\right) (Z_s) \right| \leq \\
& \leq \sum_{i=1}^s \int_{\mathbb{R}^d} \int_{\mathbb{S}^{d-1}}
d\omega dv_{s+1} |v_{s+1}-v_i|  
e^{-(\beta-\beta^\prime)E_s(Z_s)} e^{-(\mu-\mu^\prime)s}\times \\
& \times e^{-\frac{1}{2}\beta |v_{s+1}|^2} 
e^{\frac{1}{2} \beta \left( |x_i|^2 -|x_i-(v_i^*-v_i)\tau|^2 - 
|x_i-(v_{s+1}^*-v_i)\tau|^2\right)}e^{-\mu}\times \\
& \times
 e^{\mu (s+1)} e^{\frac{1}{2} \beta \sum_{i=1}^{s+1} |v_i|^2}
e^{\frac{1}{2} \beta \left( |x_1|^2 + \dots + 
|x_i - (v_i^* - v_i)\tau|^2 + \dots + |x_s|^2 +
|x_i-(v_{s+1}^*-v_i)\tau|^2\right)}\times \\
& \times\left| g_\infty^{(s+1)}\left(t,x_1,v_1,\dots,
x_i-(v_i^*-v_i)\tau,v_i^*,\dots \right.\right. \\
& \qquad \qquad \qquad \qquad \qquad
\left.\left.\dots,x_s,v_s,x_i-(v_{s+1}^*-v_i)\tau,v_{s+1}^*\right)\right| \\
& \leq \sum_{i=1}^s \int_{\mathbb{R}^d} \int_{\mathbb{S}^{d-1}}
d\omega dv_{s+1}
|v_{s+1}-v_i| e^{-(\beta-\beta^\prime)E_s(Z_s)} e^{-(\mu-\mu^\prime)s}
\times \\
& \times e^{-\frac{1}{2}\beta |v_{s+1}|^2} 
e^{\frac{1}{2} \beta \left( |x_i|^2 -|x_i-(v_i^*-v_i)\tau|^2 - 
|x_i-(v_{s+1}^*-v_i)\tau|^2\right)} e^{-\mu} \times \\
& \times \left\Vert e^{\mu (s+1)} 
e^{\beta \left(E_{s+1} (Z_{s+1}^\prime)+I_{s+1}(Z_{s+1}^\prime)\right)}
 g_\infty^{(s+1)}\left(t,Z_{s+1}^\prime\right)
\right\Vert_{L^\infty_{Z_{s+1}^\prime}}  \\
\end{aligned}
\end{equation*}
and similarly
\begin{equation*}
\begin{aligned}
& \left| \left( e^{\mu^\prime s} e^{\beta^\prime (E_s(Z_s)+I_s(Z_s))}
V_{s+1}^{0,-} (\tau) g_\infty^{(s+1)}(t)\right) (Z_s) \right| \leq \\
& \leq \sum_{i=1}^s \int_{\mathbb{R}^d} \int_{\mathbb{S}^{d-1}}
d\omega dv_{s+1} |v_{s+1}-v_i| 
e^{-(\beta-\beta^\prime)E_s(Z_s)} e^{-(\mu-\mu^\prime)s} \times \\
& \qquad \times e^{-\frac{1}{2}\beta |v_{s+1}|^2} 
e^{-\frac{1}{2} \beta |x_i-(v_{s+1}-v_i)\tau|^2} e^{-\mu} \times \\
& \qquad \times e^{\mu (s+1)} e^{\frac{1}{2} \beta \sum_{i=1}^{s+1} |v_i|^2}
e^{\frac{1}{2} \beta \left( |x_1|^2 + \dots + 
|x_i|^2 + \dots + |x_s|^2 +
|x_i-(v_{s+1}-v_i)\tau|^2\right)}\times \\
& \qquad \times\left| g_\infty^{(s+1)}\left(t,x_1,v_1,\dots,x_i,v_i,
\dots,x_s,v_s,x_i-(v_{s+1}-v_i)\tau,v_{s+1}\right)\right|  \\
& \leq \sum_{i=1}^s \int_{\mathbb{R}^d} \int_{\mathbb{S}^{d-1}}
d\omega dv_{s+1} |v_{s+1}-v_i| 
 e^{-(\beta-\beta^\prime)E_s(Z_s)} e^{-(\mu-\mu^\prime)s} \times \\
& \qquad \times e^{-\frac{1}{2}\beta |v_{s+1}|^2} 
e^{-\frac{1}{2} \beta |x_i-(v_{s+1}-v_i)\tau|^2} e^{-\mu} \times \\
& \qquad \times
\left\Vert e^{\mu (s+1)} 
e^{\beta\left(E_{s+1} (Z_{s+1}^\prime)+I_{s+1}(Z_{s+1}^\prime)\right)}
g_\infty^{(s+1)}\left(t,Z_{s+1}^\prime\right)
\right\Vert_{L^\infty_{Z_{s+1}^\prime}}  \\
\end{aligned}
\end{equation*}
The following identity follows from elementary manipulation:
\begin{equation}
\label{eq:s11-equality}
|x_i|^2 + |x_i-(v_{s+1}-v_i)\tau|^2 -
|x_i-(v_i^*-v_i)\tau|^2 - |x_i-(v_{s+1}^*-v_i)\tau|^2 = 0
\end{equation}

Therefore we obtain a bound on the full operator
$V_{s+1}^0 (\tau)$,
\begin{equation}
\label{eq:s11-V5}
\begin{aligned}
& \left| \left( e^{\mu^\prime s} e^{\beta^\prime (E_s(Z_s)+I_s(Z_s))}
V_{s+1}^{0} (\tau) g_\infty^{(s+1)}(t)\right) (Z_s) \right| \leq \\
& \leq 2 \sum_{i=1}^s \int_{\mathbb{R}^d} \int_{\mathbb{S}^{d-1}}
d\omega dv_{s+1} |v_{s+1}-v_i| 
 e^{-(\beta-\beta^\prime)E_s(Z_s)} e^{-(\mu-\mu^\prime)s} \times \\
& \qquad \qquad \times e^{-\frac{1}{2}\beta |v_{s+1}|^2} 
e^{-\frac{1}{2} \beta |x_i-(v_{s+1}-v_i)\tau|^2} e^{-\mu} \times \\
& \qquad \qquad \times  \left\Vert e^{\mu (s+1)} 
e^{\beta\left( E_{s+1} (Z_{s+1}^\prime)+I_{s+1}(Z_{s+1}^\prime)\right)}
g_\infty^{(s+1)}\left(t,Z_{s+1}^\prime\right)
\right\Vert_{L^\infty_{Z_{s+1}^\prime}} 
\end{aligned}
\end{equation}
We use the following dispersive inequality \cite{BD1985}:
\begin{equation}
\label{eq:s11-decay}
\left\Vert \zeta ( x-vt,v) \right\Vert_{L^\infty_x L^1_v} \leq
|t|^{-d} \left\Vert \zeta(x,v) \right\Vert_{L^1_x L^\infty_v}
\end{equation}
which implies the \emph{pointwise} bound
\begin{equation}
\label{eq:s11-V6}
\begin{aligned}
& \left| \left( e^{\mu^\prime s} e^{\beta^\prime (E_s(Z_s)+I_s(Z_s))}
V_{s+1}^{0} (\tau) g_\infty^{(s+1)}(t)\right) (Z_s) \right| \leq \\
& \;\; \leq C_d e^{-\mu} \beta^{-\frac{d}{2}} \left(1+\tau\right)^{-d}
\left( s^{\frac{1}{2}} E_s (Z_s)^{\frac{1}{2}} +
s \beta^{-\frac{1}{2}} \right) 
 e^{-(\beta-\beta^\prime)E_s(Z_s)} e^{-(\mu-\mu^\prime)s}\times \\
& \qquad \;\; \times \left\Vert e^{\mu (s+1)} 
e^{\beta \left(E_{s+1} (Z_{s+1}^\prime)+I_{s+1}(Z_{s+1}^\prime)\right)}
g_\infty^{(s+1)}\left(t,Z_{s+1}^\prime\right)
\right\Vert_{L^\infty_{Z_{s+1}^\prime}}
\end{aligned}
\end{equation}
and therefore also implies
\begin{equation}
\label{eq:s11-V7}
\begin{aligned}
& \left\Vert \left( e^{\mu^\prime s} e^{\beta^\prime (E_s(Z_s)+I_s(Z_s))}
V_{s+1}^{0} (\tau) g_\infty^{(s+1)}(t)\right) (Z_s)
 \right\Vert_{L^\infty_{Z_s}} \leq \\
& \qquad \leq C_d e^{-\mu} \beta^{-\frac{d}{2}} \left(1+\tau\right)^{-d}
\left( \frac{1}{\sqrt{\beta-\beta^\prime}\cdot\sqrt{\mu-\mu^\prime}} +
\frac{\beta^{-\frac{1}{2}}}{\mu-\mu^\prime} \right) \times \\
&\qquad \qquad \times \left\Vert e^{\mu (s+1)} 
e^{\beta \left(E_{s+1} (Z_{s+1}^\prime)+I_{s+1}(Z_{s+1}^\prime)\right)}
g_\infty^{(s+1)}\left(t,Z_{s+1}^\prime\right)
\right\Vert_{L^\infty_{Z_{s+1}^\prime}}
\end{aligned}
\end{equation}

Fix a sequence of positive numbers $r_0,r_1,r_2,\dots$
such that $0<r_{k+1}<r_{k}$ and 
$\sum_{k=0}^{\infty}r_k=1$. We define continuous decreasing 
functions $\beta\left(t\right)$, $\mu\left(t\right)$,
for $t\geq0$:
\begin{equation}
\label{eq:s11-weight1}
\beta\left(t\right)=\beta_{0}\cdot
\left[1-\frac{1}{2}\sum_{0\leq k<n}r_k-
\frac{1}{2}r_n \left(t-n\right)\right]
\qquad\forall\qquad t\in\left[n,n+1\right)
\end{equation}
\begin{equation}
\label{eq:s11-weight2}
\mu\left(t\right)=\mu_{0}-\sum_{0\leq k<n}r_k -
r_n \left(t-n\right)
\qquad\forall\qquad t\in\left[n,n+1\right)
\end{equation}
Using the \emph{pointwise} bound
 (\ref{eq:s11-V5}), we obtain
\begin{equation}
\label{eq:s11-V8}
\begin{aligned}
& \left| e^{\mu(t) s} e^{\beta(t) \left(E_s(Z_s)+I_s(Z_s)\right)}
\ell^{-1} \int_0^t
 \left(V_{s+1}^{0} (\tau) g_\infty^{(s+1)}(\tau)\right)(Z_s)d\tau \right|
 \leq \\
& \leq C_d \ell^{-1} e^{-(\mu_0 - 1)}
\left(\frac{ \beta_0}{2}\right)^{-\frac{d}{2}} 
\left( s^{\frac{1}{2}} E_s (Z_s)^{\frac{1}{2}} +
s \left(\frac{\beta_0}{2}\right)^{-\frac{1}{2}} \right) \times \\
& \qquad \;\; \times \int_0^t \left(1+\tau\right)^{-d}
e^{-\left(\beta(\tau)-\beta(t)\right) E_s (Z_s)}
e^{-\left( \mu(\tau)-\mu(t)\right) s} d\tau \times \\
& \qquad \;\; \times \left\Vert e^{\mu(t^\prime) (s+1)} 
e^{\beta(t^\prime) \left(E_{s+1} (Z_{s+1}^\prime)+I_{s+1}
(Z_{s+1}^\prime)\right)}
g_\infty^{(s+1)}\left(t^\prime,Z_{s+1}^\prime\right)
\right\Vert_{L^\infty_{t^\prime} L^\infty_{Z_{s+1}^\prime}}
\end{aligned}
\end{equation}
Then by a straightforward computation we have
\begin{equation}
\label{eq:s11-tau}
\int_0^t \left(1+\tau\right)^{-d}
e^{-\left(\beta(\tau)-\beta(t)\right) E_s (Z_s)}
e^{-\left(\mu(\tau)-\mu(t)\right) s} d\tau \leq
\frac{\sum_{k=0}^\infty r_k^{-1} \left(1+k\right)^{-d}}
{s+\frac{\beta_0}{2} E_s(Z_s)}
\end{equation}
Observe that if $d\geq 3$ then we may choose $r_k$ such that
$r_k \sim k^{-d+\frac{3}{2}}$ as $k\rightarrow \infty$,
and $\sum_{k=0}^\infty r_k = 1$; then, we will also have
$\sum_{k=0}^\infty r_k^{-1} \left(1+k\right)^{-d}<\infty$.
Hence for $d\geq 3$ there holds
\begin{equation}
\label{eq:s11-V9}
\begin{aligned}
& \left\Vert e^{\mu(t) s} e^{\beta(t) \left(E_s(Z_s)+I_s(Z_s)\right)}
\ell^{-1} \int_0^t
 \left(V_{s+1}^{0} (\tau) g_\infty^{(s+1)}(\tau)\right)(Z_s)d\tau 
\right\Vert_{L^\infty_t L^\infty_{Z_s}} \leq \\
&\leq C_d^\prime \ell^{-1} e^{-\mu_0} \beta_0^{-\frac{d+1}{2}}  \times \\
& \qquad \times  \left\Vert e^{\mu(t) (s+1)} 
e^{\beta(t) \left(E_{s+1} (Z_{s+1})+I_{s+1}
(Z_{s+1})\right)}
g_\infty^{(s+1)}\left(t,Z_{s+1}\right)
\right\Vert_{L^\infty_{t} L^\infty_{Z_{s+1}}}
\end{aligned}
\end{equation}

The Boltzmann hierarchy can be written in the following vector form:
\begin{equation}
\label{eq:s11-boltz6}
G_\infty\left(t\right)=G_\infty\left(0\right)+
\ell^{-1} \int_{0}^{t}V^{0}\left(\tau\right)
G_\infty \left(\tau\right)d\tau
\end{equation}
where $V^0 (\tau) G_\infty(t) = \left\{
V_{s+1}^0 (\tau) g_\infty^{(s+1)} (t)\right\}_{s\in\mathbb{N}}$.
We work in the Banach space 
$\left(\mathcal{X},\left\Vert\cdot\right\Vert\right)$
 of sequences $G_\infty(t)=
\left\{g_\infty^{(s)} (t)\right\}_{s\in\mathbb{N}}$
with each function
$g_\infty^{(s)} (t):[0,\infty)\times\mathbb{R}^{2ds}\rightarrow\mathbb{R}$
continuous and symmetric, and with norm
\begin{equation}
\label{eq:s11-norm1}
\left\Vert G_\infty \right\Vert =
\sup_{t \geq 0} \sup_{s\in\mathbb{N}} \sup_{Z_s \in \mathbb{R}^{2ds}}
e^{\mu (t) s} e^{\beta (t) \left( E_s (Z_s) + I_s (Z_s)\right)}
\left| g_\infty^{(s)} (t,Z_s)\right|
\end{equation}
Then we may define the operator 
$\mathcal{V}:\mathcal{X}\rightarrow\mathcal{X}$,
\begin{equation}
\label{eq:s11-VV}
\left( \mathcal{V} G_\infty \right)(t) =
\ell^{-1} \int_0^t V^0 (\tau) G_\infty (\tau) d\tau
\end{equation}
We may view the data $G_\infty (0)$ as an element of $\mathcal{X}$ which
simply does not depend on time. Then the Boltzmann
hierarchy may be written as
\begin{equation}
\label{eq:s11-boltz7}
G_\infty = G_\infty (0) + \mathcal{V} G_\infty
\end{equation}
Since $\left\Vert \mathcal{V} \right\Vert_{\textnormal{op}} \leq
C_d^\prime \ell^{-1} e^{-\mu_0} \beta_0^{-\frac{d+1}{2}}$, as soon
as $\ell^{-1} e^{-\mu_0} \beta_0^{-\frac{d+1}{2}}$ is sufficiently
small we can invert this equation to give
\begin{equation}
\label{eq:s11-boltz8}
G_\infty = \left(\mathcal{I}-\mathcal{V}\right)^{-1} G_\infty (0)
= \sum_{j=0}^\infty \mathcal{V}^j G_\infty (0)
\end{equation}
which is the unique solution of the Boltzmann hierarchy.
\qed
\end{proof}

\begin{remark}
We cannot apply the above argument, as written, in the case $d=2$;
this is due to the failure of integrability at large times.
However, this is a technical restriction since 
Theorem \ref{thm:s7-ip}
gives us \emph{a priori} bounds for the BBGKY hierarchy, independent
of $N$, for all $d\geq 2$.
Indeed, a slightly different argument from the one above
 actually implies that Theorem \ref{thm:s11-ip1} 
holds when $d=2$ (see \cite{IP1986}); note that the only
difference in their proof was that while they could not show that
$\sum_j \left\Vert \mathcal{V}\right\Vert_{\textnormal{op}}^j < \infty$,
they could at least prove that
$\sum_j \left\Vert \mathcal{V}^j G_\infty (0)\right\Vert < \infty$,
under the same assumptions. Alternatively, for chaotic data, we can use 
the solvability of the Boltzmann equation near vacuum 
(see \cite{CIP1994}), combined with the local well-posedness 
of the Boltzmann hierarchy; this line of reasoning would still be 
completely sufficient to reach the conclusions of
Theorem \ref{thm:s13-conv1} in the case $d=2$.
\end{remark}

To conclude this section, we quote a couple of
 local-in-time well-posedness results for the Boltzmann hierarchy. 
The proofs are well-known and similar to the proof presented above.

\begin{theorem}
\label{thm:s11-ip2}
Suppose $F_\infty (0) = 
\left\{ f_\infty^{(s)} (0)\right\}_{s\in\mathbb{N}}$ is a
sequence of functions such that each 
$f_\infty^{(s)}(0):\mathbb{R}^{2ds}\rightarrow\mathbb{R}$
is continuous and symmetric, and for some $\beta_0>0,\mu_0\in\mathbb{R}$,
\begin{equation}
\label{eq:s11-boltz9}
\sup_{s\in\mathbb{N}} \sup_{Z_s\in\mathbb{R}^{2ds}}
\left| f_\infty^{(s)} (0,Z_s)\right| 
e^{\beta_0 \left[ E_s(Z_s)+I_s(Z_s)\right]}e^{\mu_0 s} \leq 1
\end{equation}
Then there is a constant $C_d > 0$, depending only on $d$, such
that if $T_L < C_d \ell e^{\mu_0} \beta_0^{\frac{d+1}{2}}$, then
there exists a unique sequence
$F_\infty (t) = \left\{ f_\infty^{(s)} (t)\right\}_{s\in\mathbb{N}}$, 
with each
$f_\infty^{(s)}(t,Z_s):[0,T_L]\times\mathbb{R}^{2ds}\rightarrow\mathbb{R}$
continuous and symmetric, such that
\begin{equation}
\label{eq:s11-boltz10}
\sup_{0\leq t \leq T_L} \sup_{s\in\mathbb{N}} \sup_{Z_s\in\mathbb{R}^{2ds}}
\left| f_\infty^{(s)} (t,Z_s)\right| 
e^{\frac{1}{2} \beta_0
 \left[ E_s(Z_s)+I_s((X_s-V_s t,V_s))\right]}
e^{(\mu_0-1) s} \leq 2
\end{equation}
and for each $s\in\mathbb{N}$ there holds
\begin{equation}
\label{eq:s11-boltz11}
\left(\frac{\partial}{\partial t} + V_s\cdot \nabla_{X_s}\right)
f_\infty^{(s)} (t,Z_s) = \ell^{-1} C_{s+1}^0 f_\infty^{(s+1)} (t,Z_s)
\end{equation}
in the sense of distributions, for $0\leq t \leq T_L$.
\end{theorem}

\begin{theorem} 
\label{thm:s11-lwp}
Suppose $F_\infty(0) = 
\left\{ f_\infty^{(s)} (0)\right\}_{s\in\mathbb{N}}$ is a 
sequence of functions such that each 
$f_\infty^{(s)}(0):\mathbb{R}^{2ds}\rightarrow\mathbb{R}$
is continuous and symmetric, and for some $\beta_0>0,\mu_0\in\mathbb{R}$,
\begin{equation}
\label{eq:s11-boltz12}
\sup_{s\in\mathbb{N}} \sup_{Z_s\in\mathbb{R}^{2ds}}
\left| f_\infty^{(s)} (0,Z_s)\right| 
e^{\beta_0 E_s(Z_s) }e^{\mu_0 s} \leq 1
\end{equation}
Then there is a constant $C_d > 0$, depending only on $d$, such that
if $T_L < C_d \ell e^{\mu_0} \beta_0^{\frac{d+1}{2}}$, then
there exists a unique sequence
$F_\infty (t) = \left\{ f_\infty^{(s)} (t)\right\}_{s\in\mathbb{N}}$, 
with each
$f_\infty^{(s)}(t,Z_s):[0,T_L]\times\mathbb{R}^{2ds}\rightarrow\mathbb{R}$
continuous and symmetric, such that
\begin{equation}
\label{eq:s11-boltz13}
\sup_{0\leq t \leq T_L} \sup_{s\in\mathbb{N}} \sup_{Z_s\in\mathbb{R}^{2ds}}
\left| f_\infty^{(s)} (t,Z_s)\right| 
e^{\frac{1}{2} \beta_0 E_s(Z_s)} e^{(\mu_0-1) s} \leq 2
\end{equation}
and for each $s\in\mathbb{N}$ there holds
\begin{equation}
\label{eq:s11-boltz14}
\left(\frac{\partial}{\partial t} + V_s\cdot \nabla_{X_s}\right)
f_\infty^{(s)} (t,Z_s) = \ell^{-1} C_{s+1}^0 f_\infty^{(s+1)} (t,Z_s)
\end{equation}
in the sense of distributions, for $0\leq t \leq T_L$.
\end{theorem}

\section{Construction of the initial data}
\label{sec:12}

We introduce the $N$-particle density $f_N$ 
\begin{equation}
\label{eq:s12-data}
f_N ( 0,Z_N) = \mathcal{Z}_N^{-1}
\mathbf{1}_{Z_N \in \mathcal{D}_N} f_0^{\otimes N} (Z_N)
\end{equation}
where $\mathcal{Z}_N$ is the partition function,
\begin{equation}
\label{eq:s12-ZZ1}
\mathcal{Z}_N = \int_{\mathbb{R}^{2dN}}
\mathbf{1}_{Z_N \in \mathcal{D}_N} f_0^{\otimes N} (Z_N) dZ_N
\end{equation}
We also use the notation $\mathcal{Z}_s$ for $1\leq s \leq N$
(note carefully the implicit dependence on $\varepsilon$),
\begin{equation}
\label{eq:s12-ZZ2}
\mathcal{Z}_s = \int_{\mathbb{R}^{2ds}}
\mathbf{1}_{Z_s \in \mathcal{D}_s} f_0^{\otimes s} (Z_s) dZ_s
\end{equation}
The proofs in this section are almost identical to those in the
literature; we include them for the sake of completeness. \cite{GSRT2014}

\begin{lemma}
\label{lemma:s12-ZZ3}
For $1\leq s < N$, and any probability density $f_0 (x,v)$ on
$\mathbb{R}^{2d}$ with
$ f_0 \in L^\infty_x L^1_v$, in the
Boltzmann-Grad scaling $N\varepsilon^{d-1}=\ell^{-1}$ there holds
\begin{equation}
\label{eq:s12-ZZ4}
\mathcal{Z}_{s+1} \geq \mathcal{Z}_s \left( 1 - \ell^{-1} |B_1^d| 
\left\Vert f_0 \right\Vert_{L^\infty_x L^1_v} \varepsilon \right)
\end{equation} 
where $B_1^d$ is the unit ball in $\mathbb{R}^d$ and $\mathcal{Z}_s$
is given by (\ref{eq:s12-ZZ2}).
\end{lemma}
\begin{proof}
For $1 \leq s < N$, we have
\begin{equation*}
\begin{aligned}
\mathcal{Z}_{s+1} 
& = \int_{\mathbb{R}^{2d(s+1)}}
\mathbf{1}_{Z_{s+1} \in \mathcal{D}_{s+1}}
f_0^{\otimes (s+1)} (Z_{s+1}) dZ_{s+1} \\
& = \int_{\mathbb{R}^{2d(s+1)}}
\mathbf{1}_{Z_{s} \in \mathcal{D}_s}
\left( \prod_{i=1}^s \mathbf{1}_{|x_i - x_{s+1}|>\varepsilon} \right)
f_0^{\otimes (s+1)} (Z_{s+1}) dZ_{s+1} \\
& = \int_{\mathbb{R}^{2ds}}
\mathbf{1}_{Z_{s} \in \mathcal{D}_s} 
 \left[\int_{\mathbb{R}^{2d}}
f_0 (z_{s+1})  \left(\prod_{i=1}^s 
\mathbf{1}_{|x_i - x_{s+1}|>\varepsilon}\right) dz_{s+1} \right]
f_0^{\otimes s} (Z_s) dZ_s \\
\end{aligned}
\end{equation*}
We bound the quantity in brackets from below, uniformly
in $Z_s$.
\begin{equation*}
\begin{aligned}
& \int_{ \mathbb{R}^{2d}} 
f_0 (z_{s+1})\left(\prod_{i=1}^s 
\mathbf{1}_{|x_i - x_{s+1}|>\varepsilon}\right) dz_{s+1}  \\
&\qquad \qquad \geq  \int_{\mathbb{R}^{2d}}
f_0 (z_{s+1}) \left(1-\sum_{i=1}^s 
\mathbf{1}_{|x_i - x_{s+1}|\leq \varepsilon}\right) dz_{s+1}  \\
& \qquad \qquad \geq  1 -
 s  \varepsilon^d |B_1^d|
 \left\Vert f_0 \right\Vert_{L^\infty_x L^1_v} \\
& \qquad \qquad \geq  1 - N \varepsilon^{d-1} |B_1^d|
 \left\Vert f_0 \right\Vert_{L^\infty_x L^1_v} \varepsilon \\
& \qquad \qquad = 1 - \ell^{-1} |B_1^d|
 \left\Vert f_0 \right\Vert_{L^\infty_x L^1_v} \varepsilon \\
\end{aligned}
\end{equation*}
We have used the Boltzmann-Grad scaling
$N \varepsilon^{d-1} =\ell^{-1}$
in the last step. Finally we are able to conclude, for $1 \leq s < N$,
\begin{equation}
\label{eq:s12-ZZ5}
\mathcal{Z}_{s+1} \geq \mathcal{Z}_s \left( 1 - \ell^{-1} 
|B_1^d|\left\Vert f_0 \right\Vert_{L^\infty_x L^1_v} \varepsilon \right)
\end{equation} 
as claimed.
\qed
\end{proof}

\begin{lemma}
\label{lemma:s12-ZZ6}
For $1\leq s < N$, and any probability density $f_0 (x,v)$ on
$\mathbb{R}^{2d}$ with $f_0 \in L^\infty_x L^1_v$, in the
Boltzmann-Grad scaling $N\varepsilon^{d-1}=\ell^{-1}$ there holds
\begin{equation}
\label{eq:s12-ZZ7}
1 \leq  \mathcal{Z}_N^{-1} \mathcal{Z}_{N-s} \leq
\left( 1 - \ell^{-1} |B_1^d|
\left\Vert f_0 \right\Vert_{L^\infty_x L^1_v} \varepsilon \right)^{-s}
\end{equation}
where $B_1^d$ is the unit ball in $\mathbb{R}^d$ and $\mathcal{Z}_s$
is given by (\ref{eq:s12-ZZ2}).
\end{lemma}
\begin{proof}
For the first inequality, we note that clearly
$\mathcal{Z}_N \leq \mathcal{Z}_s \mathcal{Z}_{N-s}$, then
use the fact that $\mathcal{Z}_s \leq 1$.
 The second inequality follows directly
from Lemma \ref{lemma:s12-ZZ3} by induction on $s$.
\qed
\end{proof}

\begin{lemma}
\label{lemma:s12-ZZ8}
For $1\leq s \leq N$, and any probability density $f_0 (x,v)$ on
$\mathbb{R}^{2d}$ with $f_0 \in L^\infty_x L^1_v$, in the
Boltzmann-Grad scaling $N\varepsilon^{d-1}=\ell^{-1}$ there holds
\begin{equation}
\label{eq:s12-f0-1}
f_N^{(s)} (0,Z_s) \leq
\mathbf{1}_{Z_s \in \mathcal{D}_s} f_0^{\otimes s} (Z_s)
\left( 1 - \ell^{-1} |B_1^d|
\left\Vert f_0 \right\Vert_{L^\infty_x L^1_v} \varepsilon \right)^{-s} 
\end{equation}
where $B_1^d$ is the unit ball in $\mathbb{R}^d$ and $f_N^{(s)}(0)$ is
the marginal of the data $f_N (0)$ given by (\ref{eq:s12-data}).
\end{lemma}
\begin{proof}
We proceed by computation.
\begin{equation*}
\begin{aligned}
f_N^{(s)} (0,Z_s) & = \int_{\mathbb{R}^{2d(N-s)}}
\mathcal{Z}_N^{-1} \mathbf{1}_{Z_N \in \mathcal{D}_N}
f_0^{\otimes N} (0,Z_N) dZ_{(s+1):N} \\
& \leq \int_{\mathbb{R}^{2d(N-s)}}
\mathcal{Z}_N^{-1}
\mathbf{1}_{Z_s \in \mathcal{D}_s}
\mathbf{1}_{Z_{(s+1):N} \in \mathcal{D}_{N-s}}
f_0^{\otimes N} (0,Z_N) dZ_{(s+1):N} \\
& = \mathcal{Z}_N^{-1} \mathcal{Z}_{N-s} 
\mathbf{1}_{Z_s \in \mathcal{D}_s} f_0^{\otimes s} (Z_s)
\end{aligned}
\end{equation*}
Then the result follows from Lemma \ref{lemma:s12-ZZ6}.
\qed
\end{proof}

\begin{lemma}
\label{lemma:s12-ZZ9}
For $1\leq s \leq N$, and any probability density $f_0 (x,v)$ on
$\mathbb{R}^{2d}$ with $f_0 \in L^\infty_x L^1_v$, in the
Boltzmann-Grad scaling $N\varepsilon^{d-1}=\ell^{-1}$ there holds
\begin{equation}
\label{eq:s12-f0-2}
f_N^{(s)} (0,Z_s) \geq
\mathbf{1}_{Z_s \in \mathcal{D}_s}
f_0^{\otimes s} (Z_s) \left( 1 - (s+1) \ell^{-1} |B_1^d|
\left\Vert f_0 \right\Vert_{L^\infty_x L^1_v} \varepsilon \right)
\end{equation}
where $B_1^d$ is the unit ball in $\mathbb{R}^d$ and $f_N^{(s)}(0)$ is
the marginal of the data $f_N (0)$ given by (\ref{eq:s12-data}).
\end{lemma}
\begin{proof}
We proceed by computation.
\begin{equation*}
\begin{aligned}
&f_N^{(s)} (0,Z_s) =
\int_{\mathbb{R}^{2d(N-s)}} \mathcal{Z}_N^{-1}
\mathbf{1}_{Z_N \in \mathcal{D}_N} f_0^{\otimes N} (Z_N) dZ_{(s+1):N} \\
& = \int_{\mathbb{R}^{2d(N-s)}} \mathcal{Z}_N^{-1}
\mathbf{1}_{Z_s \in \mathcal{D}_s}
\mathbf{1}_{Z_{(s+1):N} \in \mathcal{D}_{N-s} }  \times \\
& \qquad \qquad \qquad
 \times \left( \prod_{1\leq i \leq s} \prod_{s < j \leq N}
\mathbf{1}_{|x_i-x_j|>\varepsilon} \right) f_0^{\otimes N}(Z_N) dZ_{(s+1):N}\\
& = \mathcal{Z}_N^{-1} \mathbf{1}_{Z_s \in \mathcal{D}_s}
f_0^{\otimes s} (Z_s) \int_{\mathbb{R}^{2d(N-s)}}
\mathbf{1}_{Z_{(s+1):N}\in\mathcal{D}_{N-s}} \times \\
& \qquad \qquad \qquad 
\times \left(\prod_{1\leq i \leq s}\prod_{s < j \leq N}
\mathbf{1}_{|x_i-x_j|>\varepsilon}\right)
f_0^{\otimes (N-s)} (Z_{(s+1):N}) dZ_{(s+1):N}
\end{aligned}
\end{equation*}
Now observe that
\begin{equation}
\label{eq:s12-exclusion}
\prod_{1\leq i \leq s} \prod_{s < j \leq N}
\mathbf{1}_{|x_i-x_j|>\varepsilon} \geq
1 - \sum_{1\leq i \leq s} \sum_{s < j \leq N} 
\mathbf{1}_{|x_i-x_j| \leq \varepsilon}
\end{equation}
Then again, for $1 \leq i \leq s$, $s < j \leq N$, we have
\begin{equation}
\label{eq:s12-f0-3}
\begin{aligned}
\int_{\mathbb{R}^{2d(N-s)}}
\mathbf{1}_{Z_{(s+1):N}\in\mathcal{D}_{N-s}}
\mathbf{1}_{|x_i-x_j|\leq \varepsilon} 
& f_0^{\otimes (N-s)} (Z_{(s+1):N}) dZ_{(s+1):N} \leq \\
& \qquad  \leq  \mathcal{Z}_{N-s-1} \varepsilon^d
|B_1^d| \left\Vert f_0 \right\Vert_{L^\infty_x L^1_v}
\end{aligned}
\end{equation}
Therefore,
\begin{equation}
\label{eq:s12-f0-4}
\begin{aligned}
f_N^{(s)} (0,Z_s) \geq
\mathcal{Z}_N^{-1} & \mathbf{1}_{Z_s \in \mathcal{D}_s}
 f_0^{\otimes s}(Z_s) \times \\
& \times \left[ \mathcal{Z}_{N-s} - 
s (N-s) \mathcal{Z}_{N-s-1} \varepsilon^d
|B_1^d| \left\Vert f_0 \right\Vert_{L^\infty_x L^1_v} \right]
\end{aligned}
\end{equation}
We use Lemma \ref{lemma:s12-ZZ3},
 Lemma \ref{lemma:s12-ZZ6}, and the Boltzmann-Grad scaling
$N\varepsilon^{d-1} = \ell^{-1}$ to conclude
\begin{equation}
\label{eq:s12-f0-5}
f_N^{(s)} (0,Z_s) \geq
\mathbf{1}_{Z_s \in \mathcal{D}_s}
f_0^{\otimes s} (Z_s) \left( 1 - (s+1) \ell^{-1} |B_1^d|
\left\Vert f_0 \right\Vert_{L^\infty_x L^1_v} \varepsilon \right)
\end{equation}
\qed
\end{proof}

\begin{corollary}
\label{cor:s12-f0-6}
For any probability density $f_0 (x,v)>0$
on $\mathbb{R}^{2d}$ with $f_0 \in L^\infty_x L^1_v$,
 in the Boltzmann-Grad scaling $N\varepsilon^{d-1}=\ell^{-1}$,
if $N$ is sufficiently large, then simultaneously
for all $1\leq s \leq N$ there holds
\begin{equation}
\label{eq:s12-f0-7}
 \left\Vert \mathbf{1}_{Z_s \in \mathcal{D}_s} \left(
\frac{f_N^{(s)} (0,Z_s)}{f_0^{\otimes s}(Z_s)} - 1 \right)
\right\Vert_{L^\infty_{Z_s}} \leq
\left[\left(1-\ell^{-1} |B_1^d|
\left\Vert f_0 \right\Vert_{L^\infty_x L^1_v} \varepsilon\right)^{-(s+1)}-1
\right]
\end{equation}
where $f_N^{(s)}(0)$ is the marginal of the data $f_N (0)$ given
by (\ref{eq:s12-data}).
\end{corollary}

\begin{corollary}
\label{cor:s12-f0-8}
Let $f_0$ be a probability density on $\mathbb{R}^{2d}$ with 
\begin{equation}
\label{eq:s12-f0-9}
\left\Vert f_0 (x,v) e^{\mu} e^{\frac{1}{2} \beta |v|^2}
\right\Vert_{L^\infty_{x,v}} \leq 1
\end{equation}
 for some 
$\beta > 0 , \mu \in \mathbb{R}$. Then for any $\mu^\prime < \mu$ we
have for all sufficiently large $N$ in the Boltzmann-Grad scaling
$N\varepsilon^{d-1} = \ell^{-1}$ the estimate
\begin{equation}
\label{eq:s12-f0-10}
\sup_{1\leq s \leq N} \sup_{Z_s \in \mathcal{D}_s}
\left| f_N^{(s)} (0,Z_s) \right| e^{\beta E_s (Z_s)} e^{\mu^\prime s}
\leq 1
\end{equation}
where $f_N^{(s)}(0)$ is the marginal of the data $f_N (0)$ given
by (\ref{eq:s12-data}).
\end{corollary}

\section{Local-in-time convergence proof}
\label{sec:13}

The main result of this section is a local-in-time propagation of
chaos result for the BBGKY hierarchy. We will use the
stability result from Section \ref{sec:9} in order to prove
uniform convergence on a set of ``good'' phase points. 

\begin{theorem} 
\label{thm:s13-conv1}
Suppose $F_N (t) = \left\{ f_N^{(s)} (t)\right\}_{1\leq s \leq N}$
is a solution of the BBGKY hierarchy (\ref{eq:s4-def-bbgky}),
subject to the Boltzmann-Grad scaling $N \varepsilon^{d-1}=\ell^{-1}$,
and with each 
function $f_N^{(s)} : [0,\infty) \times \mathbb{R}^{2ds} 
\rightarrow \mathbb{R}$ symmetric under particle interchange.
Further suppose $F_\infty (0) = \left\{ f_\infty^{(s)} (0)
\right\}_{s\in\mathbb{N}}$ is a
sequence of functions such that each 
$f_\infty^{(s)}(0):\mathbb{R}^{2ds}\rightarrow\mathbb{R}$
is continuous and symmetric. Assume that
for some $\beta_0 > 0$, $\mu_0 \in \mathbb{R}$,
\begin{equation}
\label{eq:s13-bbgky1}
\sup_{1\leq s \leq N} \sup_{Z_s \in \mathcal{D}_s}
\left| f_N^{(s)} (0,Z_s)\right| 
e^{\beta_0 E_s (Z_s)} e^{\mu_0 s} \leq 1
\end{equation}
\begin{equation}
\label{eq:s13-boltz1}
\sup_{s\in\mathbb{N}} \sup_{Z_s\in\mathbb{R}^{2ds}}
\left| f_\infty^{(s)} (0,Z_s)\right| 
e^{\beta_0 E_s(Z_s)}e^{\mu_0 s} \leq 1
\end{equation}
Then there is a constant $C_d > 0$, depending only on $d$, such that
if  $T_L < C_d \ell e^{\mu_0} \beta_0^{\frac{d+1}{2}}$, then all of
the following are true: \\
(i) $F_N (t)$ satisfies the bound
\begin{equation}
\label{eq:s13-bbgky2}
\sup_{0\leq t \leq T_L} \sup_{1\leq s \leq N} \sup_{Z_s \in \mathcal{D}_s}
\left| f_N^{(s)} (t,Z_s)\right| 
e^{\frac{1}{2} \beta_0 E_s (Z_s)} e^{(\mu_0 -1)s} \leq 1
\end{equation}
(ii) the Boltzmann hierarchy has a unique continuous symmetric solution 
$F_\infty (t)$, $t\in [0,T_L]$, satisfying the bound
\begin{equation}
\label{eq:s13-boltz2}
\sup_{0 \leq t \leq T_L} \sup_{s\in\mathbb{N}} 
\sup_{Z_s \in \mathbb{R}^{2ds}}
\left| f_\infty^{(s)} (t,Z_s)\right| 
e^{\frac{1}{2} \beta_0 E_s (Z_s)} e^{(\mu_0 -1)s} \leq 2
\end{equation}
(iii) if $f_\infty^{(s)} (0) = f_0^{\otimes s} \;
\forall \; s\in\mathbb{N} $ for some
Lipschitz-continuous probability density $f_0 (x,v)$, and likewise
$\left\{ \left\{ f_N^{(s)} (0) \right\}_{1\leq s \leq N}
\right\}_{N\in\mathbb{N}}$ is nonuniformly $f_0$-chaotic (see
Section \ref{sec:2}), then
$f_\infty^{(s)} (t) = f_t^{\otimes s} \; \forall \; s\in\mathbb{N}$
for $t\in [0,T_L]$ where $f_t$ solves Boltzmann's equation, and 
$\left\{ \left\{ f_N^{(s)} (t) \right\}_{1\leq s \leq N}\right\}_{N\in\mathbb{N}}$
is nonuniformly $f_t$-chaotic for $t\in [0,T_L]$.
\end{theorem}

\begin{proof}
The local well-posedness of the Boltzmann hierarchy, and the bounds
(\ref{eq:s13-bbgky2}-\ref{eq:s13-boltz2}), are
direct consequences of Theorem \ref{thm:s6-lanford}
and Theorem \ref{thm:s11-lwp}. 

We introduce a smooth cut-off function 
$\chi : [0,\infty)\rightarrow \mathbb{R}$, decreasing, with 
$0\leq \chi \leq 1$, $\chi (z) = 1$ for $0\leq z \leq 1$, 
$\left\Vert \chi^\prime \right\Vert_\infty \leq 2$,
 and $\chi (z) = 0$ for $z\geq 2$.
Given parameters $R>0$ and $n\in \mathbb{N}$, we define
\begin{equation}
\label{eq:s13-bbgky3}
f_{N,n,R}^{(s)} (0,Z_s) = 
f_N^{(s)} (0,Z_s) \mathbf{1}_{1 \leq s \leq n}
\chi \left( \frac{1}{R^2} E_s (Z_s)\right)
\end{equation}
and let 
$F_{N,n,R}(0) = \left\{ f_{N,n,R}^{(s)}(0)\right\}_{1\leq s \leq N}$.
We let $F_{N,n,R}(t)$ be the solution of the BBGKY hierarchy
(\ref{eq:s4-def-bbgky}) with initial data $F_{N,n,R}(0)$.
Similarly, given initial data 
$F_\infty (0) = \left\{ f_\infty^{(s)} (0)\right\}_{s\in\mathbb{N}}$,
define
\begin{equation}
\label{eq:s13-boltz3}
f_{\infty,n,R}^{(s)} (0,Z_s) = f_\infty^{(s)} (0,Z_s)
\mathbf{1}_{1\leq s \leq n} \chi\left(\frac{1}{R^2} E_s (Z_s)\right)
\end{equation}
and let 
$F_{\infty,n,R} (0)=\left\{f_{\infty,n,R}^{(s)}(0)\right\}_{s\in\mathbb{N}}$.
We let $F_{\infty,n,R}(t)$ be the solution of the Boltzmann hierarchy
with data $F_{\infty,n,R}(0)$. Using Theorem
\ref{thm:s6-lanford} and Theorem \ref{thm:s11-lwp},
and the linearity of the BBGKY and Boltzmann hierarchies,
and dividing $C_d$ by $e\cdot 2^{\frac{d+1}{2}}$ in the 
statement of the theorem, we immediately obtain the following estimates:
\begin{equation}
\label{eq:s13-conv3}
\underset{\substack{1 \leq s \leq N \\
t \in [0,T_L] \\ Z_s \in \mathcal{D}_s}}{\sup}
\left| \left( f_N^{(s)}- f_{N,n,R}^{(s)}\right) (t,Z_s) \right|
e^{\frac{1}{4} \beta_0 E_s (Z_s)} e^{(\mu_0 - 2)s} \leq
e^{-\frac{1}{2} \beta_0 R^2} + e^{-n}
\end{equation}
\begin{equation}
\label{eq:s13-conv4}
\underset{\substack{s\in\mathbb{N} \\
t \in [0,T_L] \\ Z_s \in \mathbb{R}^{2ds}}}{\sup}
 \left| \left( f_\infty^{(s)} - f_{\infty,n,R}^{(s)} \right)(t,Z_s) \right|
e^{\frac{1}{4} \beta_0 E_s (Z_s)} e^{(\mu_0 - 2)s} \leq
2 \left( e^{-\frac{1}{2} \beta_0 R^2} + e^{-n} \right)
\end{equation}
The remainder of the proof consists of comparing the two
functions $f_{N,n,R}^{(s)}(t)$ and $f_{\infty,n,R}^{(s)}(t)$.

We have the following Duhamel series:
\begin{equation}
\label{eq:s13-series1}
\begin{aligned}
& f_{N,n,R}^{(s)} (t,Z_s) = \sum_{k=0}^{n-s} a_{N,k,s}\times \\
& \times \sum_{i_1 = 1}^s \dots \sum_{i_k = 1}^{s+k-1}
\int_0^t \dots \int_0^{t_{k-1}}
\int_{\mathbb{R}^{dk}} \int_{\left(\mathbb{S}^{d-1}\right)^k} 
\left( \prod_{m=1}^k d\omega_m dv_{s+m} dt_m \right) \times \\
& \times \left( b_{s,s+k}\left[\cdot\right]
f_{N,n,R}^{(s+k)} (0,Z_{s,s+k}\left[\cdot\right])\right)
 \left[ Z_s,t;\left\{t_j,v_{s+j},\omega_j,i_j
\right\}_{j=1}^k \right]
\end{aligned}
\end{equation}
\begin{equation}
\label{eq:s13-series2}
\begin{aligned}
& f_{\infty,n,R}^{(s)} (t,Z_s) = \sum_{k=0}^{n-s} \ell^{-k} \times \\
& \times \sum_{i_1 = 1}^s \dots \sum_{i_k = 1}^{s+k-1}
\int_0^t \dots \int_0^{t_{k-1}}
\int_{\mathbb{R}^{dk}} \int_{\left(\mathbb{S}^{d-1}\right)^k} 
\left( \prod_{m=1}^k d\omega_m dv_{s+m} dt_m \right) \times \\
& \times \left(b_{s,s+k}^0 \left[\cdot\right] f_{\infty,n,R}^{(s+k)} (0,
Z_{s,s+k}^0 \left[\cdot\right])\right)\left[ Z_s,t;
\left\{t_j,v_{s+j},\omega_j,i_j\right\}_{j=1}^k\right]
\end{aligned}
\end{equation}
where
\begin{equation}
\label{eq:s13-a-N-k-s}
a_{N,k,s} = \frac{(N-s)!}{(N-s-k)!} \varepsilon^{k (d-1)}
\end{equation}
It is not hard to show that all terms appearing in the \emph{finite} series 
(\ref{eq:s13-series1}-\ref{eq:s13-series2})
are finite for all $t\geq 0$. Note that the expression
(\ref{eq:s13-series1}) is meaningful as a measurable
function if the data is integrable and compactly supported
(see \cite{IP1986} for a detailed proof of this fact), whereas the expression
(\ref{eq:s13-series2}) makes sense due to
the continuity of the data $F_{\infty,n,R} (0)$.

Let us now define a new function, $\tilde{f}_{N,n,R}^{(s)}(t)$, which is
closely related to $f_{N,n,R}^{(s)} (t)$.
\begin{equation}
\label{eq:s13-series3}
\begin{aligned}
& \tilde{f}_{N,n,R}^{(s)} (t,Z_s) = \sum_{k=0}^{n-s} \ell^{-k}\times \\
& \times \sum_{i_1 = 1}^s \dots \sum_{i_k = 1}^{s+k-1}
\int_0^t \dots \int_0^{t_{k-1}}
\int_{\mathbb{R}^{dk}} \int_{\left(\mathbb{S}^{d-1}\right)^k} 
\left( \prod_{m=1}^k d\omega_m dv_{s+m} dt_m\right) \times \\
& \times \left( b_{s,s+k} \left[\cdot\right] f_{N,n,R}^{(s+k)} (0,
Z_{s,s+k} \left[ \cdot \right])\right)
\left[ Z_s,t;
\left\{t_j,v_{s+j},\omega_j,i_j\right\}_{j=1}^k\right]
\end{aligned}
\end{equation}
Note that $\left| a_{N,k,s}-\ell^{-k}\right| \leq
\left[ 1 - \left(1-\frac{n}{N}\right)^n\right] \ell^{-k}$ for
$0\leq k \leq n-s$; therefore,
\begin{equation}
\label{eq:s13-conv5}
\begin{aligned}
& \left| \tilde{f}_{N,n,R}^{(s)} (t,Z_s) - f_{N,n,R}^{(s)}(t,Z_s)\right|
 \leq \left[1-\left(1-\frac{n}{N}\right)^n\right]\times \\
& \sum_{k=0}^{n-s} \sum_{i_1 = 1}^s \dots \sum_{i_k = 1}^{s+k-1}
\ell^{-k} \int_0^t \dots \int_0^{t_{k-1}}
\int_{\mathbb{R}^{dk}} \int_{\left(\mathbb{S}^{d-1}\right)^k} 
\left( \prod_{m=1}^k d\omega_m dv_{s+m} dt_m\right) \times \\
& \times \left(\left| b_{s,s+k} \left[\cdot\right] \right| 
\left| f_{N,n,R}^{(s+k)} (0,Z_{s,s+k} [\cdot])\right| \right)
\left[ Z_s,t;\left\{t_j,v_{s+j},\omega_j,i_j\right\}_{j=1}^k\right]
\end{aligned}
\end{equation}
To estimate the series in (\ref{eq:s13-conv5}), we recall that
$f_{N,n,R}^{(s+k)} (0)$ is absolutely bounded by
$e^{-\mu_0 (s+k)}$ and is supported in the set $E_{s+k} (Z_{s+k}) \leq 2 R^2$. 
Hence, due to energy conservation, all the iterated integrals appearing
in (\ref{eq:s13-conv5}) range over compact sets and we
can evaluate the maximum possible contributions explicitly.
Note that this is a significant over-estimate since we are not using
the exponential decay of $f_{N,n,R}^{(s+k)} (0)$ at large energies;
nevertheless, this crude estimate will suffice for the proof. We obtain
\begin{equation}
\label{eq:s13-conv5-1}
\begin{aligned}
& \left| \tilde{f}_{N,n,R}^{(s)} (t,Z_s) - f_{N,n,R}^{(s)}(t,Z_s)\right|
 \leq \\
& \qquad \qquad \qquad
\leq \left[1-\left(1-\frac{n}{N}\right)^n\right] e^{-\mu_0 s}
\exp\left[ C_d \ell^{-1} n R^{d+1} e^{-\mu_0} t\right]
\end{aligned}
\end{equation}
Observe that the right-hand side of (\ref{eq:s13-conv5-1}) tends to
zero as $N\rightarrow \infty$ when $n,R,Z_s,t$ are all held fixed.

Let us now fix $Z_s \in \mathcal{K}_s \cap \mathcal{U}_s^\eta$,
$t\in [0,T_L]$, with $E_s (Z_s) \leq 2 R^2$. Let us pick parameters
$\eta,\theta,\alpha,y>0$ such that $R>\eta$ and
$\sin \theta > c_d y^{-1} \varepsilon$, where $c_d$ is as in the
statement of Proposition \ref{prop:s9-stability}. Let us define
\begin{equation}
\label{eq:s13-A-n-R}
\mathcal{A}_{n,R}
=\sum_{k=0}^n C_d^k \ell^{-k} R^{k(d+1)} n^k e^{-\mu_0 k} T_L^k
\end{equation}
where the constant $C_d$ is to be chosen in the next step.
Then, by repeated
application of Proposition \ref{prop:s9-stability}, we can construct
sets $\left\{ \mathcal{B}_k \right\}_{k=0}^{n-s}$, dependent on
$(Z_s,t)$, with
\begin{equation}
\label{eq:s13-B-k}
\mathcal{B}_k \subset \left( [0,\infty) \times \mathbb{R}^d
\times \mathbb{S}^{d-1} \times \mathbb{N}\right)^k
\end{equation} 
such that
\begin{equation}
\label{eq:s13-conv7}
\begin{aligned}
& \sum_{k=0}^{n-s} \sum_{i_1 = 1}^s \dots \sum_{i_k = 1}^{s+k-1}
\ell^{-k} \int_0^t \dots \int_0^{t_{k-1}}
\int_{\left(B_{2R}^d\right)^k} \int_{\left(\mathbb{S}^{d-1}\right)^k} 
\mathbf{1}_{\mathcal{B}_k}
\prod_{m=1}^k d\omega_m dv_{s+m} dt_m  \times \\
& \times \left(\left| b_{s,s+k} \left[\cdot\right] \right| 
\left| f_{N,n,R}^{(s+k)} (0,Z_{s,s+k} [\cdot] \right| \right)
\left[ Z_s,t;
\left\{t_j,v_{s+j},\omega_j,i_j\right\}_{j=1}^k\right] \leq \\
& \leq e^{-\mu_0 s} n^2 \mathcal{A}_{n,R} \left[ \alpha + \frac{y}{\eta T_L}
+ C_{d,\alpha} \left( \left(\frac{\eta}{R}\right)^{d-1} +
 \theta^{(d-1)/2}\right) \right]
\end{aligned}
\end{equation}
\begin{equation}
\label{eq:s13-conv8}
\begin{aligned}
& \sum_{k=0}^{n-s} \sum_{i_1 = 1}^s \dots \sum_{i_k = 1}^{s+k-1}
\ell^{-k} \int_0^t \dots \int_0^{t_{k-1}}
\int_{\left(B_{2R}^d\right)^k} \int_{\left(\mathbb{S}^{d-1}\right)^k} 
\mathbf{1}_{\mathcal{B}_k}
 \prod_{m=1}^k d\omega_m dv_{s+m} dt_m \times \\
& \times \left(\left| b_{s,s+k}^0 \left[\cdot\right]\right|
\left| f_{\infty,n,R}^{(s+k)} (0,Z_{s,s+k}^0 [\cdot])\right|\right)
\left[ Z_s,t;
\left\{t_j,v_{s+j},\omega_j,i_j\right\}_{j=1}^k\right] \leq \\
& \leq e^{-\mu_0 s} n^2 \mathcal{A}_{n,R} \left[ \alpha + \frac{y}{\eta T_L}
+ C_{d,\alpha} \left( \left(\frac{\eta}{R}\right)^{d-1} +
\theta^{(d-1)/2} \right)\right]
\end{aligned}
\end{equation}
and such that whenever
\begin{equation}
\label{eq:s13-not-B-k}
\begin{aligned}
& \left\{t_j,v_{s+j},\omega_j,i_j\right\}_{j=1}^k \\
& \qquad \;\; \in\left(\left([0,T_L]\times B_{2R}^d \times \mathbb{S}^{d-1}
\times \mathbb{N}\right)^k \backslash \mathcal{B}_k\right)\bigcap
\left\{0\leq t_k \leq \dots \leq t_1 \leq t\right\}
\end{aligned}
\end{equation}
there holds
\begin{equation}
\label{eq:s13-perturb1}
\left| \left(Z_{s,s+k} \left[\cdot\right]
-Z_{s,s+k}^0 \left[\cdot\right]\right)
\left[Z_s,t;\left\{t_j,v_{s+j},\omega_j,i_j\right\}_{j=1}^k\right]
\right|_\infty \leq k \varepsilon
\end{equation}
\begin{equation}
\label{eq:s13-perturb2}
b_{s,s+k}\left[Z_s,t;\left\{t_j,v_{s+j},\omega_j,i_j\right\}_{j=1}^k\right]
= b_{s,s+k}^0
\left[Z_s,t;\left\{t_j,v_{s+j},\omega_j,i_j\right\}_{j=1}^k \right]
\end{equation}
\begin{equation}
\label{eq:s13-perturb3}
Z_{s,s+k} \left[ Z_s,t;\left\{t_j,v_{s+j},\omega_j,i_j\right\}_{j=1}^k\right]
\in \mathcal{K}_{s+k} \cap \mathcal{U}_{s+k}^\eta
\end{equation}
Here $\left| Z_j \right|_\infty = \sup_{i=1,\dots,j}
\max\left(|x_i|,|v_i|\right)$.
\begin{remark}
The sets $\mathcal{B}_k$ collect all integration points for which
the Duhamel series (\ref{eq:s13-series2})
and (\ref{eq:s13-series3}) fail to agree. At the
remaining points, the pseudo-trajectories
$Z_{s,s+k} \left[ \dots \right]$ and $Z_{s,s+k}^0 \left[ \dots \right]$
are identical, up to $\mathcal{O}(\varepsilon)$ perturbations of
the particles' spatial positions. These perturbations are harmless
because the Boltzmann hierarchy propagates smoothness forwards in time.
\end{remark}

As long as we are away from $\mathcal{B}_k$, we can use the 
triangle inequality:
\begin{equation}
\label{eq:s13-conv9}
\begin{aligned}
&\left|\left( f_{\infty,n,R}^{(s+k)}( 0,
Z_{s,s+k}^0 \left[\cdot\right]) -
f_{N,n,R}^{(s+k)} ( 0,Z_{s,s+k}\left[\cdot\right])\right)
\left[Z_s,t;\left\{t_j,v_{s+j},\omega_j,i_j\right\}_j\right]\right|\\
& \leq 
\left| \left(f_{\infty,n,R}^{(s+k)} (0,
Z_{s,s+k}^0 \left[\cdot\right]) -
f_{\infty,n,R}^{(s+k)} ( 0,Z_{s,s+k}\left[\cdot\right])\right)
\left[Z_s,t;\left\{t_j,v_{s+j},\omega_j,i_j\right\}_j\right]\right|\\
& +
\left|\left(f_{\infty,n,R}^{(s+k)} ( 0,
Z_{s,s+k} \left[\cdot\right]) -
f_{N,n,R}^{(s+k)} ( 0,Z_{s,s+k}\left[\cdot\right]) \right)
\left[Z_s,t;\left\{t_j,v_{s+j},\omega_j,i_j\right\}_j\right]\right|\\
\end{aligned}
\end{equation}
We can easily control the first term using the regularity
assumption on $f_\infty^{(j)} (0)$ combined with the
stability estimate (\ref{eq:s13-perturb1}).
On the other hand, due to (\ref{eq:s13-perturb3}),
in order to control the second term, we only need to
estimate $\left| f_\infty^{(s+k)} - f_N^{(s+k)}\right|$ on 
$\mathcal{K}_{s+k}\cap \mathcal{U}_{s+k}^\eta$.
\begin{remark}
Carefully observe that it is entirely possible that
$Z_{s,s+k}^0 \left[ \dots \right] \notin
\mathcal{K}_{s+k} \cap \mathcal{U}_{s+k}^\eta$, even away
from $\mathcal{B}_k$. This is because in the construction of
$\mathcal{B}_k$, we never ruled out events wherein two particles
only ``barely'' miss each other under the backwards flow. 
\end{remark}

Now we easily obtain
\begin{equation}
\label{eq:s13-conv10}
\begin{aligned}
&\underset{\substack{0 \leq t \leq T_L \\ Z_s \in \mathbb{R}^{2ds}}}{\sup}
\left| \left(
f_N^{(s)}-f_\infty^{(s)}\right)(t,Z_s)\right|
\mathbf{1}_{Z_s \in \mathcal{K}_s \cap \mathcal{U}_s^\eta}
\mathbf{1}_{E_s (Z_s) \leq 2 R^2} \\
& \leq
3 e^{-(\mu_0 - 2)s} \left( e^{-\frac{1}{2} \beta_0 R^2} + e^{-n}\right)+\\
& + \left[1-\left(1-\frac{n}{N}\right)^n\right] e^{-\mu_0 s}
e^{ C_d \ell^{-1} n R^{d+1} e^{-\mu_0} T_L} + \\
& + 2 e^{-\mu_0 s} n^2 \mathcal{A}_{n,R} \left[
\alpha + \frac{y}{\eta T_L} + C_{d,\alpha} \left(\frac{\eta}{R}\right)^{d-1}
+ C_{d,\alpha} \theta^{(d-1)/2} \right] + \\
& + C_d n^{\frac{5}{2}} R^{-1} e^{|\mu_0|n}  \varepsilon
e^{C_d \ell^{-1} n R^{d+1} e^{-\mu_0} T_L}+ \\
&  + C_d n^2 \varepsilon 
e^{C_d \ell^{-1} n R^{d+1} e^{-\mu_0} T_L}
\underset{\substack{1\leq j \leq n \\ Z_j \in \mathbb{R}^{2dj}}}{\sup}
\left| \nabla_{Z_j} f_\infty^{(j)} (0,Z_j)\right|_2 
\mathbf{1}_{E_j (Z_j) \leq 2 R^2}  + \\
&  + C_d e^{C_d \ell^{-1} n R^{d+1} e^{-\mu_0} T_L}
\underset{\substack{1\leq j \leq n \\ Z_j \in \mathbb{R}^{2dj}}}{\sup}
\left|\left( f_N^{(j)}-f_\infty^{(j)}\right)(0,Z_j) \right|
\mathbf{1}_{Z_j \in \mathcal{K}_j \cap \mathcal{U}_j^\eta}
\mathbf{1}_{E_j (Z_j) \leq 2 R^2}
\end{aligned}
\end{equation}
where $\left| \nabla_{Z_s} f^{(s)} \right|_2^2 =
\sum_{i=1}^s \left( \left| \nabla_{x_i} f^{(s)} \right|^2 +
\left| \nabla_{v_i} f^{(s)} \right|^2 \right)$.
According to the definition of nonuniform $f_0$-chaoticity, we may
let $\eta = \varepsilon^\kappa$ for some fixed $\kappa \in (0,1)$.
We will then let $y = \varepsilon^{(1+\kappa)/2}$ and
$\theta \sim \varepsilon^{(1-\kappa)/4}$; in particular, the constraint
$\sin \theta \geq c_d y^{-1} \varepsilon$ is satisfied. Now let
$N\rightarrow \infty$ and $\varepsilon \rightarrow 0$ simultaneously in the
Boltzmann-Grad scaling, $N\varepsilon^{d-1}=\ell^{-1}$, and use the fact that
$f_\infty^{(j)} (0) = f_0^{\otimes j}$ and that 
$\left\{ \left\{ f_N^{(j)} (0) \right\}_{1\leq j \leq N}\right\}_{N\in
\mathbb{N}}$ is nonuniformly $f_0$-chaotic.
\begin{equation}
\label{eq:s13-conv11}
\begin{aligned}
& \limsup_{N\rightarrow \infty}
\underset{\substack{0 \leq t \leq T_L \\ Z_s \in \mathbb{R}^{2ds}}}{\sup}
\left| \left(
f_N^{(s)}-f_\infty^{(s)}\right)(t,Z_s)\right|
\mathbf{1}_{Z_s \in \mathcal{K}_s \cap \mathcal{U}_s^{\eta(\varepsilon)}}
\mathbf{1}_{E_s (Z_s) \leq 2 R^2} \\
& \qquad \qquad \qquad \leq
3 e^{-(\mu_0 - 2)s} \left( e^{-\frac{1}{2} \beta_0 R^2} + e^{-n}\right)
 + 2 e^{-\mu_0 s} n^2 \mathcal{A}_{n,R} \alpha
\end{aligned}
\end{equation}
Since $\alpha > 0$ is arbitrary we have
\begin{equation}
\label{eq:s13-conv12}
\begin{aligned}
& \limsup_{N\rightarrow \infty}
\underset{\substack{0 \leq t \leq T_L \\ Z_s \in \mathbb{R}^{2ds}}}{\sup}
\left| \left(
f_N^{(s)}-f_\infty^{(s)}\right)(t,Z_s)\right|
\mathbf{1}_{Z_s \in \mathcal{K}_s \cap \mathcal{U}_s^{\eta(\varepsilon)}}
\mathbf{1}_{E_s (Z_s) \leq 2 R^2} \\
& \qquad \qquad \qquad
\leq 3 e^{-(\mu_0 - 2)s} \left( e^{-\frac{1}{2} \beta_0 R^2} + e^{-n}\right)\\
\end{aligned}
\end{equation}
Since $n$ is arbitrary, the second term on the right-hand side can be
thrown away. On the other hand, the left-hand side only increases as
$R$ increases, so we can throw away the first term on the right-hand
side as well. Since the Boltzmann hierarchy propagates chaoticity,
we have $f_\infty^{(s)} (t) = f_t^{\otimes s}$ for $t\in [0,T_L]$; hence,
\begin{equation}
\label{eq:s13-conv13}
\begin{aligned}
& \limsup_{N\rightarrow \infty}
\underset{\substack{0 \leq t \leq T_L \\ Z_s \in \mathbb{R}^{2ds}}}{\sup}
\left| f_N^{(s)}(t,Z_s) - f_t^{\otimes s}(Z_s)\right|
\mathbf{1}_{Z_s \in \mathcal{K}_s \cap \mathcal{U}_s^{\eta(\varepsilon)}}
\mathbf{1}_{E_s (Z_s) \leq 2 R^2} = 0\\
\end{aligned}
\end{equation}
We conclude that $\left\{ \left\{ f_N^{(s)} (t) \right\}_{1\leq s \leq N}
\right\}_{N\in\mathbb{N}}$ is nonuniformly $f_t$-chaotic for $t\in[0,T_L]$.
\qed
\end{proof}

\begin{remark}
We can deduce part \emph{(i)} of Theorem \ref{thm:s2-chaos} directly from
Theorem \ref{thm:s13-conv1} by splitting the time interval
$[0,T]$ into smaller intervals $[0,T_L]$, $[T_L,2T_L]$, etc.,
for some sufficiently small time $T_L$.
\end{remark}

\appendix

\section{Proof of Part \emph{(ii)} of Theorem
\ref{thm:s2-chaos}}
\label{sec:AppA}

The proof consists of three parts. The first part is the introduction of
an unsymmetric Boltzmann-Enskog hierarchy; we show that this auxiliary 
hierarchy propagates \emph{partial} factorization. The second part is
to show that a certain class of pseudo-trajectories for the BBGKY
dynamics coincide (with high probability) with the corresponding 
pseudo-trajectories for the unsymmetric Boltzmann-Enskog hierarchy.
The third part is to add up all the sources of error pointwise, as in
Section \ref{sec:13}. We outline the proof of the first step, provide
full technical estimates for the
 second step, and
skip the third step (which is tedious yet straightforward).
We remark that a much more general version of the same result (accounting
for correlations of any finite number of particles) is currently under
investigation.\footnote{To appear, \emph{JMP Vol 58 Issue 12} -- the result and
proof presented
in this Appendix is special to two particles ($m-1=2$).}
(The  proof in the case of general $m$ is signficantly more difficult
than the two-particle case and therefore deserves a separate treatment.)
 This result and the proof were largely inspired by
the techniques of M. Pulvirenti and S. Simonella. \cite{PS2014}

\subsection{An Unsymmetric Boltzmann-Enskog Hierarchy}

We are going to construct an infinite hierarchy of equations which
tracks correlations between the first $m-1$ labeled particles
while ignoring \emph{all} correlations between the remaining
particles. Clearly, such a hierarchy cannot preserve symmetry between
all particles. Nevertheless, we will be able to prove a partial
factorization property which will be the key to part \emph{(ii)} of 
Theorem \ref{thm:s2-chaos}.
The factorization property we will prove for the resulting
hierarchy is that if $s \geq m \geq 2$ then
\begin{equation}
g_\varepsilon^{(s)} (t) =
g_\varepsilon^{(m-1)} (t) \otimes
g_{\varepsilon}(t)^{\otimes (s - m + 1)}
\end{equation}
if such factorization holds at the initial time; here $g_{\varepsilon}(t)$
is the solution to a Boltzmann-Enskog type equation.

Let us introduce the unsymmetric $s$-particle phase space, where
$m \geq 2$ is fixed and  $s\geq m-1$:
\begin{equation}
\begin{aligned}
& \tilde{\mathcal{D}}_s = \left\{ \left. Z_s = 
(X_s,V_s) \in
\mathbb{R}^{ds}\times\mathbb{R}^{ds}
 \right| \; \forall 1\leq i < j \leq m-1,
 |x_i-x_j| > \varepsilon \right\} \\
\end{aligned}
\end{equation}
Observe that in the definition of
$\tilde{\mathcal{D}}_s$, we only enforce an exclusion condition
between the first $m-1$ particles. We do \emph{not} have exclusion for
any pair of particles for which \emph{at least one} particle index is
greater than $m-1$. We define the collision operators,
\begin{equation}
\label{eq:AppA-coll-1}
\tilde{C}_{s+1} = \sum_{i=1}^s \left( 
\tilde{C}_{i,s+1}^+ - \tilde{C}_{i,s+1}^- \right)
\end{equation}
where
\begin{equation}
\label{eq:AppA-coll-2}
\begin{aligned}
& \tilde{C}_{i,s+1}^+ g_\varepsilon^{(s+1)}(t,Z_s)
 = \int_{\mathbb{R}^d} \int_{\mathbb{S}^{d-1}}
\left[ \omega \cdot (v_{s+1}-v_i)\right]_+ \times \\
&\qquad\quad \times g_\varepsilon^{(s+1)} \left(t,x_1,v_1,\dots,x_i,
v_i^*,\dots,x_{s},v_{s},x_i+\varepsilon \omega,
v_{s+1}^*\right) d\omega dv_{s+1}
\end{aligned}
\end{equation}
\begin{equation}
\label{eq:AppA-coll-3}
\begin{aligned}
& \tilde{C}_{i,s+1}^- g_\varepsilon^{(s+1)}(t,Z_s)
 = \int_{\mathbb{R}^d} \int_{\mathbb{S}^{d-1}}
\left[ \omega \cdot (v_{s+1}-v_i)\right]_- \times \\
&\qquad\quad \times g_\varepsilon^{(s+1)} \left(t,x_1,v_1,\dots,x_i,
v_i,\dots,x_s,v_s,x_i+\varepsilon \omega,
v_{s+1}\right) d\omega dv_{s+1}
\end{aligned}
\end{equation}
and
\begin{equation}
\begin{aligned}
v_i^* & = v_i + \omega \omega \cdot (v_{s+1} - v_i) \\
v_{s+1}^* & = v_{s+1} - \omega \omega \cdot 
(v_{s+1}-v_i)
\end{aligned}
\end{equation}

The function $g_\varepsilon^{(s)} (t,Z_s)$ is defined for
$0\leq t < T$ and $Z_s \in \tilde{\mathcal{D}}_s$,
$s \geq m-1$, as the solution to the following hierarchy of equations:
\begin{equation}
\label{eq:AppA-UBEH}
\left( \partial_t + V_s \cdot \nabla_{X_s}\right)
g_\varepsilon^{(s)} (t) = 
\ell^{-1} \tilde{C}_{s+1} 
g_\varepsilon^{(s+1)} (t)
\qquad \textnormal{(if } s \geq m-1 \textnormal{)}
\end{equation}
with boundary condition
\begin{equation}
\label{eq:AppA-UBEH-bc}
g_\varepsilon^{(s)} (t,Z_s^*) = 
g_\varepsilon^{(s)} (t,Z_s)
\qquad \textnormal{ a.e. } (t,Z_s) \in [0,T)\times
\partial \tilde{\mathcal{D}}_s
\end{equation}
and initial conditions $g_\varepsilon^{(s)} (0,Z_s)$
defined for $s\geq m-1$ and $Z_s \in \tilde{\mathcal{D}}_s$.
We also introduce the function
$g_{\varepsilon} (t,x,v)$ ($t\geq 0$, $x,v \in \mathbb{R}^d$)
which is defined to be the solution to the equation
\begin{equation}
\label{eq:AppA-BE}
\left( \partial_t + v\cdot\nabla_x\right) g_{\varepsilon} (t) =
\ell^{-1} \tilde{C}_2 \left( g_{\varepsilon} (t) \otimes g_{\varepsilon} (t)
 \right)
\end{equation}
with prescribed initial data $g_{\varepsilon} (0)$.

We now introduce a mild form for
(\ref{eq:AppA-UBEH}-\ref{eq:AppA-UBEH-bc}).
For any $s\geq m-1$, let $\tilde{T}_s (t)$ denote the strongly
continuous semigroup on $L^2 \left(\tilde{\mathcal{D}}_s\right)$
with generator $-V_s \cdot \nabla_{X_s}$ and specular
reflection boundary conditions along $\partial \tilde{\mathcal{D}}_s$.
The operators $\tilde{T}_s(t)$ extend to other functional
spaces by standard density arguments. Then, under sufficiently strong
regularity conditions, the hierarchy 
(\ref{eq:AppA-UBEH}-\ref{eq:AppA-UBEH-bc})
is equivalent to the following hierarchy written in mild form:
\begin{equation}
\label{eq:AppA-UBEH-mild}
g_\varepsilon^{(s)} (t) =
\tilde{T}_s (t) g_\varepsilon^{(s)} (0) +
\ell^{-1} \int_0^t \tilde{T}_s (t-\tau)
\tilde{C}_{s+1} g_\varepsilon^{(s+1)} (\tau) d\tau \qquad
(s\geq m-1)
\end{equation}
Following Lanford's fixed point argument, we are able to prove existence
and uniqueness of solutions to (\ref{eq:AppA-UBEH-mild})
on a short time interval. However, under the conditions of Lanford's
proof, the distributional form (\ref{eq:AppA-UBEH}-\ref{eq:AppA-UBEH-bc})
and the mild form (\ref{eq:AppA-UBEH-mild}) are equivalent, so we are free to
work with either formulation for our computations. Note that, in a similar
fashion, we can define mild solutions for (\ref{eq:AppA-BE}), and solutions
can be constructed on a short time interval by a fixed point argument.

We will state a well-posedness theorem for
(\ref{eq:AppA-UBEH}-\ref{eq:AppA-UBEH-bc}) so that we can refer to the
result later. The proof follows Lanford's fixed point argument so we omit it.

\begin{proposition}
\label{prop:AppA-UBEH-Lanford}
Fix an integer $m\geq 2$.
Let $\left\{ g_\varepsilon^{(s)}(0) \right\}_{s\geq m-1}$
be a sequence of functions, with each
$g_\varepsilon^{(s)} (0)$ defined for 
$Z_s \in \tilde{\mathcal{D}}_s$. Furthermore,
suppose that there exists $\beta_0 > 0$ and $\mu_0 \in \mathbb{R}$
such that
\begin{equation}
\sup_{s\geq m-1} \sup_{Z_s \in 
\tilde{\mathcal{D}}_s} e^{\mu_0 s} 
e^{\beta_0 E_s (Z_s)} \left|
g_\varepsilon^{(s)} (0,Z_s) \right| \leq 1
\end{equation}
Then there exists a constant $C_d > 0$ such that the following is
true: If $T_L < C_d \ell e^{\mu_0} \beta_0^{\frac{d+1}{2}}$ then
there exists a unique sequence of functions 
$\left\{ g_\varepsilon^{(s)}(t) \right\}_{s \geq m-1}$
defined for $t\in [0,T_L]$ such that (i), (ii), and (iii) below all hold. \\
(i) For any bounded open set $\mathcal{O} \subset
[0,T_L] \times \tilde{\mathcal{D}}_s$,
we have $\left( \partial_t + V_s \cdot \nabla_{X_s}\right)
g_\varepsilon^{(s)} \in L^1 (\mathcal{O})$. \\
(ii) We have the bound:
\begin{equation}
\sup_{s\geq m-1} \sup_{t\in [0,T_L]}
\sup_{Z_s \in \tilde{\mathcal{D}}_s}
e^{(\mu_0 - 1)s} e^{\frac{1}{2} \beta_0 E_s (Z_s)} 
\left| g_\varepsilon^{(s)} (t,Z_s) \right| \leq 2
\end{equation}
(iii) The sequence $\left\{ g_\varepsilon^{(s)} (t)
\right\}_{s \geq m-1}$ solves (\ref{eq:AppA-UBEH}-\ref{eq:AppA-UBEH-bc})
in the sense of distributions on $[0,T_L]$ with initial data
$\left\{g_\varepsilon^{(s)}(0)\right\}_{s \geq m-1}$;
note that the equation is well-defined thanks to (i) and (ii).
\end{proposition}

We now turn to the main result of this section:

\begin{proposition}
\label{prop:AppA-factorize}
Fix an integer $m\geq 2$.
Let $\left\{ g_\varepsilon^{(s)}(t)\right\}_{s\geq m-1}$
be a sequence of functions, with each $g_\varepsilon^{(s)}
(t,Z_s)$ defined for $(t,Z_s) \in
[0,T)\times \tilde{\mathcal{D}}_s$.
Let $g_{\varepsilon}(t,x,v)$ be defined for $t\in [0,T)$ and
$x,v \in \mathbb{R}^d$.
Further suppose that there exists $\beta_T > 0$ and $\mu_T \in \mathbb{R}$
such that
\begin{equation}
\sup_{s \geq m-1} \sup_{t\in [0,T)}
\sup_{Z_s \in \tilde{\mathcal{D}}_s}
e^{\mu_T s}
e^{\beta_T E_s (Z_s)}
\left| g_\varepsilon^{(s)} (t,Z_s) \right|
\leq 1
\end{equation}
and
\begin{equation}
\sup_{t\in [0,T)} \sup_{x,v \in \mathbb{R}^d}
e^{\mu_T} 
e^{\frac{1}{2} \beta_T |v|^2} \left| g_{\varepsilon} (t,x,v) \right|
\leq 1
\end{equation}
and that $\left( \partial_t + V_s \cdot \nabla_{X_s}\right)
g_\varepsilon^{(s)}\in L^1 \left(\mathcal{O}\right)$ for any bounded
open set $\mathcal{O} \subset [0,T)\times \tilde{\mathcal{D}}_s$.
Then, if $\left\{ g_\varepsilon^{(s)}(t)\right\}_{s \geq m-1}$
solve (\ref{eq:AppA-UBEH}-\ref{eq:AppA-UBEH-bc}), $g_{\varepsilon}(t)$ solves
(\ref{eq:AppA-BE}), and
\begin{equation}
g_\varepsilon^{(s)} (0) = 
g_\varepsilon^{(m-1)} (0) \otimes
g_{\varepsilon}(0)^{\otimes (s-m+1)}
\end{equation}
for all $s\geq m$, then for $t\in [0,T)$ there holds
\begin{equation}
g_\varepsilon^{(s)} (t) =
 g_\varepsilon^{(m-1)} (t) \otimes
g_{\varepsilon}(t)^{\otimes (s-m+1)}
\end{equation}
for all $s \geq m$.
\end{proposition}
\begin{proof}
We proceed by constructing a solution of the unsymmetric
Boltzmann-Enskog hierarchy (\ref{eq:AppA-UBEH}-\ref{eq:AppA-UBEH-bc})
 with the desired property; then, the conclusion follows by uniqueness.
Let $T_L < C_d \ell e^{\mu_T} \beta_T^{\frac{d+1}{2}}$, where
$C_d$ is the constant appearing in 
Proposition \ref{prop:AppA-UBEH-Lanford}.

Recall that $g_{\varepsilon}(t)$ is  the solution to
(\ref{eq:AppA-BE}), with initial data $g_{\varepsilon}(0)$. Let us
now define $u_\varepsilon^{(m-1)} (t)$ to be the solution of
the following equation, for $0\leq t \leq T_L$:
\begin{equation}
\label{eq:AppA-u-epsilon}
\left( \partial_t + V_{m-1} \cdot \nabla_{X_{m-1}}\right) 
u_\varepsilon^{(m-1)} (t)
= \ell^{-1} \tilde{C}_m \left( 
u_\varepsilon^{(m-1)} (t) \otimes g_{\varepsilon}(t)\right)
\end{equation}
with boundary condition $u_\varepsilon^{(m-1)} 
(t,Z_{m-1}^*) = u_\varepsilon^{(m-1)} 
(t,Z_{m-1})$ along $[0,T_L]\times
\partial \tilde{\mathcal{D}}_{m-1}$, and initial
data $g_\varepsilon^{(m-1)} (0)$.
The existence and uniqueness for (\ref{eq:AppA-u-epsilon}) on a time
interval of size $T_L$ follows from a modified version of Lanford's
fixed point argument; moreover, the solution obeys the following bound:
\begin{equation}
\sup_{t\in [0,T_L]}
\sup_{Z_{m-1} \in \tilde{\mathcal{D}}_{m-1}}
e^{2(\mu_T-1)}
e^{\frac{1}{2} \beta_T E_{m-1} (Z_{m-1})} 
\left| u_\varepsilon^{(m-1)} (t,Z_{m-1}) \right|
\leq 2
\end{equation}
Having defined
$u_\varepsilon^{(m-1)} (t)$,
let us define, for $s \geq m$,
\begin{equation}
u_\varepsilon^{(s)} (t) = u_\varepsilon^{(m-1)} (t)
\otimes g_{\varepsilon}(t)^{\otimes (s-m+1)}
\end{equation}
Now it is straightforward to verify that the sequence
$\left\{ u_\varepsilon^{(s)}(t) \right\}_{s \geq m-1}$
satisfies (\ref{eq:AppA-UBEH}-\ref{eq:AppA-UBEH-bc}) for $t\in [0,T_L]$; by
uniqueness, we conclude that $g_\varepsilon^{(s)} (t) =
 u_\varepsilon^{(s)} (t)$
for all $s \geq m-1$ and $t\in [0,T_L]$.

We can iterate the same argument on the time intervals $[T_L,2T_L]$,
$[2T_L,3T_L]$, etc., until we have covered the full time interval
$[0,T)$.
\qed
\end{proof}

\subsection{Series Solution for the Unsymmetric Boltzmann-Enskog Hierarchy}

We will develop a series expansion and corresponding pseudo-trajectories
for the unsymmetric Boltzmann-Enskog hierarchy 
(\ref{eq:AppA-UBEH}-\ref{eq:AppA-UBEH-bc}). The main differences between the
unsymmetric Boltzmann-Enskog hierarchy and the BBGKY hierarchy are twofold:
first, the former is an infinite hierarchy, whereas the latter is finite;
and second, the former tracks correlations between $m-1$ particles,
whereas the latter tracks correlations between \emph{all} particles.
Since the two hierarchies are so similar, the developments in this section
will be almost identical to those of Section \ref{sec:8}. Nevertheless,
there are a few subtle differences which are important in our proof, so
in the interest of completeness we repeat the construction in this case.

The main point we wish to emphasize is that there is a new dynamics,
given by a measurable measure-preserving map
$\tilde{\psi}_s^t : \tilde{\mathcal{D}}_s \rightarrow
\tilde{\mathcal{D}}_s$, with the property that
\begin{equation}
\left(\tilde{T}_s (t) g^{(s)} \right) (Z_s) =
g^{(s)} \left( \tilde{\psi}_s^{-t} Z_s\right)
\end{equation}
where $g^{(s)}(Z_s)$ is an arbitrary measurable function with finite
integral, and $\tilde{T}_s$ is the transport operator appearing in
the mild form (\ref{eq:AppA-UBEH-mild}) of the unsymmetric Boltzmann-Enskog
hierarchy. The dynamics $\tilde{\psi}_s^t$ 
forces collisions between the first $m-1$ labeled particles, whereas any
pair of particles with
$(i,j)$ with $1\leq i \leq s$ and $m \leq j \leq s$ may pass through each
other without colliding.
We will need to use $\tilde{\psi}_s^t$ in place of $\psi_s^t$ in the
construction of pseudo-trajectories for the unsymmetric Boltzmann-Enskog
hierarchy.

Similar to the BBGKY hierarchy, we can write down an iterated Duhamel
series for the unsymmetric Boltzmann-Enskog hierarchy 
(\ref{eq:AppA-UBEH-mild}), like so: 
\begin{equation}
\begin{aligned}
& g_\varepsilon^{(s)} (t) = \sum_{k=0}^\infty
\ell^{-k} \times \\
& \qquad  \times  \int_0^t \int_0^{t_1} \dots \int_0^{t_{k-1}}
\tilde{T}_s (t-t_1) \tilde{C}_{s+1} \dots
\tilde{T}_{s+k} (t_k) g_\varepsilon^{(s+k)} (0) 
dt_k \dots dt_1 \\
& \qquad \qquad \qquad \qquad \qquad \qquad
\qquad \qquad \qquad \qquad \qquad \qquad
\textnormal{(if } s \geq m-1 \textnormal{)}
\end{aligned}
\end{equation}
Notice that the collision operators $\tilde{C}_{s+1}$ have replaced the
collision operators $C_{s+1}$ which appear in the Duhamel series for the
BBGKY hierarchy, and the transport operators $\tilde{T}_s$ have replaced
$T_s$. The collision operators $\tilde{C}_{s+1}$ \emph{do not} enforce
any exclusion condition, as can be seen from
(\ref{eq:AppA-coll-1}-\ref{eq:AppA-coll-3}); this fact will have to be
reflected in the construction of pseudo-trajectories.

Fix an integer $m\geq 2$ and let $s\geq m-1$.
We will be defining the symbols
\begin{equation}
\tilde{Z}_{s,s+k}
 \left[ Z_s,t;\left\{t_j,v_{s+j},\omega_j,i_j\right\}_{j=1}^k
\right]
\end{equation}
where $Z_s \in \overline{\tilde{\mathcal{D}}}_s$,
$0 \leq t_k < \dots < t_2 < t_1 < t$, 
$i_1 \in \left\{ 1,2,\dots,s\right\}$, $i_2 \in
\left\{ 1,2,\dots,s+1\right\}$, \dots, 
$i_k \in \left\{ 1,2,\dots,s+k-1\right\}$,
$v_{s+j} \in \mathbb{R}^d$, and $\omega_j \in \mathbb{S}^{d-1}$.
Given $Z_s \in \overline{\tilde{\mathcal{D}}}_s$ and $t>0$ we define
\begin{equation}
\tilde{Z}_{s,s} \left[ Z_s,t \right] = \tilde{\psi}_s^{-t} Z_s
\end{equation}
and $\tilde{Z}_{s,s} \left[ Z_s,0\right] = Z_s$. If the symbol
\begin{equation}
\tilde{Z}_{s,s+k} 
\left[ Z_s,t;\left\{ t_j,v_{s+j},\omega_j,i_j\right\}_{j=1}^k\right]
\in \overline{\tilde{\mathcal{D}}}_{s+k}
\end{equation}
is defined, then for all $\tau > 0$ we define
\begin{equation}
\begin{aligned}
& \tilde{Z}_{s,s+k} \left[ Z_s,t+\tau;
\left\{ t_j + \tau, v_{s+j},\omega_j,i_j \right\}_{j=1}^k \right]  =\\
& \qquad \qquad \qquad \qquad \qquad
 = \tilde{\psi}_{s+k}^{-\tau} \tilde{Z}_{s,s+k} \left[ Z_s,t;
\left\{ t_j,v_{s+k},\omega_j,i_j \right\}_{j=1}^k \right]
\end{aligned}
\end{equation}
Now suppose that the symbol
\begin{equation}
\tilde{Z}_{s,s+k} \left[ Z_s, t;
\left\{ t_j,v_{s+j},\omega_j,i_j\right\}_{j=1}^k\right] =
\left( X_{s+k}^\prime, V_{s+k}^\prime\right) \in 
\tilde{\mathcal{D}}_{s+k}
\end{equation}
is defined, $t_{k+1}=0$, $v_{s+k+1}\in\mathbb{R}^d$,
$\omega_{k+1}\in \mathbb{S}^{d-1}$, and $i_{k+1} \in
\left\{ 1,2,\dots,s+k\right\}$. Further suppose that
$\omega_{k+1} \cdot \left( v_{s+k+1}-v_{i_{k+1}}^\prime\right) \leq 0$. Then
we define
\begin{equation}
\tilde{Z}_{s,s+k+1}
\left[ Z_s,t;\left\{ t_j,v_{s+j},\omega_j,i_j \right\}_{j=1}^{k+1}
\right] =
\left( X_{s+k}^\prime,V_{s+k}^\prime,
x_{i_{k+1}}^\prime+\varepsilon \omega_{k+1},v_{s+k+1}\right)
\end{equation}
Similarly, suppose that the symbol
\begin{equation}
\tilde{Z}_{s,s+k}
\left[ Z_s,t;\left\{ t_j,v_{s+j},\omega_j,i_j \right\}_{j=1}^k
\right] = \left( X_{s+k}^\prime,V_{s+k}^\prime\right) \in
\tilde{\mathcal{D}}_{s+k}
\end{equation}
is defined, $t_{k+1} = 0$, $v_{s+k+1} \in \mathbb{R}^d$,
$\omega_{k+1} \in \mathbb{S}^{d-1}$, and
$i_{k+1} \in \left\{ 1,2,\dots,s+k\right\}$. Further suppose
that $\omega \cdot \left( v_{s+k+1} - v_{i_{k+1}}^\prime\right) > 0$.
Then we define
\begin{equation}
\begin{aligned}
& \tilde{Z}_{s,s+k+1}
\left[ Z_s,t; \left\{ t_j,v_{s+j},\omega_j,i_j\right\}_{j=1}^{k+1}
\right] = \\
& = \left( x_1^\prime,v_1^\prime,\dots,
x_{i_{k+1}}^\prime,
v_{i_{k+1}}^\prime+\omega_{k+1} \omega_{k+1}\cdot
\left( v_{s+k+1}-v_{i_{k+1}}^\prime\right), \dots, x_s^\prime,v_s^\prime,
 \right. \\
& \left. 
\qquad \qquad \qquad \qquad x_{i_{k+1}}^\prime +\varepsilon \omega_{k+1},
v_{s+k+1} - \omega_{k+1} \omega_{k+1} \cdot
\left( v_{s+k+1}-v_{i_{k+1}}^\prime \right)\right)
\end{aligned}
\end{equation}

Now we define the iterated collision kernel, again using induction.
If $Z_s \in \overline{\tilde{\mathcal{D}}}_s$ and $t\geq 0$ we define
\begin{equation}
\tilde{b}_{s,s} \left[ Z_s,t\right] = 1
\end{equation}
If the symbol
\begin{equation}
\tilde{b}_{s,s+k}
\left[ Z_s,t;\left\{ t_j,v_{s+j},\omega_j,i_j\right\}_{j=1}^k
\right]
\end{equation}
is defined and $\tau > 0$ then we define
\begin{equation}
\begin{aligned}
& \tilde{b}_{s,s+k} \left[ Z_s,t+\tau;\left\{t_j+\tau,
v_{s+j},\omega_j,i_j\right\}_{j=1}^k \right] = \\
& \qquad \qquad \qquad \qquad \qquad \qquad
= \tilde{b}_{s,s+k}
\left[ Z_s,t;\left\{ t_j,v_{s+j},\omega_j,i_j\right\}_{j=1}^k
\right]
\end{aligned}
\end{equation}
If the symbol
\begin{equation}
\tilde{b}_{s,s+k} \left[ Z_s,t;\left\{t_j,v_{s+j},\omega_j,i_j\right\}_{j=1}^k
\right]
\end{equation}
is defined, and $t_{k+1}=0$, $v_{s+k+1}\in\mathbb{R}^d$,
$\omega_{k+1} \in \mathbb{S}^{d-1}$, and
$i_{k+1} \in \left\{ 1,2,\dots,s+k\right\}$ then we define
\begin{equation}
\begin{aligned}
& \tilde{b}_{s,s+k+1} \left[ Z_s, t ;
\left\{ t_j,v_{s+j},\omega_j,i_j \right\}_{j=1}^{k+1}\right] =\\
&\qquad \qquad = \tilde{b}_{s,s+k} \left[ Z_s,t;
\left\{ t_j,v_{s+j},\omega_j,i_j\right\}_{j=1}^k \right] 
\times \omega_{k+1}\cdot \left( v_{s+k+1}-v_{i_{k+1}}^\prime\right) 
\end{aligned}
\end{equation}

We now have the following identity which holds pointwise for
$Z_s \in \overline{\tilde{\mathcal{D}}}_s$ and $t\geq 0$:
\begin{equation}
\begin{aligned}
& g_\varepsilon^{(s)} (t,Z_s) =  \sum_{k=0}^\infty \ell^{-k}\times \\
& \times \sum_{i_1=1}^s \dots \sum_{i_k=1}^{s+k-1} \int_0^t \dots
\int_0^{t_{k-1}} \int_{\mathbb{R}^{dk}}
\int_{\left(\mathbb{S}^{d-1}\right)^k} \left( \prod_{m=1}^k
d\omega_m dv_{s+m} dt_m \right) \times \\
& \times \left( \tilde{b}_{s,s+k}
 [\cdot] g_\varepsilon^{(s+k)} 
\left(0,\tilde{Z}_{s,s+k} [\cdot]\right)\right)
\left[ Z_s,t;\left\{t_j,v_{s+j},\omega_j,i_j\right\}_{j=1}^k\right]
\end{aligned}
\end{equation}

\subsection{Stability of pseudo-trajectories}

This subsection is concerned purely with the psuedo-trajectories
generated by the BBGKY hierarchy.
We are going to show that, if $Z_s$ is such that all but the
first two particles have free trajectories under the backwards
particle flow\footnote{including the possibilities that the first
two particles
collide \emph{or} ``pass through'' each other, or miss entirely},
then adding a particle preserves this property with
high probability. 
 This result is important because it allows us to
compare pseudo-trajectories for the BBGKY hierarchy with those
of the unsymmetric Boltzmann-Enskog hierarchy. Then it is 
straightforward to conclude that partial factorization is propagated by the
BBGKY hierarchy, because by Proposition \ref{prop:AppA-factorize},
partial factorization is propagated by the unsymmetric Boltzmann-Enskog
hierarchy. 

\begin{remark}
We fix $m=3$ in order to justify the
result $f_N^{(s)} (t) \approx f_N^{(2)} (t) \otimes f_t^{\otimes (s-2)}$
on the set
$\mathcal{G}_s \cap \hat{\mathcal{U}}_s^{\eta(\varepsilon)}$,
introduced in Section \ref{sec:2}, under the assumption that
the entire sequence $\left\{ F_N (0) \right\}_N$ is
2-nonuniformly $f_0$-chaotic.
\end{remark}

We will require the following sets:
\begin{equation}
\label{eq:AppA-G-s}
\mathcal{G}_s = \left\{ Z_s = (X_s,V_s) \in \overline{\mathcal{D}_s}
\left| 
\begin{aligned}
& \forall \tau > 0,\; \forall 3 \leq i \leq s,\\
& \qquad \qquad \qquad
\left( \psi_s^{-\tau} Z_s \right)_i = (x_i - v_i \tau, v_i) \\
& \textnormal{and, } \forall \tau > 0, \;
\forall 1 \leq i \leq 2, \; \forall 3 \leq j \leq s, \\
& \qquad \qquad \qquad
|(x_i - x_j) - (v_i - v_j)\tau| > \varepsilon
\end{aligned}
\right. \right\}
\end{equation} 
\begin{equation}
\mathcal{V}_s^\eta = \left\{
(Z_s,Z_s^\prime) \in \overline{\mathcal{D}_s} \times
\overline{\mathcal{D}_s} \left|
\begin{aligned}
& \inf_{1 \leq i \neq j \leq s} |v_i-v_j^\prime| > \eta  \\
& \qquad \qquad \textnormal{ and } \\
& \inf_{1 \leq i \leq s \; : \; |v_i - v_i^\prime|\neq 0}
|v_i - v_i^\prime|>\eta
\end{aligned}
\right. \right\}
\end{equation}
\begin{equation}
\label{eq:AppA-U-hat-s-eta}
\hat{\mathcal{U}}_s^\eta = \left\{ Z_s = (X_s,V_s) \in
\mathcal{U}_s^\eta \left| \forall \tau,\tau^\prime > 0,
(\psi_s^{-\tau} Z_s,\psi_s^{-\tau^\prime} Z_s) \in
\mathcal{V}_s^\eta  \right. \right\}
\end{equation}
Note carefully that
 $Z_s \in \mathcal{G}_{s}$ \emph{does not} guarantee
that the backwards trajectory $\left\{
\psi_s^{-t} Z_s \right\}_{t\geq 0}$ is free.

We are ready to state the main result for this section.
\begin{proposition}
\label{prop:AppA-bad-set}
There is a constant $c_d > 0$ such that all the following holds:
Assume that
\begin{equation}
\begin{aligned}
& Z_{s,s+k} \left[ Z_s,t; t_1,\dots,t_k;
v_{s+1},\dots,v_{s+k}; \omega_1,\dots,\omega_k;
i_1,\dots,i_k\right] = \\
&\qquad \qquad \qquad \qquad \qquad
 = (X_{s+k}^\prime,V_{s+k}^\prime) \in
\mathcal{G}_{s+k} \cap
\hat{\mathcal{U}}^\eta_{s+k}
\end{aligned}
\end{equation}
and $E_{s+k} (Z_{s+k}^\prime) \leq 2 R^2$ with $\eta < R$; then,\\
(i) for all $\tau \geq 0$ we have
\begin{equation}
\begin{aligned}
& Z_{s,s+k} \left[ Z_s,t+\tau; t_1+\tau,\dots,t_k+\tau;
v_{s+1},\dots,v_{s+k}; \omega_1,\dots,\omega_k;
i_1,\dots,i_k\right]  \\
& \qquad \qquad \qquad \qquad
\qquad \qquad \qquad \qquad \qquad \qquad \qquad
\in
\mathcal{G}_{s+k} \cap
\hat{\mathcal{U}}^\eta_{s+k}
\end{aligned}
\end{equation}
(ii) for any $i_{k+1} \in \left\{ 1,2,\dots,s+k\right\}$, and for
any $\alpha, y >  0$ and $\theta \in \left(0,\frac{\pi}{2}\right)$ such that
$\sin \theta > c_d y^{-1} \varepsilon$, there exists a measurable
set $\mathcal{B} \subset [0,\infty) \times \mathbb{R}^d \times
\mathbb{S}^{d-1}$, which may depend on 
$Z_s$, $t$, and $\left\{ t_j,v_{s+j},\omega_j,i_j\right\}_{j=1}^k$,
such that
\begin{equation}
\begin{aligned}
& \forall T > 0, \\
& \int_0^T \int_{B_{2R}^d} \int_{\mathbb{S}^{d-1}}
\mathbf{1}_{(\tau,v_{s+k+1},\omega_{k+1}) \in \mathcal{B}}
d\omega_{k+1} dv_{s+k+1} d\tau \leq \\
& \qquad \qquad \leq
C_{d,s,k} TR^d \left[ \alpha + \frac{y}{\eta T} + C_{d,\alpha}
\left( \frac{\eta}{R}\right)^{d-1} + C_{d,\alpha}
\theta^{(d-1)/2} \right]
\end{aligned}
\end{equation}
and
\begin{equation}
\begin{aligned}
& Z_{s,s+k+1} \left[ Z_s,t+\tau;t_1+\tau,\dots,t_k+\tau,0;
v_{s+1},\dots,v_{s+k},v_{s+k+1};\right.\\
& \left. \qquad\qquad \qquad \qquad \qquad \qquad \qquad
\omega_1,\dots,\omega_k,\omega_{k+1};i_1,\dots,i_k,i_{k+1}\right] \\
& \qquad \qquad \qquad \qquad \qquad \qquad \qquad \qquad \qquad \qquad
\in \mathcal{G}_{s+k+1} \cap \hat{\mathcal{U}}_{s+k+1}^\eta
\end{aligned}
\end{equation}
whenever $(\tau,v_{s+k+1},\omega_{k+1}) \in \left(
[0,\infty)\times \mathbb{R}^d \times \mathbb{S}^{d-1}\right)
\backslash \mathcal{B}$.
\end{proposition}
\begin{proof}
Claim \emph{(i)} is trivial so we turn to claim \emph{(ii)}.
The first step is to delete particle addition times for which particles
are too concentrated in space. However, it will not be enough to delete
times for which particles at a \emph{single} time-slice are nearby.
Instead, in order to eventually control \emph{all} possible recollisions
arising from the collisional dynamics, we gather together \emph{all}
line segments generated by the flow prior to time $\tau$ and project
them via free flight to land on the $\tau$ time-slice. Then we ask that
the entire set of phase points generated in this way does not
concentrate in space.
For $t^\prime \in \mathbb{R}$ and $Z_s^0 = (X_s^0,V_s^0) 
\in \mathbb{R}^{ds} \times
\mathbb{R}^{ds}$ we define
\begin{equation}
\hat{\psi}_s^{t^\prime} Z_s^0 = (X_s^0 + V_s^0 t^\prime, V_s^0)
\end{equation}
Also, we define $Z_{s+k}^\prime (\tau;t^\prime) =
\hat{\psi}_{s+k}^{t^\prime} \left( \psi_{s+k}^{-(\tau+t^\prime)}
Z_{s+k}^\prime\right)$. If
$Z_s^0 , Z_s^1 \in \mathbb{R}^{2ds}$ then we define
\begin{equation}
d_X(Z_s^0,Z_s^1) = \min\left(
\inf_{1\leq i \neq j \leq s}
|x_i^0 - x_j^1|,
\inf_{1\leq i \leq s \; : \;
(x_i^0,v_i^0) \neq (x_i^1,v_i^1) }
|x_i^0 - x_i^1| \right)
\end{equation}
\begin{equation}
\label{eq:AppA-B-I}
\mathcal{B}_I = 
\left\{
\begin{aligned}
& (\tau,v_{s+k+1},\omega_{k+1}) \in [0,\infty) \times 
\mathbb{R}^d \times \mathbb{S}^{d-1} \textnormal{ such that }
\tau = 0 \textnormal{ or } \\
& \qquad \qquad \qquad \qquad 
\exists t^\prime, t^{\prime \prime} \geq 0 \; : \;
d_X \left( Z_{s+k}^\prime (\tau; t^\prime),
Z_{s+k}^\prime (\tau;t^{\prime \prime}) \right) \leq y
\end{aligned}
\right\}
\end{equation}
We can easily estimate the measure of $\mathcal{B}_I$ due to the
condition $Z_{s+k}^\prime \in \mathcal{G}_{s+k}$; it suffices to
consider (at most) 
two possible line segments for each of the first two particles
(corresponding to whether the particles are allowed to collide,
or pass through each other, or miss entirely),
and for each $3 \leq i \leq s+k$, the \emph{unique} backwards
line segment
available to one of the $i$th particle.
 Distinct line segments are compared
pairwise to find collisions or near-collisions.
Since $Z_{s+k}^\prime \in
\hat{\mathcal{U}}_{s+k}^\eta$, any two line segments  can
only be within a distance $y$ (along a fixed time slice $\tau$)
 for a time $\Delta \tau$ of order
$y\eta^{-1}$ (this is where we explicitly use the integral in
the creation time $\tau$). We have
\begin{equation}
\begin{aligned}
& \int_0^T \int_{B_{2R}^d} \int_{\mathbb{S}^{d-1}}
\mathbf{1}_{(\tau,v_{s+k+1},\omega_{k+1}) \in \mathcal{B}_I}
d\omega_{k+1} dv_{s+k+1} d\tau \leq \\
& \qquad \qquad \qquad \qquad \qquad \qquad \qquad \qquad
\leq C_{d,s,k} R^d \eta^{-1} y
\end{aligned}
\end{equation}

At this point it is useful to distinguish between pre-collisional
and post-collisional configurations for the added particle.
Therefore we introduce two sets,
\begin{equation}
\mathcal{A}^+ = \left\{
\begin{aligned}
& (\tau,v_{s+k+1},\omega_{k+1}) \in
[0,\infty) \times \mathbb{R}^d \times \mathbb{S}^{d-1} 
\textnormal{ such that } \\
& \qquad \qquad \qquad
\omega_{k+1} \cdot \left(
v_{s+k+1} - v_{i_{k+1}}^\prime (\tau;0) \right) > 0
\end{aligned}
\right\}
\end{equation} 
\begin{equation}
\mathcal{A}^- = \left\{
\begin{aligned}
& (\tau,v_{s+k+1},\omega_{k+1}) \in
[0,\infty) \times \mathbb{R}^d \times \mathbb{S}^{d-1} 
\textnormal{ such that } \\
& \qquad \qquad \qquad
\omega_{k+1} \cdot \left(
v_{s+k+1} - v_{i_{k+1}}^\prime (\tau;0) \right) \leq 0
\end{aligned}
\right\}
\end{equation}
We also delete collisions which are close to grazing:
\begin{equation}
\mathcal{B}_{II} = \left\{
\begin{aligned}
& (\tau,v_{s+k+1},\omega_{k+1}) \in [0,\infty) \times
\mathbb{R}^d \times \mathbb{S}^{d-1} \textnormal{ such that } \\
& \left| \omega_{k+1} \cdot \left(v_{s+k+1} - v_{i_{k+1}}^\prime
(\tau;0) \right) \right| \leq (\sin \alpha)
\left| v_{s+k+1} - v_{i_{k+1}}^\prime ( \tau;0 ) \right|
\end{aligned}
\right\}
\end{equation} 
We have
\begin{equation}
\begin{aligned}
& \int_0^T \int_{B_{2R}^d} \int_{\mathbb{S}^{d-1}}
\mathbf{1}_{(\tau,v_{s+k+1},\omega_{k+1}) \in \mathcal{B}_{II}}
d\omega_{k+1} dv_{s+k+1} d\tau 
\leq C_d T R^d \alpha
\end{aligned}
\end{equation}
The pre-collisional configurations and post-collisional configurations
are dealt with separately.

\textbf{Pre-collisional configurations.}
We must guarantee that the $(s+k+1)$-particle state is in
$\hat{\mathcal{U}}_{s+k+1}^\eta$ at the time of particle creation.
Let us define
\begin{equation}
\mathcal{B}_{III}^- = \left\{
\begin{aligned}
& (\tau,v_{s+k+1},\omega_{k+1}) \in \mathcal{A}^- \textnormal{ such that }\\
& \exists \; i \in \left\{ 1,2,\dots,s+k\right\},\;
t^\prime \geq 0 \; : \; 
\left| v_{s+k+1} - v_i^\prime (\tau;t^\prime) \right| \leq \eta
\end{aligned}
\right\}
\end{equation} 
If $(\tau,v_{s+k+1},\omega_{k+1}) \notin \mathcal{B}_{III}^-$, and
the $(s+k+1)$ particle's backwards trajectory is free (which follows
from the next step), we can be sure that the $(s+k+1)$-particle state
is in $\hat{\mathcal{U}}_{s+k+1}^\eta$. We have
\begin{equation}
\begin{aligned}
& \int_0^T \int_{B_{2R}^d} \int_{\mathbb{S}^{d-1}}
\mathbf{1}_{(\tau,v_{s+k+1},\omega_{k+1}) \in \mathcal{B}_{III}^-}
d\omega_{k+1} dv_{s+k+1} d\tau 
\leq C_{d,s,k} T \eta^d
\end{aligned}
\end{equation}

Finally we need to make sure that the backwards trajectory of the
added particle is free (accounting for all possible histories
of particles $1$ and $2$). Let us define
\begin{equation}
\mathcal{B}_{IV}^- = \left\{
\begin{aligned}
& (\tau,v_{s+k+1},\omega_{k+1}) \in \mathcal{A}^- \textnormal{ such that }\\
& \exists i \in \left\{ 1,2,\dots,s+k\right\},\;
t^\prime \geq 0 \; : \; \\
& \frac{\left( x_{i_{k+1}}^\prime (\tau;0)+\varepsilon \omega
- x_i^\prime (\tau;t^\prime) \right)\cdot
\left(v_{s+k+1} - v_i^\prime (\tau; t^\prime)\right)}
{\left| x_{i_{k+1}}^\prime (\tau;0)+\varepsilon \omega -
x_i^\prime (\tau;t^\prime) \right| \left|
v_{s+k+1} - v_i^\prime (\tau;t^\prime) \right|} \geq \cos \theta
\end{aligned}
\right\}
\end{equation} 
We have
\begin{equation}
\begin{aligned}
& \int_0^T \int_{B_{2R}^d} \int_{\mathbb{S}^{d-1}}
\mathbf{1}_{(\tau,v_{s+k+1},\omega_{k+1}) \in \mathcal{B}_{IV}^-}
d\omega_{k+1} dv_{s+k+1} d\tau \leq \\
& \qquad \qquad \qquad \qquad \qquad \qquad \qquad \qquad
\leq C_{d,s,k} T R^d \theta^{d-1}
\end{aligned}
\end{equation}

To conclude, we let $\mathcal{B}^- = \mathcal{B}_I \cup
\mathcal{B}_{II} \cup \mathcal{B}_{III}^- \cup
\mathcal{B}_{IV}^-$; then we have
\begin{equation}
\begin{aligned}
& \int_0^T \int_{B_{2R}^d} \int_{\mathbb{S}^{d-1}}
\mathbf{1}_{(\tau,v_{s+k+1},\omega_{k+1}) \in \mathcal{B}^-}
d\omega_{k+1} dv_{s+k+1} d\tau \leq \\
& \qquad \qquad \qquad \qquad \qquad
\leq C_{d,s,k} T R^d \left[\alpha +
\frac{y}{\eta T} + 
\left(\frac{\eta}{R}\right)^d + \theta^{d-1}\right]
\end{aligned}
\end{equation}
Then again, by assumption, $\sin \theta > c_d y^{-1} \varepsilon$;
by choosing $c_d$ sufficiently large we may guarantee that
\begin{equation}
\begin{aligned}
& Z_{s,s+k+1} \left[ Z_s,t+\tau;t_1+\tau,\dots,t_k+\tau,0;
v_{s+1},\dots,v_{s+k},v_{s+k+1};\right.\\
& \left. \qquad\qquad \qquad \qquad \qquad \qquad \qquad
\omega_1,\dots,\omega_k,\omega_{k+1};i_1,\dots,i_k,i_{k+1}\right] \\
& \qquad \qquad \qquad \qquad \qquad \qquad \qquad \qquad \qquad \qquad
\in \mathcal{G}_{s+k+1} \cap \hat{\mathcal{U}}_{s+k+1}^\eta
\end{aligned}
\end{equation}
whenever $(\tau,v_{s+k+1},\omega_{k+1}) \in
\mathcal{A}^- \backslash \mathcal{B}^-$.

\textbf{Post-collisional configurations.} The post-collisional case is very
similar to the pre-collisional case; the only difference is that we must
account for the collsional change of variables. We define
\begin{equation}
\begin{aligned}
& v_{s+k+1}^* = v_{s+k+1} - \omega_{k+1} \omega_{k+1} \cdot
\left( v_{s+k+1} - v_{i_{k+1}}^\prime (\tau;0) \right) \\
& v_{i_{k+1}}^{\prime *} = v_{i_{k+1}}^\prime (\tau;0) +
\omega_{k+1} \omega_{k+1} \cdot \left(
v_{s+k+1} - v_{i_{k+1}}^\prime (\tau;0) \right)
\end{aligned}
\end{equation}

We remove particle creations with velocities being too close to some other
particle's velocity:
\begin{equation}
\mathcal{B}_{III}^+ = \left\{
\begin{aligned}
& (\tau,v_{s+k+1},\omega_{k+1}) \in \mathcal{A}^+ \backslash
\mathcal{B}_{II} \textnormal{ such that }\\
& \exists \; i \in \left\{ 1,2,\dots,s+k\right\},\;
t^\prime \geq 0 \; : \; 
\left| v_{s+k+1}^* - v_i^\prime (\tau;t^\prime) \right| \leq \eta
\end{aligned}
\right\}
\end{equation} 
\begin{equation}
\mathcal{B}_{IV}^+ = \left\{
\begin{aligned}
& (\tau,v_{s+k+1},\omega_{k+1}) \in \mathcal{A}^+ \backslash
\mathcal{B}_{II} \textnormal{ such that }\\
& \exists \; i \in \left\{ 1,2,\dots,s+k\right\}\backslash \left\{ i_{k+1} \right\},
\;
t^\prime \geq 0 \; : \; 
 \left| v_{i_{k+1}}^{\prime *} - v_i^\prime (\tau;t^\prime) \right| \leq \eta
\end{aligned}
\right\}
\end{equation} 
\begin{equation}
\mathcal{B}_{V}^+ = \left\{
(\tau,v_{s+k+1},\omega_{k+1}) \in \mathcal{A}^+ \textnormal{ such that }
 \left| v_{s+k+1} - v_{i_{k+1}}^\prime (\tau;0) \right| \leq \eta
\right\}
\end{equation} 
Using Lemma \ref{lemma:s9-sphere} we have
\begin{equation}
\begin{aligned}
& \int_0^T \int_{B_{2R}^d} \int_{\mathbb{S}^{d-1}}
\mathbf{1}_{(\tau,v_{s+k+1},\omega_{k+1}) \in \mathcal{B}_{III}^+}
d\omega_{k+1} dv_{s+k+1} d\tau 
\leq C_{d,s,k} C_{d,\alpha} T R \eta^{d-1}
\end{aligned}
\end{equation}
\begin{equation}
\begin{aligned}
& \int_0^T \int_{B_{2R}^d} \int_{\mathbb{S}^{d-1}}
\mathbf{1}_{(\tau,v_{s+k+1},\omega_{k+1}) \in \mathcal{B}_{IV}^+}
d\omega_{k+1} dv_{s+k+1} d\tau 
\leq C_{d,s,k} C_{d,\alpha} T R \eta^{d-1}
\end{aligned}
\end{equation}
\begin{equation}
\begin{aligned}
& \int_0^T \int_{B_{2R}^d} \int_{\mathbb{S}^{d-1}}
\mathbf{1}_{(\tau,v_{s+k+1},\omega_{k+1}) \in \mathcal{B}_{V}^+}
d\omega_{k+1} dv_{s+k+1} d\tau 
\leq C_{d} T \eta^{d}
\end{aligned}
\end{equation}

The last estimate will remove possible recollisions; define the sets
\begin{equation}
\mathcal{B}_{VI}^+ = \left\{
\begin{aligned}
& (\tau,v_{s+k+1},\omega_{k+1}) \in \mathcal{A}^+ \backslash
\mathcal{B}_{II} \textnormal{ such that }\\
& \exists i \in \left\{ 1,2,\dots,s+k\right\} \backslash
\left\{ i_{k+1} \right\},\;
t^\prime \geq 0 \; : \; \\
& \frac{\left( x_{i_{k+1}}^\prime (\tau;0)+\varepsilon \omega
- x_i^\prime (\tau;t^\prime) \right)\cdot
\left(v_{s+k+1}^* - v_i^\prime (\tau; t^\prime)\right)}
{\left| x_{i_{k+1}}^\prime (\tau;0)+\varepsilon \omega -
x_i^\prime (\tau;t^\prime) \right| \left|
v_{s+k+1}^* - v_i^\prime (\tau;t^\prime) \right|} \geq \cos \theta
\end{aligned}
\right\}
\end{equation}
\begin{equation}
\mathcal{B}_{VII}^+ = \left\{
\begin{aligned}
& (\tau,v_{s+k+1},\omega_{k+1}) \in \mathcal{A}^+ \backslash
\mathcal{B}_{II} \textnormal{ such that }\\
& \exists i \in \left\{ 1,2,\dots,s+k\right\} \backslash
\left\{ i_{k+1} \right\},\;
t^\prime \geq 0 \; : \; \\
& \frac{\left( x_{i_{k+1}}^\prime (\tau;0)+\varepsilon \omega
- x_i^\prime (\tau;t^\prime) \right)\cdot
\left(v_{i_{k+1}}^{\prime *} - v_i^\prime (\tau; t^\prime)\right)}
{\left| x_{i_{k+1}}^\prime (\tau;0)+\varepsilon \omega -
x_i^\prime (\tau;t^\prime) \right| \left|
v_{i_{k+1}}^{\prime *}
 - v_i^\prime (\tau;t^\prime) \right|} \geq \cos \theta
\end{aligned}
\right\}
\end{equation}
Using Lemma \ref{lemma:s9-sphere} and Lemma \ref{lemma:s9-cylinder},
we have
\begin{equation}
\begin{aligned}
& \int_0^T \int_{B_{2R}^d} \int_{\mathbb{S}^{d-1}}
\mathbf{1}_{(\tau,v_{s+k+1},\omega_{k+1}) \in \mathcal{B}_{VI}^+}
d\omega_{k+1} dv_{s+k+1} d\tau 
\leq C_{d,s,k} C_{d,\alpha} T R^d \theta^{(d-1)/2}
\end{aligned}
\end{equation}
\begin{equation}
\begin{aligned}
& \int_0^T \int_{B_{2R}^d} \int_{\mathbb{S}^{d-1}}
\mathbf{1}_{(\tau,v_{s+k+1},\omega_{k+1}) \in \mathcal{B}_{VII}^+}
d\omega_{k+1} dv_{s+k+1} d\tau 
\leq C_{d,s,k} C_{d,\alpha} T R^d \theta^{(d-1)/2}
\end{aligned}
\end{equation}

To conclude, we let
$\mathcal{B}^+ = \mathcal{B}_I \cup \mathcal{B}_{II} \cup
\mathcal{B}_{III}^+ \cup \mathcal{B}_{IV}^+ \cup
\mathcal{B}_V^+ \cup \mathcal{B}_{VI}^+ \cup \mathcal{B}_{VII}^+$; then we have
\begin{equation}
\begin{aligned}
& \int_0^T \int_{B_{2R}^d} \int_{\mathbb{S}^{d-1}}
\mathbf{1}_{(\tau,v_{s+k+1},\omega_{k+1}) \in \mathcal{B}^+}
d\omega_{k+1} dv_{s+k+1} d\tau \leq \\
& \qquad \qquad \qquad \qquad \qquad
\leq C_{d,s,k} T R^d \left[\alpha +
\frac{y}{\eta T} + 
C_{d,\alpha} \left( \frac{\eta}{R} \right)^{d-1}
+ C_{d,\alpha} \theta^{(d-1)/2} \right]
\end{aligned}
\end{equation}
Then again, by assumption, $\sin \theta > c_d y^{-1} \varepsilon$; by choosing
$c_d$ sufficiently large we may guarantee that
\begin{equation}
\begin{aligned}
& Z_{s,s+k+1} \left[ Z_s,t+\tau;t_1+\tau,\dots,t_k+\tau,0;
v_{s+1},\dots,v_{s+k},v_{s+k+1};\right.\\
& \left. \qquad\qquad \qquad \qquad \qquad \qquad \qquad
\omega_1,\dots,\omega_k,\omega_{k+1};i_1,\dots,i_k,i_{k+1}\right] \\
& \qquad \qquad \qquad \qquad \qquad \qquad \qquad \qquad \qquad \qquad
\in \mathcal{G}_{s+k+1} \cap \hat{\mathcal{U}}_{s+k+1}^\eta
\end{aligned}
\end{equation}
whenever $(\tau,v_{s+k+1},\omega_{k+1}) \in
\mathcal{A}^+ \backslash \mathcal{B}^+$.
\qed
\end{proof}

\section{Acknowledgements}
\label{sec:acknowledgements}
This paper is an expanded version of 
 the author's dissertation at New York University.
I would like to thank my PhD advisor, Nader Masmoudi,
for valuable guidance and immeasurable patience. I would also like to
thank Laure Saint-Raymond, for reading multiple drafts of the paper and
providing feedback which has improved the presentation significantly.
I would also like to thank Cl{\'e}ment Mouhot,
Pierre Germain, and Pierre-Emmanuel Jabin, for delightful discussions and
many insights into the mathematical universe. Finally, I would like to
thank the anonymous reviewers for 
pointing out a number of necessary corrections and clarifications,
which have improved the overall quality of the manuscript.

\bibliography{bbgky-preprint-3} 

\end{document}